\documentclass[preprint,3p,times]{elsarticle}
\usepackage{framed,multirow}

\usepackage{amsmath}
\usepackage{amsthm}
\usepackage{amsfonts}
\usepackage{amssymb}
\usepackage{graphicx}
\usepackage{paralist}
\usepackage{listings}
\usepackage{csquotes}
\usepackage{dirtree}
\usepackage{empheq}
\usepackage{mathrsfs}
\usepackage{physics}
\usepackage{mathtools}
\usepackage{algorithmicx}
\usepackage{bm}
\usepackage{float}
\usepackage{hyperref}
\usepackage{xcolor}

\newcommand{\hlo}[1]{#1}
\newcommand{\hlb}[1]{#1}

\newtheorem{proposition}{Proposition}[section]
\newtheorem{lemma}{Lemma}[section]
\newtheorem{remark}{Remark}[section]
\newcommand{\aone}{\frac{1}{\alpha}}
\newcommand{\atwo}{\frac{1}{1-\alpha}}

\newcommand{\mom}{\rho u}

\newcommand{\Etot}{\rho E}
\newcommand{\RealSet}{\mathbb{R}}
\newcommand{\half}{\frac{1}{2}}
\newcommand{\Tau}{\mathcal{T}}
\newcommand{\EOS}{\text{EOS}}

\newcommand{\dt}{\partial_t}
\newcommand{\dx}{\partial_x}
\newcommand{\dy}{\partial_y}

\newcommand{\Dt}{\Delta t}
\newcommand{\Dx}{\Delta x}

\newcommand{\admissibleset}{\Omega}

\newcommand{\sourceterm}{\bm{S}}
\newcommand{\acousticop}{\bm{P}}

\newcommand{\bW}{\bm{W}}
\newcommand{\bU}{\bm{U}}
\newcommand{\bF}{\bm{F}}
\newcommand{\bG}{\bm{G}}
\newcommand{\br}{\bm{r}}
\newcommand{\relaxparam}{\nu}
\newcommand{\diagomatrix}{\bm{M}}
\newcommand{\appRP}{\bU_{\text{RP}}}
\newcommand{\TauRP}{\Tau_{\text{RP}}}
\newcommand{\PiRP}{\Pi_{\text{RP}}}
\newcommand{\phiRP}{\phi_{\text{RP}}}
\newcommand{\Pijump}{\mathcal{M}}
\newcommand{\mach}{\operatorname{Ma}}
\newcommand{\froude}{\operatorname{Fr}}
\newcommand{\entropy}{\eta}
\newcommand{\specificentropy}{\mathcal{U}}

\newcommand{\eg}[0]{\textit{e.g.}\,}

\begin{document}

	\begin{frontmatter}

		\title{Recasting an operator splitting solver into a standard finite volume flux-based algorithm. The case of a Lagrange-Projection-type method for gas dynamics}

		\author[1]{R\'{e}mi {Bourgeois}\corref{cor1}}
        \ead{remi.bourgeois@cea.fr}
        \author[1]{Pascal {Tremblin}}
         \author[2]{Samuel {Kokh}}
        \author[1]{Thomas {Padioleau}}
        \cortext[cor1]{Corresponding author}

        \address[1]{Université Paris-Saclay, UVSQ, CNRS, CEA, Maison de la Simulation, 91191, Gif-sur-Yvette, France}
\address[2]{Université Paris-Saclay, CEA, Service de Génie Logiciel pour la Simulation, 91191, Gif-sur-Yvette, France.}

\begin{abstract}
	In this paper, we propose a modification of an
	acoustic-transport operator splitting Lagrange-projection method for simulating compressible flows with gravity. The original method involves two steps that respectively account for acoustic and transport effects. Our work proposes a simple modification of the transport step, and the resulting modified scheme turns out to be a flux-splitting method.
	This new numerical method is less computationally expensive \hlo{in the low-Mach regime}, more memory efficient, and easier to implement than the original one. 
	We prove stability properties for this new scheme by showing that under classical CFL conditions, the method is positivity preserving for mass, energy and entropy satisfying.
	The flexible flux-splitting structure of the method enables straightforward extensions of the method to multi-dimensional problems (with respect to space)
	and high-order discretizations that are presented in this work.
	We also propose an interpretation of the flux-splitting solver as a relaxation approximation. Both the stability and the accuracy of the new method are tested against one-dimensional and two-dimensional numerical experiments that involve highly compressible flows and low-Mach regimes.
\end{abstract}

	\end{frontmatter}
\newpage
	\noindent

	%
	%
	\newpage
\newpage

\section{Introduction}
In this work, we consider the approximation of the compressible Euler equations in the presence of source
 terms derived from a smooth potential using a finite volume method.
This paper aims to showcase the recasting of an Operator Splitting Lagrange-Projection (OSLP) finite volume 
algorithm into a corresponding flux-splitting method (FSLP). The flux-splitting method we consider here has several computational and implementation advantages compared to OSLP. It requires a smaller stencil, no intermediate state storage, and can be implemented as a fully explicit flux-based solver. The simplicity of the FSLP method allows us to combine effortlessly with standard means to derive higher-order methods such as MUSCL, ENO, WENO, and MOOD frameworks.

The OSLP algorithm we use as ground material for implementing an FSLP method
is presented in \cite{padioleau2019high}. It relies on a separate treatment of acoustic and transport effects,
 and it enjoys several interesting properties: it is stable under a CFL condition so that it ensures positivity for mass and internal energy and satisfies a discrete entropy inequality.
The treatment of the source term in \cite{padioleau2019high} allows us to
 preserve stationary solution profiles at the discrete level so that the OSLP scheme satisfies a well-balanced property (see \eg \cite{Gosse1996,Greenberg1996, LeVeque1998a, gosse_well-balanced_2000,Gosse2004,Audusse2004,
 	LukacovaMedvidova2007,Noelle2007,Diaz2007,Pelanti2008,
 	 gosse_computing_2013,Kaeppeli2014, Chandrashekar2015,Desveaux2016,Chalons2016,
 	 Touma2016,
 	 michel-dansac_well-balanced_2016,Castro2017,Chertock2018,padioleau2019high,Castro2020,
 	Luna2020,Berberich2021,Grosso2021}).
Moreover, when the Mach number that characterizes the ratio of the material velocity to the sound velocity is 
low, cell-centered finite volume methods may suffer an important loss of 
accuracy \cite{Turkel1987,Guillard1999,Guillard2004,Dellacherie2010}. This
question is connected to several delicate issues like the
influence of the mesh geometry \cite{Rieper2009,Dellacherie2010a}, the numerical diffusion 
(see for example \cite{Dauvergne2008, Dellacherie2010, chalons_all-regime_2016, Dellacherie2016, Zakerzadeh2016, Barsukow2021})
or the Asymptotic Preserving property with respect to incompressible models
 \cite{Degond2011, Cordier2012, Zakerzadeh2016,Bispen2017, Berthon2020,Dimarco2017,Boscarino2018,Bouchut2020a} and has been extensively
investigated in the literature for the past years through several
approaches (see also \cite{Paillere2000,Guillard2004,Beccantini2008,Dimarco2018,Boscheri2020,Bouchut2020,Zeifang2020,Bruel2019}).
Although it does not address the full spectrum of problems connected to the simulation of 
flows in the low Mach regime, a simple modification of the OSLP method ensures a uniform truncation error with respect to the Mach number\cite{chalons_all-regime_2016,padioleau2019high}.
The resulting FSLP algorithm we obtain performs equally concerning these  aspects. Moreover, 
it profits from all the advantages of FSLP methods over OSLP mentioned above. It is also less
 computationally expensive \hlo{in the low-Mach regime}, requiring fewer sweeps over the numerical solution to reach the same physical time. The derivation of the stability properties of the FSLP method requires novel mathematical developments that we present in this study.

The paper is organized as follows: we first introduce the set of equations with the 
thermodynamical-related hypotheses that support the stability properties of the model,
 and we present the stationary profiles and difficult regimes we will be interested in.
  Then, we will recall the OSLP method we aim to recast into its FSLP version. We will modify 
  the transport step in the original OSLP method so that both steps are revamped into one that
   can be viewed as a flux-splitting step. We will then provide proof of stability for the FSLP 
   method we obtained. We examine standard ways to extend the FSLP method to higher-order discretizations
    and multi-dimensional problems. Then we will see that the FSLP method can be connected to a new relaxation 
	approximation of the Euler equations that proposes a single-step but separate treatment of the acoustic and
	 transport effects.
Finally, we will present one-dimensional and two-dimensional numerical experiments that
 demonstrate the good behavior of the scheme.

\section{Flow model}

For the sake of clarity but without loss of generality,
we focus on one-dimensional problems. We consider the Euler equations supplemented with a smooth
 potential source term $x\mapsto \phi(x)$,

\begin{equation}
		\dt \bU + \dx \bF ( \bU ) = \sourceterm (\bU,\phi)
		,\qquad\text{for $x\in\RealSet$, $t>0$,}
		\label{Euler's eqn}
\end{equation}

with
\(\bU=(\rho, \rho u, \rho E)^T\),
\(\bF ( \bU ) = (\mom,\ u\mom+ p,\ u\rho E + p u)^T\)
and
\(\sourceterm (\bU,\phi ) = -\rho \dx \phi (0,1,u)^T\) where $\phi$ is smooth enough so that we can
 consider that $\partial_x \phi$ is also regular and bounded.

 Although \eqref{Euler's eqn} is not strictly limited to flows accounting for gravitational forces, 
 the stationary potential $x\mapsto \phi(x)$ will be referred to as the gravitational potential.
The fields $\rho$, $u$, $p$, and $E$ respectively denote the density, velocity, pressure, and specific
 total energy of the fluid. If $e=E-u^2/2$ is the specific internal energy, we define the set of admissible 
 states
\begin{equation}
\admissibleset =
\left\{
(\rho, \mom, \Etot)\in \mathbb{R}^3~\big|~\rho > 0,\ e>0\right
\}
.
\label{Def omega}
\end{equation}

Let $s$ be the specific entropy of the fluid. We consider an Equation Of State (EOS) in the form of a 
mapping $(1 / \rho,s)\mapsto e^{\EOS}(1 / \rho,s)$
that satisfies the classic Weyl assumptions~\cite{Weyl1949,chalons_all-regime_2016}:

\begin{subequations}
	\begin{align}
		\pdv{e^\EOS}{(1 / \rho)}
		&<0
		, &
		\pdv{e^\EOS}{s}
		&>0
		, &
		\pdv[2]{e^\EOS}{(1 / \rho)}
		&>0
		,
		\\
		\pdv[2]{e^\EOS}{s}
		&>0
		, &
		\qty[\pdv[2]{e^\EOS}{(1 / \rho)}]
		\qty[\pdv[2]{e^\EOS}{s}]
		&>
		\qty
		[
		\pdv{e^\EOS}{s}{(1 / \rho)}
		]^{2}
		, &
		\pdv[3]{e^\EOS}{(1 / \rho)}
		&<0
		.
	\end{align}
	\label{weyl}
\end{subequations}
The temperature $T$ and the pressure $p$ of the fluids are related to the other parameters, respectively by
$T=T^\EOS(1 / \rho,s)=\pdv*{e^\EOS}{s}$
and
$p=p^\EOS(1 / \rho,s) = -\pdv*{e^\EOS}{(1 / \rho)}$.
It is possible to define a mapping $(1 / \rho,e)\mapsto s^\EOS(1 / \rho,e)$ such that
$e=e^\EOS(1 / \rho,s)$ if $s=s^\EOS(1 / \rho,e)$ so that we have the Gibbs relation
\begin{equation}
	\dd e + p\dd (1/\rho)  = T\dd s.
	\label{eq: Gibbs}
\end{equation}
 Note that \eqref{weyl} imply that $-s^{\EOS}(1/\rho, e)$ and $e^{\EOS}(1/\rho, s)$ are strictly convex functions.
 Relations~\eqref{weyl} also ensure that
\begin{equation}
	{\pdv{p^\EOS}{(1 / \rho)}}(1/\rho,s)<0
,
\end{equation}
 so that the sound velocity
 $c = \rho^{-1}\sqrt{-\pdv*{p^\EOS(1 / \rho,s)}{(1 / \rho)}(1/\rho,s)}$
is real valued.
Let us recall now that the dimensionless quantity $\mach = |u|/c $ is called the Mach number.
We also make the classic assumption \cite{Callen1985} that
\begin{equation}
	\mathscr{M} s(\mathscr{V}/\mathscr{M},\mathscr{E}/\mathscr{M})
	=
	S(\mathscr{M}, \mathscr{V},\mathscr{E})
	,
\label{eq: specific entropy to entropy}
\end{equation}
where the (non-specific) entropy
\((\mathscr{M}, \mathscr{V},\mathscr{E})\mapsto S(\mathscr{M}, \mathscr{V},\mathscr{E})\)
is a strictly concave homogeneous first-order function. Let us note that as
$
{\pdv{S}{\mathscr{E}}} (\mathscr{M}, \mathscr{V},\mathscr{E})
= {\pdv{s}{e}}(\mathscr{V}/\mathscr{M},\mathscr{E}/\mathscr{M})
= 1/T^\EOS(\mathscr{V}/\mathscr{M},\mathscr{E}/\mathscr{M})
>0
$, then \(\mathscr{E}\mapsto S(\bar{\mathscr{M}}, \bar{\mathscr{V}},\mathscr{E})\) is a strictly increasing 
function for a fixed $\mathscr{M}$ and $\mathscr{V}$.

 Weak solutions of \eqref{Euler's eqn} also satisfy the entropy inequality
\begin{equation}
	\dt (\rho s) +  \dx  (u\rho s)  \geq 0,
	\label{inégalité entropie}
\end{equation}
 where the inequality~\eqref{inégalité entropie} is indeed
 an equality in the case of smooth solutions (see \cite{smoller,LeVeque2002,Godlewski-1990,serre}).

We also are interested in the study of particular steady-state solutions of \eqref{Euler's eqn} called the 
hydrostatic equilibria that are classically defined by
\begin{align}
	\dx p &= - \rho \dx \phi
	,&
	\ u &= 0.
	\label{hydrostatic continu}
\end{align}
For many years, significant efforts have been dedicated to developing so-called well-balanced numerical methods (see \eg \cite{Gosse1996,Greenberg1996, LeVeque1998a, gosse_well-balanced_2000,Gosse2004,Audusse2004,
	LukacovaMedvidova2007,Noelle2007,Diaz2007,Pelanti2008,
	gosse_computing_2013,Kaeppeli2014, Chandrashekar2015,Desveaux2016,Chalons2016,
	Touma2016,
	michel-dansac_well-balanced_2016,Castro2017,Chertock2018,padioleau2019high,Castro2020,
	Luna2020,Berberich2021,Grosso2021}) that allow preserving discrete equivalents of equilibrium solutions like \eqref{hydrostatic continu}.
In the present work, we intend to investigate well-balanced finite volume approximations of
\eqref{Euler's eqn} that are compatible with discrete equivalents of \eqref{inégalité entropie} and 
ensure that the fluid states $(\rho, \rho u, \rho E)$ remain in $\admissibleset$.

Before going any further, let us introduce the notations for our space-time discretization:
we consider a strictly increasing sequence $(x_{j+1/2})_{j\in\mathbb{Z}}$ and divide the
real line into cells where the $j^\text{th}$ cell is the interval
$\left(x_{j-1 / 2}, x_{j+1 / 2}\right)$. The space step of $j^\text{th}$ cell is
$\Dx_j = x_{j+1/2} - x_{j-1/2}>0$ that we suppose constant and equal to $\Delta x$ for the sake of 
simplicity. We note $\Delta t>0$ the time step such that $t^{n+1}-t^n= \Dt$ with $n \in \mathbb{N}$. 
For a given initial condition $x \mapsto \bU^{0}(x)$, we consider a discrete initial data $\bU_{j}^{0}$ 
defined by $\bU_{j}^{0}=\frac{1}{\Dx} \int_{x_{j-1 / 2}}^{x_{j+1 / 2}} \bU^{0}(x) \mathrm{d} x$, for $j \in \mathbb{Z}$.
The algorithm proposed in this paper aims at computing a first-order accurate (in both space and time) approximation of the cell-averaged values $\bU_{j}^{n}$ of $\frac{1}{\Dx} \int_{x_{j-1 / 2}}^{x_{j+1 / 2}} \bU\left(x, t^{n}\right) \dd x$ where $x \mapsto\bU\left(x, t^{n}\right)$ is the exact solution of \eqref{Euler's eqn} at time $t^{n}$ by means of a conservative finite volume discretization of \eqref{Euler's eqn} of the form
\begin{equation}
\bU_j^{n+1} - \bU_j^{n}
+
\frac{\Delta t}{\Dx}
\qty(
\bF_{j+1/2}
-
\bF_{j-1/2}
)
=
\Delta t
\sourceterm_j.
\label{eq: flux-splitting discretization}
\end{equation}

\section{The original Operator Splitting Lagrange-Projection (OSLP) strategy}\label{sect:OSLP}
Operator splitting strategies allow simpler derivation of numerical methods by solving parts of the system separately and successively. However, this requires storing intermediate state values and may also necessitate specific treatments to implement higher order extension (see, for example \cite{Pino2006,Duboc2010,Luna2020,Grosso2021}).

In this section, we recall the properties of the OSLP method presented in \cite{padioleau2019high}. It combines the all-regime method for gas dynamics proposed by \cite{chalons_all-regime_2016} and the well-balanced treatment of source terms introduced in \cite{Chalons2016} in the context of the shallow water system.
We chose to re-introduce all the discretization as the goal of the present paper is to recast this particular OSLP algorithm into a flux-splitting Lagrange-Projection (FSLP) finite volume method, using very similar expressions. We emphasize that the algorithm presented in this section is not new and comes entirely from \cite{chalons_all-regime_2016,Chalons2016,padioleau2019high} and that
the novelty of our work lies in a modification of this algorithm that will be detailed in section~\ref{recast}.
The method is based on the splitting of \eqref{Euler's eqn} into an acoustic sub-system:
\begin{subequations}
\begin{empheq}[left=\empheqlbrace]{align}
\partial_t \rho+\rho \partial_x u &=0,
\\
\partial_t(\rho u)+\rho u \partial_x u+\partial_x p &=-\rho \dx \phi,
\\
\partial_t(\rho E)+\rho E \partial_x u+\partial_x(p u) &=-\rho u\dx \phi,
\end{empheq}
\label{acoustic}
\end{subequations}
and a transport sub-system:
\begin{subequations}
\begin{empheq}[left=\empheqlbrace]{align}
\partial_t \rho+u \partial_x \rho &=0,
\\
\partial_t(\rho u)+u \partial_x(\rho u) &=0,
\\
\partial_t(\rho E)+u \partial_x(\rho E) &=0
.
\end{empheq}
\label{transport}
\end{subequations}

Given a fluid state $U^n$, this operator splitting  algorithm can be decomposed as follows.
\begin{enumerate}
	\item Update the fluid state $U^n$ to the value $U^{n+1-}$ by approximating the solution of \eqref{acoustic}:
\begin{subequations}
\begin{empheq}[left=\empheqlbrace]{align}
L_j \rho_j^{n+1-} &=\rho_j^n
,\\
L_j(\rho u)_j^{n+1-} &=(\rho u)_j^n-\frac{\Delta t}{\Delta x}\left(\Pi_{j+1 / 2}^{*,\theta}-\Pi_{j-1 / 2}^{*,\theta}\right) -\Dt \{\rho \dx\phi\}_j^n
,\\
L_j(\rho E)_j^{n+1-} &=(\rho E)_j^n-\frac{\Delta t}{\Delta x}\left(\Pi_{j+1 / 2}^{*,\theta} u_{j+1 / 2}^*-\Pi_{j-1 / 2}^{*,\theta} u_{j-1 / 2}^*\right) -\Dt \{\rho u\dx\phi\}_j^n
,\\
L_j &=1+\frac{\Delta t}{\Delta x}\left(u_{j+1 / 2}^*-u_{j-1 / 2}^*\right)
.
\end{empheq}
\label{acoustic step}
\end{subequations}
	\item Update the fluid state  $U^{n+1-}$ to the value $U^{n+1}$ by approximating the solution of \eqref{transport}: for $\varphi \in\{\rho, \rho u, \rho E\}$
	\begin{equation}
\varphi_j^{n+1}=\varphi_j^{n+1-} L_j\hlo{-}\frac{\Delta t}{\Delta x}\left(u_{j+1 / 2}^* \varphi_{j+1 / 2}^{n+1-}-u_{j-1 / 2}^* \varphi_{j-1 / 2}^{n+1-}\right)
\label{transport step split}
\end{equation}
\end{enumerate}
with the upwind choice
\begin{equation}
\varphi_{j+1 / 2}^{n+1-}
=
\begin{cases}
\varphi_j^{n+1-}, \quad\text{if  \(u_{j+1 / 2}^* \geq 0\),}
\\
\varphi_{j+1}^{n+1-}, \quad\text{if  \(u_{j+1 / 2}^*<0\),}
\end{cases}
\end{equation}
and the following formulas for the interface pressures and velocities
\begin{subequations}
\begin{empheq}[left=\empheqlbrace]{align}
u^*_{j+1/2}
&=	
\frac{\left(u_{j+1}^n+u_{j}^n\right)}{2}-\frac{1}{2 a_{j+1/2}}\left(p_{j+1}^n-p_j^n + \frac{\rho_{j+1}^n + \rho_j^n}{2} (\phi_{j+1}^n - \phi_j^n)\right)
,\\
\Pi^{*,\theta}_{j+1/2}
&=
\frac{\left(p_{j+1}^n+p_{j}^n\right)}{2}-\theta_{j+1/2}\frac{a_{j+1/2}}{2}\left(u_{j+1}^n-u_{j}^n\right),
\end{empheq}
\label{u*p*}
\end{subequations}
as well as the source terms discretization:

\begin{subequations}
	\begin{empheq}[left=\empheqlbrace]{align}
\{\rho \dx\phi\}_j^n
&=
\frac{\{\rho \dx\phi\}_{j+1/2}+\{\rho \dx\phi\}_{j-1/2}}{2}
,\\
\{\rho u \dx\phi\}_j^n
&=
\frac{u^*_{j+1/2}\{\rho \dx\phi\}_{j+1/2}+u^*_{j-1/2}\{\rho \dx\phi\}_{j-1/2}}{2}
,\\
\{\rho \dx\phi\}_{j+1/2}
&=
\frac{\rho_{j+1}^n+\rho_j^n}{2}\frac{\phi_{j+1}-\phi_{j}}{\Dx}
.
\end{empheq}
\end{subequations}
The constant parameter $a_{j+1/2}$ is a local choice of an approximate acoustic impedance $a$ associated with each interface $j+1/2$. It should be chosen large enough so that \eqref{acoustic CFL} is satisfied, guaranteeing stability for the acoustic step. In practice, we choose
\begin{equation}
	a_{j+1 / 2}=K \max \left(\rho_{j}^{n} c_{j}^{n}, \rho_{j+1}^{n} c_{j+1}^{n}\right)\qquad \text{ with \(K > 1\)}.
	\label{def a}
\end{equation}
In the tests of section~\ref{numerical exp} we will use $K=1.1$.

The parameter $\theta$ enables the implementation of a low-Mach flux correction that ensures a control of the numerical diffusion
in the momentum equation.
This simple strategy is modeled after \cite{Dauvergne2008, Dellacherie2010,Dellacherie2016}.
Depending on the choice of $\theta$, this correction takes effect whenever $\mach < 1$. In our case, 
its sole purpose is to help preserve the accuracy
in the low-Mach regions of the computational domain by providing a uniform control of the truncation error with respect to $\mach$.
We need to emphasize that this approach does not aim at addressing the full complexity of simulating flows in the low-Mach regime that has been widely investigated in the literature and spans for example: from the study of the influence of the grid \cite{Rieper2009,Dellacherie2010a}, the potential development of spurious modes\cite{Dellacherie2009,Jung2022}, the development of asymptotic preserving methods
 \cite{Degond2011, Cordier2012, Zakerzadeh2016,Bispen2017, Berthon2020,Dimarco2017,Boscarino2018,Bouchut2020a}, implicit-explicit methods
\cite{chalons_all-regime_2016,Dimarco2018,Boscheri2020,Bouchut2020,Zeifang2020} multi-dimensional control of the numerical diffusion \cite{Barsukow2021}, use of preconditioning methods \cite{Turkel1987,Guillard1999,Paillere2000,Guillard2004,Beccantini2008} to the study of acoustics in low-Mach regime \cite{Bruel2019}.

The discretization of the gravitational source term allows to exactly preserve the following discrete equivalent of the hydrostatic equilibrium \eqref{hydrostatic continu}:
\begin{align}
	\Pi^n_{j+1} - \Pi^n_j
	&=
	-\frac{\rho^n_{j+1} + \rho^n_j}{2} (\phi_{j+1} - \phi_j)
	,&
	u^n_j &= 0
	,&\forall j &\in \mathbb{Z}, \forall n\in\mathbb{N}.
	\label{eq profile 1st order}
\end{align}
Note that the resolution of the acoustic system is performed via a Suliciu-type relaxation \cite{Suliciu1998,Bouchut2004,chalons2008,coquel2012} following 
\cite{chalons_all-regime_2016,Chalons2016}. Both steps can be rewritten as a fully conservative update formula:
\begin{subequations}
\begin{empheq}[left=\empheqlbrace]{align}
\rho_j^{n+1} &=\rho_j^n\hlo{-}\frac{\Delta t}{\Delta x}\left(u_{j+1 / 2}^* \rho_{j+1 / 2}^{n+1-}-u_{j-1 / 2}^* \rho_{j-1 / 2}^{n+1-}\right)
,\\
(\rho u)_j^{n+1} &=(\rho u)_j^n\hlo{-}\frac{\Delta t}{\Delta x}\left(u_{j+1 / 2}^*(\rho u)_{j+1 / 2}^{n+1-}+\Pi_{j+1 / 2}^{\theta,*}-u_{j-1 / 2}^*(\rho u)_{j-1 / 2}^{n+1-}-\Pi_{j-1 / 2}^{\theta,*}\right)-\Dt \{\rho \dx\phi\}_j^n
,\\
(\rho E)_j^{n+1} &=(\rho E)_j^n\hlo{-}\frac{\Delta t}{\Delta x}\left(u_{j+1 / 2}^*(\rho E)_{j+1 / 2}^{n+1-}+\Pi_{j+1 / 2}^{\theta,*} u_{j+1 / 2}^*-u_{j-1 / 2}^*(\rho E)_{j-1 / 2}^{n+1-}-\Pi_{j-1 / 2}^{\theta,*} u_{j-1 / 2}^*\right)-\Dt \{\rho u \dx\phi\}_j^n
.
\end{empheq}
\label{split method}
\end{subequations}
 The scheme~\eqref{split method} is proven to be positivity preserving for the density and the internal
 energy as well as entropy stable when $\Delta t$
 verifies both
  the acoustic CFL condition:
\begin{equation}
\frac{\Delta t}{\Delta x} \max _{j \in \mathbb{Z}}\left(\max \left(1/\rho_j^n, 1/\rho_{j+1}^n\right) a_{j+1 / 2}\right) \leq \frac{1}{2},
\label{acoustic CFL}
\end{equation}
and the transport CFL condition:
\begin{equation}
\Delta t \max _{j \in \mathbb{Z}}\left(\left(u_{j-\frac{1}{2}}^*\right)^{+}-\left(u_{j+\frac{1}{2}}^*\right)^{-}\right)<\Delta x ,
\label{transport CFL}
\end{equation}
granted that the following inequality:
\begin{equation}
-\frac{1}{2 a^2}\left(p^{\operatorname{EOS}}\left(\tau_k^{*, \theta}, s_k\right)-\Pi^*\right)^2+\frac{(1-\theta)^2\left(u_{j+1}-u_j\right)^2}{8} \leq 0, \quad k=j, j+1,
\label{Co split}
\end{equation}
 where $\tau_j^{*, \theta}=1/\rho_j^n+\frac{1}{a_{j+1/2}}\left(u^*_{j+1/2}-u_j^n\right)$ and $\tau_{j+1}^{*, \theta}=1/\rho_{j+1}^n+\frac{1}{a_{j+1/2}}\left(u_{j+1}^n-u^*_{j+1/2}\right)$
is satisfied at each interface $j+1/2$.
\hlb{Just like in the original OSLP paper \cite{chalons_all-regime_2016}, the inequality \eqref{Co split} is not ensured 
by any mechanism in the numerical scheme.
As a result, for small values of $\theta$, we cannot guarantee that inequality \eqref{Co split} remains valid. 
This is a known issues of the low-Mach correction proposed in \cite{chalons_all-regime_2016} that is not adressed in the present study.
Let us emphasize that entropy stability can be achieved through alternative criteria 
(see  \cite{Gallice2003} and \cite{Chan2021}), however the study of their performance in the low Mach regime is beyond the scope of this paper.}

In section~\ref{recast}, we discuss how a simple modification of the transport step allows recasting this two-step OSLP algorithm into a one-step FSLP method while keeping the interesting properties of the original method: the well-balanced property, the accuracy in the low Mach regime, mass, and energy positivity and the discrete entropy inequality.

\section{Recasting the OSLP method into a flux-splitting Lagrange-Projection (FSLP) method; a modification of the transport step}\label{recast}

In this section, we discuss how a simple modification of the transport step \eqref{transport step split} 
of the OSLP method \eqref{split method} proposed by \cite{Chalons2016} leads to a much simpler FSLP algorithm.
Flux-splitting methods have been used in many application contexts thanks to their ease of
 implementation that relies on building a discrete evaluation of the fluxes
  (see, for example, \cite{liu_computation_1998,bruneau_time-accurate_1998,evje_hybrid_2002,
 paillere_extension_2003,garcia-cascales_application_2006}).
These methods have been extensively developed for several decades (see, for
 example, \cite{steger_flux_1981,zha1993,liou1993,jameson_analysis_1995,jameson_analysis_1995,
 lioua,liou_sequel_2006,liou_sequel_1996,bouchut_entropy_2003,Toro2012} and the 
 references therein) yielding efficient simulation tools. Unfortunately, deriving theoretical 
 results that ensure the good behavior of these methods is difficult, which contrasts with their good performance in practice. Before going any further, let us mention that the question of building Eulerian numerical fluxes relying on a Lagrangian approximation of the flow equations has been successfully investigated in the literature with different approaches \cite{Dubroca1999, Gallice2000,Gallice2003,bouchut_entropy_2003,Chan2021}.

A key contribution of the present paper is the derivation of stability properties for the flux-splitting algorithm.
These proofs are based on the following observation; let us consider a given hyperbolic problem with a source term for which the set of admissible states is convex (\eg Euler's equations of gas dynamics or ideal Magneto-hydrodynamics);

\begin{equation}
		\dt \bU + \dx \bF ( \bU ) = S (\bU).
        \label{eqexemple}
\end{equation}

We design a separation of the flux and source term into $N$ parts $(F_p,S_p)_{1\leq p\leq N}$ so that:
\begin{align}
    \sum_{p=1}^{N}F_p(\bU)&=F(\bU),& \sum_{p=1}^{N}S_p(\bU)&=S(\bU),
\end{align}
as well as a series of coefficients $\alpha^p_j\in]0,1[$ that sums up to 1; $\sum_{p=1}^{N}\alpha^p_j=1$ 
for each cell $j$. Let us assume that we can build a discretization for each part where the sub-fluxes  and sub-source terms are multiplied by the inverses of the coefficients. 
\hlo{ This allows to consider partially updated value or sub-updated value  $U_j^{p,n+1}$ of the initial state $U_j^n$ due to the influence of to the $p-th$ flux and source term, obtaining the $p-th$ sub-update}:
\begin{equation}
    \frac{\bU^{p,n+1}_j-\bU^{n}_j}{\Dt} - \frac{1}{\alpha_j^p}[\dx F_p(\bU)]_j = \frac{1}{\alpha^p_j}[S_p(\bU)]_j \ \ \ \forall p\in[1,N].
    \label{sub updates}
\end{equation}
Moreover, let us assume that each of these discretizations is stable under their respective 
local CFL condition:
\begin{equation}
	\Dt < \alpha_j^p \frac{\Dx}{v_p^j}\
\end{equation}
where $v_p^j$ is the local characteristic velocity associated with the discretization 
of the $p$-th flux/source term.
By re-assembling the result of each part with the convex combination defined by the coefficients $\alpha^p$,
\begin{equation}
    \bU^{n+1}_j := \sum_{p=0}^{N}\alpha_j^pU^{p,n+1}_j
		\label{defupdate}
\end{equation}
we obtain a discretization consistent with \eqref{eqexemple}, regardless of the value of the
 coefficients $\alpha_j^p\in (0,1) $. The full update is stable as a convex combination of
  the stable sub-updates \eqref{sub updates}. This means we can freely choose the 
  coefficients $\alpha_j^p$ to optimize the CFL condition. Indeed, the update \eqref{defupdate}
   is stable as long as each sub-update is stable i.e.:
\begin{equation}
	\Dt < 
		\min\left(\alpha_j^1 \frac{\Dx}{v_j^1},\ldots,\alpha_j^N \frac{\Dx}{v_j^N}\right)
		.
\end{equation}
For $p=1,\ldots,N$, let us now choose
$\frac{\alpha_j^p}{v_p}=\frac{1}{v^1_j+v^2_j+\cdots+v^N_j}$,
then $\underset{p}{\min}\left(\frac{\alpha_j^p}{v^p_j}\right)=\underset{p}{\min}\left(\frac{1}{v^1_j+v^2_j+\cdots+v^N_j}\right)=\frac{1}{v^1_j+v^2_j+\cdots+v^N_j}$.
This provides the following local CFL condition:
\begin{equation}
	\Dt < \frac{\Delta x}{v^1_j+v^2_j+\cdots+v^N_j}
	.
	\label{CFL general flux}
\end{equation}
In this work, we separate the system into $N=2$ parts corresponding to the pressure and advection terms. This type of splitting is not new and can be found in \cite{zha1993,deshpande_pvu_1994,Toro2012,borah_novel_2016} without entropy stability theorems.
Discretization techniques that also feature a separate treatment for the pressure and advection effects have been proposed for fractional step methods \cite{baraille1992,buffard_conservative_1997,fort_large_2011,coquel_splitting_2012,chalons_all-regime_2016,Chalons2016,Chalons2017,padioleau2019high}.

By modifying the transport step of the original operator splitting algorithm \eqref{split method} by computing the fluxes on the initial states $n$ instead of the acoustic state $n+1-$:

\begin{equation}
\varphi_j^{n+1}=\varphi_j^{n+1-} L_j-\frac{\Delta t}{\Delta x}\left(u_{j+1 / 2}^* \varphi_{j+1 / 2}^{n}-u_{j-1 / 2}^* \varphi_{j-1 / 2}^{n}\right)
\label{transport step unsplit}
\end{equation}
we obtain the following fully conservative update that we refer to as our FSLP method:

\begin{equation}
\left\{\begin{aligned}
\rho_j^{n+1} &=\rho_j^n\hlo{-}\frac{\Delta t}{\Delta x}\left(u_{j+1 / 2}^* \rho_{j+1 / 2}^{n}-u_{j-1 / 2}^* \rho_{j-1 / 2}^{n}\right) \\
(\rho u)_j^{n+1} &=(\rho u)_j^n\hlo{-}\frac{\Delta t}{\Delta x}\left(u_{j+1 / 2}^*(\rho u)_{j+1 / 2}^{n}+\Pi_{j+1 / 2}^{\theta,*}-u_{j-1 / 2}^*(\rho u)_{j-1 / 2}^{n}-\Pi_{j-1 / 2}^{\theta,*}\right)-\Dt \{\rho  \dx\phi\}_j^n, \\
(\rho E)_j^{n+1} &=(\rho E)_j^n\hlo{-}\frac{\Delta t}{\Delta x}\left(u_{j+1 / 2}^*(\rho E)_{j+1 / 2}^{n}+\Pi_{j+1 / 2}^{\theta,*} u_{j+1 / 2}^*-u_{j-1 / 2}^*(\rho E)_{j-1 / 2}^{n}-\Pi_{j-1 / 2}^{\theta,*} u_{j-1 / 2}^*\right)-\Dt \{\rho u \dx\phi\}_j^n .
\end{aligned}\right.
\label{unsplit method}
\end{equation}
Note that we keep the upwind choice for the transport scheme:
\begin{equation}
\varphi_{j+1 / 2}^{n}=\left\{\begin{array}{l}
\varphi_j^{n}, \text { if } u_{j+1 / 2}^* \geq 0, \\
\varphi_{j+1}^{n}, \text { if } u_{j+1 / 2}^*<0,
\end{array}\right.
\label{upwind}
\end{equation}
where $(u,\Pi)^*$ are given by \eqref{u*p*}. We provide the CFL condition associated with the new method:
\begin{equation}
\frac{\Delta t}{\Delta x}  \max _{j \in \mathbb{Z}}\left(2\max \left(1/\rho_j^n, 1/\rho_{j+1}^n\right) a_{j+1 / 2} + \left(u_{j-\frac{1}{2}}^*\right)^{+}-\left(u_{j+\frac{1}{2}}^*\right)^{-} \right)< 1
	\label{CFL unsplit}
\end{equation}
This CFL condition is indeed of the form \eqref{CFL general flux} with $N=2$. It has the same characteristic speeds as the acoustic condition in \eqref{acoustic CFL} and the transport condition in \eqref{transport CFL}, except that they are summed rather than checked separately. \hlo{As a result, \eqref{CFL unsplit} is generally more 
restrictive than conditions \eqref{acoustic CFL}, \eqref{transport CFL}}.
The new method has several advantages compared to the original numerical scheme \eqref{split method}:

\begin{enumerate}
	\item The implementation of the flux-splitting version is much simpler than the operator-splitting version. Indeed, it can be implemented as a standard, simple flux-based finite volume method with the following numerical flux formula:
	\begin{equation}
		\mathbf{F}^{\text{FSLP}}(U_L,U_R) =
		\left\{\begin{aligned} u^*&\rho_{LR} \\ u^*&(\rho u)_{LR} + \Pi^{*,\theta}\\u^*&(\rho E)_{LR} + \Pi^{*,\theta}u^*\end{aligned}\right.
		\label{RS}
	\end{equation}
	with
	\begin{equation}
	\varphi_{LR} =
	\begin{cases}
	\varphi_L& \text{ if  \(u^*>0\),} \\
	\varphi_R& \text{ otherwise.}
	\end{cases}
	\label{upwind flux}
	\end{equation}
	We can see in \eqref{RS} that the flux evaluation clearly separates the pressure-related terms 
	from the advection terms so that it can be affiliated with a family of methods proposed in the
	literature like \cite{zha1993,deshpande_pvu_1994,Toro2012,borah_novel_2016}.
	\item As the method can be implemented as a simple flux-based solver, it can be seamlessly combined with any existing flux-based high-order algorithm such as MUSCL\cite{Leer1977,Leer1977a,Leer1979,toro}, (W)ENO \cite{Liu1994,Jiang1996} or MOOD methods \cite{diot2013multidimensional, clain2011high}. We detail the procedure for the extension to second order in section \ref{MUSCL} and give some numerical examples in section \ref{numerical exp}. Note, however, that the well-balanced treatment of gravity is not straightforward to extend to high order and requires a careful examination that is beyond the scope of this paper. Also, using the low-Mach correction $\theta$ combined with a highly accurate high-order method can amplify numerical instabilities that already exist at first-order (checkerboard modes, for example). We do not address this issue in this paper, as our focus is on demonstrating the recasting of the OSLP method into the FSLP method.

	\item The FSLP method is more computationally efficient than the original OSLP method. 
The OSLP method requires two update loops per time step to compute a time step of 
size $\sim \Dx/\max(v,c)$, where $v$ and $c$ are the velocities associated with transport
 and acoustic effects, respectively, as they appear in the CFL conditions.
In contrast, the FSLP method only requires one loop per time step of size 
$\sim \Dx/(v+c)$. This means that the FSLP method requires fewer sweeps to reach the same 
physical time, especially in the low-Mach regime where $v \ll c$ or in the hypersonic regime 
where $v \gg c$, where it is \hlo{expected to be  more efficient}. If $v = c$, both methods \hlo{should have a comparable efficiency}. \hlo{We provide a 
performance analysis and discussion in section \ref*{sect:perf}.}

	\item The new update formula eliminates the need to store the intermediate state $U^{n+1-}$, as it can be
	 computed in a single sweep. This reduces the algorithm's memory footprint by approximately $2/3$, and 
	 reduces the stencil radius from two to one cell. The decrease in memory storage requirements can improve 
	 performance by reducing the time spent accessing the data arrays.

\end{enumerate}

Despite the update formula being very similar, the mathematical background required to derive the stability 
properties of \eqref{unsplit method} is new. It is the object of the next section \ref{stability}.

\section{Derivation of the stability properties for our new method}\label{stability}

In this section, we focus on deriving the stability properties of our new FSLP scheme \eqref{unsplit method}. 
To this end, we will perform a Suliciu-type relaxation \cite{Suliciu1998,Bouchut2004,chalons2008,coquel2012} of 
the pressure term and introduce a surrogate specific volume. We then isolate two new sub-systems, the advection
 and pressure sub-systems, for which we derive numerical fluxes. We then re-obtain our new method and derive its 
 stability properties by performing a convex combination of the two fluxes. Note that the proof of stability for
  the pressure subsystem is similar to the acoustic sub-system in \cite{chalons_all-regime_2016}. For this 
  reason, we only recall this proof in the appendix for completeness.

\subsection{Relaxation and flux-splitting}
We first apply a relaxation of the original Euler system.
Manipulations of smooth solutions of \eqref{Euler's eqn} gives
 $\dt (\rho p)+ \dx(u\rho p)+ \rho^2 c^2 \dx u = 0$ . We choose to perform a Suliciu-type approximation of the 
 system \eqref{Euler's eqn} for $t\in [t^n, t^{n+1})$ by introducing a surrogate pressure $\Pi$ and considering
  the relaxed system:

 \begin{equation}
		\left\{\begin{aligned}
				\dt \rho+\dx(\mom) &=0, \\
				\dt (\mom)+\dx\left(u\mom+\Pi\right) &=-\rho\dx\phi, \\
				\dt (\rho E)+\dx(u\rho E+\Pi u)&=-\rho u \dx\phi,\\
				\dt (\rho \Pi) + \dx(u \rho \Pi + a^2u) &= \rho \lambda(p - \Pi).
		\end{aligned}\right.
		\label{Euler-Sulicu}
\end{equation}

The parameter $\lambda$ is a frequency that characterizes the strength of the source term that drives 
$\Pi$ towards the equilibrium $\Pi = p$. In the regime $\lambda \rightarrow \infty$, we formally recover
 \eqref{Euler's eqn}. In our numerical solver context, we classically mimic the $\lambda \rightarrow \infty$ 
 regime by enforcing  $\Pi_j^n=p^{\text{EOS}}(1/\rho^n_j,e_j^n) $ at each time step and then solving 
 \eqref{Euler-Sulicu} with $\lambda =0$, which will be the case in all computations below without any ambiguities. 
 We now introduce another auxiliary variable $\Tau$ and impose that it verifies
     \begin{equation}
         \dt (\rho\Tau)  = 0.
         \label{evolution rhoTau}
     \end{equation}
We suppose that $\Tau(t=0) = 1/\rho(t=0)$ at the initial instant so that $\Tau(x,t)$ is equal to the specific 
volume \(1/\rho(x,t) \) for all $x$ and $t>0$. Let us now re-write the 
system~\eqref{Euler-Sulicu}-\eqref{evolution rhoTau} in order to highlight
three different operators  that compose the flux and the source term of
 \eqref{Euler-Sulicu}-\eqref{evolution rhoTau} following similar lines as \cite{zha1993,deshpande_pvu_1994,borah_novel_2016}
\begin{align}
\dt
\begin{pmatrix}
	\rho\\ \rho u \\ \rho E \\ \rho\Pi \\ \rho\Tau
\end{pmatrix}
+
\dx
\begin{pmatrix}
	\rho u\\ \rho u^2 \\ \rho E u \\ \rho\Pi u \\ u
\end{pmatrix}
+
\dx
\begin{pmatrix}
	0 \\ \Pi \\ \Pi u \\ a^2 u \\ -u
\end{pmatrix}
=
-
\begin{pmatrix}
	 0 \\ \rho  \\ \rho u \\ 0 \\ 0
\end{pmatrix}
\dx \phi.
\label{eq: operators highlight}
\end{align}
Let us underline that both $\Pi$ and $\rho\Tau$ are only mathematical intermediates used to derive the 
scheme's stability properties. Indeed, these variables
do not appear in the update formula \eqref{unsplit method}, so that there is no need
to evaluate and store them while implementing the algorithm. Let us introduce the convex combination
 parameter $\alpha\in ]0,1[ $ and
 two subsystems associated with different parts of the fluxes and source terms featured in
  \eqref{eq: operators highlight}. The first system gathers the source term and the flux associated with 
  pressure terms ponderated by $1/\alpha$
	\begin{equation}
	\left\{\begin{aligned}
		\dt \rho&=0, \\
		\dt (\mom)+\aone\dx\left(\Pi\right) &=-\aone\rho\dx\phi, \\
		\dt (\rho E)+\aone\dx(\Pi u)&=-\aone\rho u \dx\phi,\\
		\dt (\rho \Pi) + \aone\dx( a^2u) &= 0,\\
		\dt (\rho\Tau ) - \aone\dx u &= 0.
	\end{aligned}\right.
	\label{pressure system}
\end{equation}
We will refer to \eqref{pressure system} as the pressure system. The second sub-system is 
composed of the remaining terms that pertain to transport effects ponderated by $1/1-\alpha$, it reads
\begin{equation}
	\left\{\begin{aligned}
		\dt (\rho\varphi) + \atwo\dx (u\rho \varphi ) &= 0,\qquad \text{$\varphi \in \{1, u,  E,  \Pi\}$} \\
		\dt (\rho\Tau) + \atwo\dx u &= 0,
	\end{aligned}\right.
	\label{advection system}
\end{equation}
and will be called the advection system.

The pressure system~\eqref{pressure system} is hyperbolic and involves the characteristic velocities $\{\pm \aone a/\rho, 0, 0, 0\}$ that are all associated with linearly degenerate fields. The advection system~\eqref{advection system} is only weakly hyperbolic as its Jacobian matrix admits $(1-\alpha)u$ as multiple eigenvalues but is not diagonalizable. Nevertheless, let us underline that the algorithms we will consider
for approximating  the solutions of \eqref{advection system} will verify a local maximum principle under a CFL condition so that stability will be ensured for the advection step (see section~\ref{sect:cvx}).

Before continuing, let us comment on equations \eqref{pressure system} and \eqref{advection system}. The factors $\alpha$ and $1-\alpha$ that appear in the fluxes and source terms of these equations correspond to the case $N=2$ of the flux splitting stability argument presented at the beginning of section \ref{recast}.

Then, although
the trivial equation stationary~\eqref{evolution rhoTau} is now split into two non-stationary
parts within \eqref{pressure system} and \eqref{advection system}, the overall scheme will
indeed guarantee that $(\rho \Tau)_j^n=1$ for $j \in \mathbb{Z}$ and $n\in \mathbb{N}$.

 \subsection{The convex combination}\label{sect:cvx}
 We propose the following discretization strategy:

 \begin{enumerate}
     \item Compute $\bm{U}_j^P$ as the update of the initial state $\bm{U}_j^{n}$ by approximating the solution of \eqref{pressure system}:
    \begin{equation}
    \left\{\begin{aligned}
    \rho_{j}^P &=\rho_{j}^{n},\\
    (\mom)_{j}^P &=(\mom)_{j}^{n}-\frac{1}{\alpha}\frac{\Dt}{\Dx}\left(\Pi_{j+1 / 2}^{*,\theta}-\Pi_{j-1 / 2}^{*,\theta}\right)-\frac{1}{\alpha}\Dt \left\{\rho \dx \phi\right\}_j^n, \\
    (\rho E)_{j}^P &=(\rho E)_{j}^{n}-\frac{1}{\alpha}\frac{\Dt}{\Dx}\left(\Pi_{j+1 / 2}^{*,\theta} u_{j+1 / 2}^{*}-\Pi_{j-1 / 2}^{*,\theta} u_{j-1 / 2}^{*}\right)-\frac{1}{\alpha}\Dt \left\{\mom\dx \phi\right\}_j^n,\\
    (\rho \Pi)_j^P&= (\rho \Pi)_j^n -\frac{1}{\alpha}\frac{\Dt}{\Dx}\left(a_{j+1/2}^2u^*_{j+1/2}-a_{j-1/2}^2 u^*_{j-1/2}\right),\\
    (\rho \Tau)_{j}^P&=1+\frac{1}{\alpha}\frac{\Dt}{\Dx}\left(u_{j+1 / 2}^{*}-u_{j-1 / 2}^{*}\right).
    \label{pressure step}
\end{aligned}\right.
\end{equation}

       \item Compute $\bm{U}_j^A$ as the update of the initial state $\bm{U}_j^{n}$ by approximating the solution of \eqref{advection system}: for $\varphi \in \{1, u,  E,  \Pi\}$
			 \begin{equation}
			 \left\{\begin{aligned}
			 		(\rho\varphi)_j^{A}=(\rho\varphi)_j^{n} -\frac{1}{1-\alpha}\frac{\Delta t}{\Delta x}\left(u_{j+1 / 2}^* (\rho\varphi)_{j+1 / 2}^{n}-u_{j-1 / 2}^* (\rho\varphi)_{j-1 / 2}^{n} \right),\\
					(\rho \Tau)^A_j= (\rho \Tau)^n_j -\frac{1}{1-\alpha}\frac{\Delta t}{\Delta x}\left(u_{j+1 / 2}^* -u_{j-1 / 2}^*\right).
										\end{aligned}\right.
	\label{advection step}
	 \end{equation}
     \item Evaluate $\bm{U}_j^{n+1}$ as the convex combination of $\bm{U}_j^P$ and $\bm{U}_j^A$:
		 \begin{equation}
		 \bm{U}_j^{n+1}= \alpha\bm{U}_j^P+(1-\alpha)\bm{U}_j^A
		 \label{average}
		 \end{equation}
 \end{enumerate}

It can be verified that the update in \eqref{average} is equivalent to the FSLP scheme in \eqref{unsplit method}, for any value of $\alpha \in ]0,1[$. This means that the flux of the FSLP scheme can be expressed as an arbitrary convex combination of the fluxes involved in the update. As explained at the beginning of section \ref{recast}, this interpretation allows us to choose $\alpha$ optimally in order to obtain the least restrictive CFL condition, given by \eqref{CFL unsplit}.
\begin{remark}
	\hlo{The original operator splitting method proposed by \cite{chalons_all-regime_2016} 
	coincides with a Lagrange-Projection scheme when used in a 1D context. As a result, 
	the Lagrange-Projection appellation is used to design finite 
	volume, acoustic/transport operator splitting methods for various hyperbolic systems 
	in the literature. However, for 2D problems, the OSLP method does not
	correspond to a Lagrange-Projection method despite sharing similarities with the 1D 
	version. Also, Lagrange-Projection methods are operator-splitting methods consisting of a
	Lagrange step and a projection step. Our method does not split operators but fluxes, so we doubt 
	it can still be interpreted as a Lagrange-Projection method. However, we choose to keep the 
	appellation as FSLP inherits its formula from the line of work stemming from the 
	Lagrange-Projection literature.}
\end{remark}

\subsection{Stability of the pressure step}

In this section, we prove the stability of the pressure step. We chose to move all the derivations in the appendix as the arguments we use are already present in \cite{Chalons2016} in the proof of the stability of the acoustic step \eqref{acoustic step} of the OSLP method \eqref{split method}. We introduce the CFL condition associated with the pressure step.

\begin{equation}
	\frac{1}{\alpha}\frac{\Delta t}{\Delta x} \max _{j \in \mathbb{Z}}\left( \max \left(1/\rho_j^n, 1/\rho_{j+1}^n\right) a_{j+1 / 2}\right) \leq \frac{1}{2},
\label{pressure CFL}
\end{equation}
It is identical to the acoustic CFL \eqref{acoustic CFL} but $1/\alpha$ times as restrictive.
\begin{proposition}\label{prop:stab pressure}

Suppose that $a$ is chosen large enough so that \eqref{def a} is verified and that both $\Tau_L^*>0$, $\Tau_R^* >0$ from \eqref{eq: Tau_k^* def} are positive. Suppose also that the low-Mach correction $\theta$  is chosen large enough so that \eqref{inegalité sur Co} is valid. Under the CFL condition~\eqref{pressure CFL} we have that:
\begin{enumerate}
  \item the density and the internal energy verify $\rho^P_j>0$ and $e^P_j>0$,
  for all $j$,
		\item the discretization \eqref{pressure step} satisfies the entropy inequality
		\begin{equation}
					\rho^P_js^{\text{EOS}}(\Tau^P_j,e^P_j) - \rho^n_js(1/\rho_j^n,e^n_j) + \frac{1}{\alpha}\frac{\Dt}{\Dx}( q_{j+1/2}^n -  q_{j-1/2}^n ) \geq 0,
					\label{ineq pressure}
				\end{equation}
				with $q^n_{j+ 1/2}=q_\Delta(\bU^n_j,\bU^n_{j+1})$, where $q_\Delta$ is a flux function consistent with $0$ as $\Dt,\Dx\to{}0$.
\end{enumerate}

		\begin{proof}
			 The positivity of the internal energy and the entropy inequality of \eqref{ineq pressure} are direct consequences of the approximate Riemann solver properties of proposition~\ref{prop: e*>0 and s*>s modified scheme} and the consistency in the integral sense~\cite{Bouchut2004}.
		\end{proof}

\end{proposition}

 The condition \eqref{inegalité sur Co} is identical to the OSLP low-Mach stability condition \eqref{Co split} but with the surrogate density $\Tau$ instead of $1/\rho$. In practice, conditions for stability for the pressure update are strictly the same as for the acoustic step \eqref{acoustic step} from \cite{chalons_all-regime_2016} apart from the factor $1/\alpha$ in the CFL condition.

\subsection{Stability of the advection step}
We introduce the CFL condition associated with the advection step.

\begin{equation}
	\frac{1}{1-\alpha}\frac{\Delta t}{\Delta x} \max _{j \in \mathbb{Z}}  \left( \left(u_{j-\frac{1}{2}}^*\right)^{+}-\left(u_{j+\frac{1}{2}}^*\right)^{-}\right)<1 ,
\label{CFL Advection}
\end{equation}
It is identical to the transport CFL \eqref{transport CFL} but $1/(1-\alpha)$ times as restrictive.
\begin{proposition}\label{prop:stab advection}
Under the CFL condition \eqref{CFL Advection}, the discretization~\eqref{advection step} of the advection subsystem verifies the following properties.
\begin{enumerate}
	\item $\bU_j^A$ is a positive linear combination of $\bU_{j-1}^n$, $\bU_{j}^n$ and $\bU_{j+1}^n$.
	\item $b_j^A$ is a convex  combination of $b_{j-1}^n$, $b_{j}^n$ and $b_{j+1}^n$ for $b\in\{u,E,\Tau\}$.
	\item if $e_j^n>0$ for all $j\in\mathbb{Z}$ then $e_j^A>0$ for all $j\in\mathbb{Z}$,
	\item The discretization \eqref{advection step} satisfies the entropy inequality
	 \begin{equation}
			\rho_j^A s^{\text{EOS}}\left(\mathcal{T}_j^A, e_j^A\right)-\rho_j^n s_j^n+\atwo\frac{\Delta t}{\Dx}\left( u_{j+1 / 2}^* \rho_{j+1 / 2}^n s_{j+1 / 2}^n- u_{j-1 / 2}^* \rho_{j-1 / 2}^n s_{j-1 / 2}^n\right) \geq 0
			\label{advection entropy}
		\end{equation}
	\end{enumerate}
\begin{proof}
The advection scheme~\eqref{advection step} can be recast into
\begin{equation}
	\bU^A_j =
	-\atwo\frac{\Dt}{\Dx}u_{j+1 / 2}^{*,-} \ \bU^n_{j+1}
	+
	\atwo\frac{\Dt}{\Dx}u_{j-1 / 2}^{*,+}  \ \bU^n_{j-1}
	+ \left[  1 - \atwo\frac{\Dt}{\Dx}(u_{j+1 / 2}^{*,+} -u_{j-1 / 2}^{*,-}) \right]  \ \bU^n_{j},
	\label{eq: positive combination advection}
\end{equation}
which proves 1. One can also write

\begin{equation}
\qty(\frac{\bU}{\rho})^A_j =
\lambda_j^{(+1)}
\qty(\frac{\bU}{\rho})^n_{j+1}
+
\lambda_j^{(0)}
\qty(\frac{\bU}{\rho})^n_{j}
+
\lambda_j^{(-1)}
\qty(\frac{\bU}{\rho})^n_{j-1}
,
\end{equation}
with
\begin{align}
\lambda_j^{(+1)}
&=
-\atwo\frac{\Dt}{\Dx}u_{j+1 / 2}^{*,-} \qty(\frac{\rho_{j+1}^n}{\rho_j^A})
,&
\lambda_j^{(0)}
&=
\left[  1 - \atwo\frac{\Dt}{\Dx}(u_{j+1 / 2}^{*,+} -u_{j-1 / 2}^{*,-}) \right]
\qty(\frac{\rho_{j}^n}{\rho_j^A})
,&
\lambda_j^{(-1)}
&=
\atwo\frac{\Dt}{\Dx}u_{j-1 / 2}^{*,+}  \qty(\frac{\rho_{j-1}^n}{\rho_j^A})
.
\end{align}
By \eqref{eq: positive combination advection} we have that
\begin{equation}
\rho^A_j =
-\atwo\frac{\Dt}{\Dx}u_{j+1 / 2}^{*,-} \ \rho^n_{j+1}
+
\atwo\frac{\Dt}{\Dx}u_{j-1 / 2}^{*,+}  \ \rho^n_{j-1}
+ \left[  1 - \atwo\frac{\Dt}{\Dx}(u_{j+1 / 2}^{*,+} -u_{j-1 / 2}^{*,-}) \right]  \ \rho^n_{j},
\label{eq: rhojA combination}
\end{equation}
so that $\lambda_j^{(+1)}+\lambda_j^{(0)}+\lambda_j^{(-1)}=1$,
which proves that $b_j^A$ is a convex  combination of $b_{j-1}^n$, $b_{j}^n$ and $b_{j+1}^n$ for $b\in\{u,E\}$. Let us now consider the case of $\Tau^A$. By \eqref{advection step}, we have that
\begin{equation}
\Tau_j^A
=
\Tau_j^n \frac{\rho_j^n}{\rho_j^A}
-\atwo\frac{\Dt}{\Dx}u_{j+1 / 2}^{*,+} \frac{1}{\rho_j^A}
-\atwo\frac{\Dt}{\Dx}u_{j+1 / 2}^{*,-} \frac{1}{\rho_j^A}
+\atwo\frac{\Dt}{\Dx}u_{j-1 / 2}^{*,+} \frac{1}{\rho_j^A}
+\atwo\frac{\Dt}{\Dx}u_{j-1 / 2}^{*,-} \frac{1}{\rho_j^A}
.
\end{equation}
However, since we chose $\rho^n_j\Tau_j^n = 1$ for all $i\in\mathbb{Z}$, We can write that
\begin{align}
\Tau_j^A
&=
\Tau_j^n \frac{\rho_j^n}{\rho_j^A}
-\atwo \frac{\Dt}{\Dx}u_{j+1 / 2}^{*,+} \frac{\rho_{j}^n}{\rho_j^A}\Tau_{j}^n
-\atwo \frac{\Dt}{\Dx}u_{j+1 / 2}^{*,-} \frac{\rho_{j+1}^n}{\rho_j^A}\Tau_{j+1}^n
+\atwo \frac{\Dt}{\Dx}u_{j-1 / 2}^{*,+} \frac{\rho_{j-1}^n}{\rho_j^A}\Tau_{j-1}^n
+\atwo \frac{\Dt}{\Dx}u_{j-1 / 2}^{*,-} \frac{\rho_{j}^n}{\rho_j^A}\Tau_{j}^n
\\
&=
\lambda_j^{(+1)}\Tau_{j+1}^n
+
\lambda_j^{(-1)}\Tau_{j-1}^n
+
\lambda_j^{(0)}
\Tau_j^n
.
\end{align}
Consequently
$\Tau_j^A$ is also a  convex  combination of $\Tau_{j-1}^n$, $\Tau_{j}^n$ and $\Tau_{j+1}^n$, which proves 2.
For statement 3, we consider the concave function $K$ introduced in the proof of
lemma~\ref{Omega is convex}, and we have that
$e^A_j = K(u_j^A,E_j^A)$. Thanks to statement 2, we can thus write that.
\begin{multline}
e^A_j =
K\qty(\sum_{k=0,\pm 1} \lambda_{j}^{(k)} u_{j+k}^n, \sum_{k=0,\pm 1} \lambda_{j}^{(k)} E_{j+k}^n)
\geq
\sum_{k=0,\pm 1} \lambda_{j}^{(k)} K\qty(u_{j+k}^n, E_{j+k}^n)
=
\sum_{k=0,\pm 1} \lambda_{j}^{(k)}
e_{j+k}^n
>0
,
\end{multline}
which proves statement~3.

Now using the lemma~\ref{Omega is convex}, we have that
\begin{align}
s(\Tau_j^A,e_j^A)
&=
\specificentropy(\Tau_j^A,u_j^A,E_j^A)
=
\specificentropy\qty(
\sum_{k=0,\pm 1}\lambda_j^{(k)}\Tau_{j+k}^n
,
\sum_{k=0,\pm 1}\lambda_j^{(k)} u_{j+k}^n
,
\sum_{k=0,\pm 1}\lambda_j^{(k)} E_{j+k}^n
)
\\
&\geq
\sum_{k=0,\pm 1}\lambda_j^{(k)}
\specificentropy\qty(
\Tau_{j+k}^n
,
u_{j+k}^n
,
E_{j+k}^n
)
=
\sum_{k=0,\pm 1}\lambda_j^{(k)}
s(
\Tau_{j+k}^n
,
e_{j+k}^n
).
\end{align}
This inequality also reads
\begin{equation}
s(\Tau_j^A,e_j^A)
\geq
s_j^n \frac{\rho_j^n}{\rho_j^A}
-\atwo \frac{\Dt}{\Dx}u_{j+1 / 2}^{*,+} \frac{\rho_{j}^n}{\rho_j^A}s_{j}^n
-\atwo \frac{\Dt}{\Dx}u_{j+1 / 2}^{*,-} \frac{\rho_{j+1}^n}{\rho_j^A}s_{j+1}^n
+\atwo \frac{\Dt}{\Dx}u_{j-1 / 2}^{*,+} \frac{\rho_{j-1}^n}{\rho_j^A}s_{j-1}^n
+\atwo \frac{\Dt}{\Dx}u_{j-1 / 2}^{*,-} \frac{\rho_{j}^n}{\rho_j^A}s_{j}^n
.
\label{entropy advection intermediate ineq}
\end{equation}
If the CFL condition~\eqref{CFL Advection} is met then $\rho_j^A\geq 0$ and by multiplying
\eqref{entropy advection intermediate ineq} by $\rho_j^A$ we get \eqref{advection entropy}.
\end{proof}

\end{proposition}
\subsection{Stability of the FSLP method}
\begin{proposition}
	If the following conditions are met
	\begin{enumerate}
        \item  The CFL condition \eqref{CFL unsplit} is met,
		\item the parameter $a$ is large enough so that \eqref{def a} is verified and $\Tau_L^*>0$, $\Tau_R^*>0$ in \eqref{eq: Tau_k^* def} for all $j\in\mathbb{Z}$,
		\item both density and internal energies are positive, i.e.
		$\rho_j^{n}>0$ and $e_j^{n}>0$ for all $j\in\mathbb{Z}$,
		\item the parameter $\theta$ is large enough so that \eqref{inegalité sur Co} is valid at each interface,
	\end{enumerate}
	then the flux-splitting update~\eqref{unsplit method} \begin{enumerate}
		\item[\hlo{(a)}] preserves positivity for both density and internal energy i.e. $\rho_j^{n+1}>0$ and $e_j^{n+1}>0$ for all $j\in\mathbb{Z}$,
		\item[\hlo{(b)}] is endowed with the following entropy inequality:
		\begin{equation}
			\rho^{n+1}_js(1/\rho_j^{n+1}, e_j^{n+1})- \rho^n_js(1/\rho^n_j,e^n_j)+ \frac{\Delta t}{\Dx}\left(  Q_{j+1/2}-Q_{j-1/2}\right) \geq 0,
			\label{entropy inequality bla}
		\end{equation}
		with $  Q_{j+1/2} =  u_{j+1 / 2}^{*}\rho_{j+1/2}^n s_{j+1/2}^n
		+  q_{j+1 / 2}^{n} $.
	\end{enumerate}

\begin{proof}
   (a) Let us start by ensuring that the CFL conditions  \eqref{pressure CFL}, \eqref{CFL Advection}
    are satisfied so that the advection and pressure steps are stable. By choosing $\alpha$ so 
	that $\alpha c_j = (1-\alpha)v_j =v_j+c_j$ where $v_j= 
	\left( \left(u_{j-\frac{1}{2}}^*\right)^{+}-\left(u_{j+\frac{1}{2}}^*\right)^{-}\right)$ and 
	$c_j=2\max\left[\max \left(1/\rho_{j-1}^n, 1/\rho_{j}^n\right) a_{j-1 / 2},\max \left(1/\rho_j^n, 1/\rho_{j+1}^n\right) a_{j+1 / 2}\right]$,
	 it is straightforward that \eqref{pressure CFL}, \eqref{CFL Advection} are equivalent and
	  corresponds to  \eqref{CFL unsplit}. This choice of $\alpha$ seems local as it depends on
	  the characteristic speed of each cell considered. However, it can be chosen globally as the minimizer of $f(\alpha)=\underset{j}{max}(\alpha c_j,(1-\alpha )v_j)=c_{i}+v_{i}$ where $i$ 
	  is the index of the cell with the largest speed sum of the simulation domain.
	
	Thanks to the propositions~\ref{prop:stab advection} and ~\ref{prop:stab pressure} we
	 have
	$\rho^{A}>0$, $\rho^{P}>0$, $e^{A}>0$ and $e^{P}>0$ thus, the positivity is straightforward
	 for the density as $\rho_j^{n+1} = (\rho^A_j + \rho^P_j)/2>0$. For the internal energy, 
	 we consider the function $\Lambda((\rho, \rho u, \rho E))=(\rho E) - \frac{(\rho u)^2}{2\rho}$, 
	 that is proven to be concave in \ref{Omega is convex}. We have that:
\begin{multline}
		(\rho e)_j^{n+1} = (\rho E)_j^{n+1}
	- \frac{((\rho u)_j^{n+1})^2}{2\rho_j^{n+1}}
	=
	\Lambda(\rho_j^{n+1}, (\rho u)_j^{n+1}, (\rho E)_j^{n+1})
	\\=
	\Lambda
	\qty(
	(1-\alpha )\rho_j^{A}+\alpha\rho_j^{P},
	(1-\alpha ) (\rho u)_j^{A}+\alpha(\rho u)_j^{P},
	(1-\alpha ) (\rho E)_j^{A}+\alpha(\rho E)_j^{P}
	)
	\\
	\geq
	(1-\alpha )
	\Lambda
	\qty(
	\rho_j^{A},
	(\rho u)_j^{A},
	(\rho E)_j^{A}
	)
+
\alpha
	\Lambda
	\qty(
	\rho_j^{P},
	(\rho u)_j^{P},
	(\rho E)_j^{P}
	)
	=
	(1-\alpha ) (\rho e)_j^A
		+
\alpha 	(\rho e)_j^P>0
\end{multline}
by concavity.

For (b): propositions \ref{prop:stab advection} and \ref{prop:stab pressure} ensure that both entropy inequalities \eqref{ineq pressure} and \eqref{advection entropy} are satisfied. We then use the concavity of the function $\entropy(\rho,\rho\Tau,\rho u,\rho E)=\rho s
\qty(\frac{\rho\Tau}{\rho},
\frac{(\rho E)}{\rho}
-
\frac{(\rho u)^2}{2\rho^2})$ that is proven in \ref{Omega is convex} and the fact that $(\rho \Tau)^{n+1}_j=\alpha (\rho \Tau)^P_j+(1-\alpha )(\rho \Tau)^A_j=1$. Noting $\alpha^P=\alpha$ and $\alpha^A=1-\alpha$, we have:
	\begin{multline}
		\rho^{n+1}_j s(1/\rho_j^{n+1},e_j^{n+1})
		=
		\rho^{n+1}_j s\qty(
			\frac{1}{\rho_j^{n+1}},
			\frac{(\rho E)_j^{n+1}}{\rho_j^{n+1}}
		-
		\frac{1}{2}
		\qty(
			\frac{(\rho u)_j^{n+1}}{\rho_j^{n+1}})^2
		)
		=
		\entropy(\rho_j^{n+1},1,(\rho u)_j^{n+1},(\rho E)_j^{n+1})
		\\
		=
		\entropy(\rho_j^{n+1},(\rho\Tau)_j^{n+1},(\rho u)_j^{n+1},(\rho E)_j^{n+1})
		=
		\entropy
		\qty(
		\sum_{k=A,P}\alpha^k \rho_j^{k},
		\sum_{k=A,P}\alpha^k (\rho\Tau)_j^{k},
		\sum_{k=A,P}\alpha^k (\rho u)_j^{k},
		\sum_{k=A,P}\alpha^k (\rho E)_j^{k}
		).
\end{multline}
\hlo{Thanks to appendix \ref{sect:cvx} we know that $\entropy$ is concave and thus we have:}
\begin{equation}
	\rho^{n+1}_j s(1/\rho_j^{n+1},e_j^{n+1})
	\geq
	\sum_{k=A,P}
	\alpha^k
	\entropy
	\qty(
	\rho_j^{k},
	(\rho\Tau)_j^{k},
	(\rho u)_j^{k},
	(\rho E)_j^{k}
	)	=
	\sum_{k=A,P}
	\alpha^k
	\rho_j^{k}
	s
	\qty(
	\Tau_j^{k},
	e_j^{k}
	)
	.
\end{equation}
by concavity. Using \eqref{ineq pressure} and \eqref{advection entropy}, we get:
\begin{multline}
	\rho^{n+1}_j s\qty(1/{\rho_j^{n+1}},e_j^{n+1})
\geq
\rho^{n}_j s\qty({1}/{\rho_j^{n}},e_j^{n})
-\frac{\Dt}{\Dx}(q_{j+1/2}^n-q_{j-1/2}^n)
-\frac{\Dt}{\Dx}(
	{u}^*_{j+1/2}{\rho}^n_{j+1/2}{s^n_{j+1/2}
-
{u}^*_{j-1/2}{\rho}^n_{j-1/2}s^n_{j-1/2}}
),
\end{multline}
which proves (b).
\end{proof}

\end{proposition}

\section{Low Mach behavior, extension to multi-dimensional and higher order of accuracy}
In this section, we briefly address the behavior of the scheme in the low Mach regime and propose simple means to extend the FSLP method to multi-dimensional problems and improve its accuracy with  higher-order techniques.
\subsection{Low Mach behavior}\label{section: low mach}
Many simulation cases involve flows in which the material velocity is relatively low compared to the sound velocity.
A common way to characterize this situation is to consider the numbers
\(L\), \(t_0\), \(\rho_0\), \(u_0\), \(p_0\), \(u_0=p_0\rho_0\), \(c_0=\sqrt{p_0/\rho_0}\) and
\((\dx\phi)_0\)
that are the characteristic magnitudes for length, time, density, velocity, pressure, sound velocity, and $\dx\phi$, respectively.
We then introduce the following non-dimensional variables:
\(\tilde{x} = {x}/{L}\),
\(\tilde{t} = {t}/{t_0}\),
\(\tilde{\rho} = {\rho}/{\rho_0}\),
\(\tilde{u} = {u}/{u_0}\),
\(\tilde{e} = {e}/{e_0}\),
\(\tilde{p} = {p}/{p_0}\),
\(\widetilde{(\dx\phi)} = {\dx\phi}/{(\dx\phi)_0}\),
and we define the Mach number $\mach$ and the Froude number $\froude$ by $\mach=u_0/c_0$ and $\froude=u_0/\sqrt{L(\dx\phi)_0}$.
Following \cite{Bispen2017,thomann2020all}, we consider a particular flow regime such that $\mach=\froude$ so that the system~\eqref{Euler's eqn} takes the following non-dimensional form
\begin{subequations}
	\begin{align}
		\partial_{\tilde{t}}{\tilde{\rho}}
		+
		\partial_{\tilde{x}}({\tilde{\rho}\tilde{u}})
		&=0
		,
		&
		\partial_{\tilde{t}}({\tilde{\rho}\tilde{u}})
		+
		\partial_{\tilde{x}}({\tilde{\rho}\tilde{u}^2})
		+
		\frac{1}{\mach^2}
		\qty(\partial_{\tilde{x}}\tilde{p} + \tilde{\rho}\widetilde{(\dx\phi)})
		&=
		0
		,
		&
		\partial_{\tilde{t}}({\tilde{\rho}\tilde{E}})
		+
		\partial_{\tilde{x}}(
		{\tilde{\rho}\tilde{E}\tilde{u}}
		+
		{\tilde{p}\tilde{u}}
		)
		&=
		-
		\tilde{\rho}\tilde{u}\widetilde{(\dx\phi)}
		.
	\end{align}
	\label{eq: non dimensional eq}
\end{subequations}
Thanks to system~\eqref{eq: non dimensional eq}, one can see that in the limit $\mach\to 0$, a singularity may appear in the momentum equation.
Supposing now that $\mach \ll 1$, this suggests to distinguish two cases
similarly as in \cite{chalons_all-regime_2016}: in the first case the term
\(\partial_{\tilde{x}}\tilde{p} + \tilde{\rho}\widetilde{(\dx\phi)}\)
will always remain of magnitude $O(\mach^2)$, so that $\tilde{\rho}$, $\tilde{u}$ and
$\tilde{E}$ will also remain of order $O(\mach^0)$. In this case, we will say that the system is in the low Mach regime. In the second case, the term
\(\partial_{\tilde{x}}\tilde{p} + \tilde{\rho}\widetilde{(\dx\phi)}\)
will not remain of magnitude $O(\mach^2)$ in such way that $\tilde{\rho}\tilde{u}$,
 may experience large variations from $O(\mach^2)$ to $O(\mach^0)$, yielding
 significant growth of $\mach$ and thus a change in the Mach regime. These variations characterize  all-regime flows with respect to the Mach number. Let us remark that
 the finer definition of well-prepared initial conditions used in \cite{Bispen2017} verifies the looser notion of low Mach regime considered in this work.

As it was mentioned earlier, the behavior of the Euler equations in the low Mach regime
and adapted simulation strategies raise issues that have been intensively investigated for many years and are still very actively studied (see \cite{Turkel1987,Guillard1999,Paillere2000,
Guillard2004,Beccantini2008,Rieper2009,Dauvergne2008,Dellacherie2010,Degond2011, Cordier2012,
Dellacherie2010a, chalons_all-regime_2016, Dellacherie2016, Zakerzadeh2016,
Barsukow2021,Zakerzadeh2016,Bispen2017, Berthon2020,Dimarco2017,Boscarino2018,Bouchut2020a,
Dimarco2018,Bruel2019,Boscheri2020,Bouchut2020,Zeifang2020}
and the references therein). In this work, we propose transposing the low-Mach error analysis of the OSLP method presented in \cite{chalons_all-regime_2016} to the FSLP scheme.
This task is straightforward, although it requires lengthy and tedious calculations. Therefore, for the sake of brevity, we only recall the main points of this approach.
We consider a non-dimensional expression of the FSLP solver for a one-dimensional problem and evaluate the truncation error obtained with a smooth solution of \eqref{eq: non dimensional eq} that satisfies the low Mach regime hypothesis
\(\partial_{\tilde{x}}\tilde{p} + \tilde{\rho}\widetilde{(\dx\phi)} = O(\mach^2)\).
Similarly to the OSLP scheme, the magnitudes of the resulting truncation error estimates
are uniform with respect to $\mach$ except for the momentum equation that features
an error term of order $O(\theta\Delta x/\mach)$. Consequently
, choosing $\theta=O(\mach)$ when $\mach\ll 1$ will help the scheme preserve a uniform truncation error with respect to $\mach$. A well-known consequence of this choice is that in regions where $\mach\ll 1$, the non-centered part of the pressure term $\Pi^{*,\theta}_{j+1/2}$ will be moderated.

The numerical tests proposed in sections~\ref{section: gresho}, \ref{section: RT} and \ref{section: vortex with gravity} show that this simple correction work similarly for both FSLP and OSLP methods: in the low Mach regime, both schemes provide accurate results. Nevertheless, we need to emphasize that the modification of the scheme induced by $\theta$ is not flawless and should be considered with care.
Spurious oscillations may occur \cite{Dellacherie2009, Jung2022} and the inequality \eqref{inegalité sur Co}
that ensures the entropy property of the scheme may not be verified in the limit $\mach\to 0$.

Let us finally highlight that as in \cite{chalons_all-regime_2016,padioleau2019high} the present approach is rather pragmatic and does not provide reliable analysis and explanation
for the low Mach issues. Indeed, we do not study the delicate question of the
asymptotic regime $\mach\to 0$ \cite{Degond2011, Cordier2012,Zakerzadeh2016,Bispen2017,Berthon2020,Dimarco2017,Boscarino2018,Bouchut2020a}, we neither address the strong time step limitation due to the CFL conditions~\eqref{pressure CFL} when  $\mach \ll 1$ that can be circumvented by using Implicit-Explicit strategies\cite{chalons_all-regime_2016,Dimarco2018,Boscheri2020,Bouchut2020,Zeifang2020}. 
\hlb{It seems possible to adapt the OSLP Implicit-Explicit strategy of \cite{chalons_all-regime_2016} to the FSLP method. However such task falls beyond the scope of the present and will be investigated in future works.}
Moreover, the present lines are derived within a one-dimensional setting that does not allow fully expressing issues related to low Mach flows.

\subsection{Extension to higher order}\label{MUSCL}

The FSLP algorithm can be implemented thanks to a simple single-step evaluation
of numerical fluxes. This enables the use of classical high-order enhancements
that are available in the literature for finite volume methods such as MUSCL-Hancock \cite{Leer1977,Leer1977a,Leer1979,toro, Godlewski2021, LeVeque2002}, (W)ENO \cite{Liu1994,Jiang1996} or MOOD \cite{diot2013multidimensional, clain2011high}. For the sake of simplicity, in this paper, we will only show numerical results with the MUSCL method for which the positivity can be proven under a half CFL condition. Let us consider a linear reconstruction of the primitive variables $\bm{V}=(\rho, u, p)$ in each cells

\begin{equation*}
	\tilde{\bm{V}_j^n}(x) = \bm{V}_j^n +(x-x_j)\bm{p^n_j}
\end{equation*}

where the slopes $\bm{p}_j^n = \bm{p}^n (\bm{V}_{j-1}^n, \bm{V}_j^n, \bm{V}_{j+1}^n)$ are obtained using a standard slope limiter such as the minmod function \cite{yee1989class}. Let us introduce the function $H:\ \bm{V}\mapsto \bm{U}$ that converts a state's conservative representation into its corresponding set of primitive variables. The reconstruction provides a second-order evaluation of the conserved quantities at each interface with
\begin{align}
	\bU_{j+1/2,-}^{n,HO}
	&=
	H(\tilde{\bm{V}_{j}}^n(x_{j+1/2})),&
	\bU_{j-1/2,+}^{n,HO}
	&=
	H(\tilde{\bm{V}_{j-1}}^n(x_{j-1/2}))
	,
\end{align}
that we use to evaluate the FSLP flux function \eqref{RS} at each interface by setting:
\begin{equation}
	\bU_j^{n+1} - \bU_j^{n}
	+
	\frac{\Delta t}{\Dx}
	\qty(
	\mathbf{F}^{\text{FSLP}}(\bU_{j+1/2,-}^{n,HO},\bU_{j+1/2,+}^{n,HO})
	-
\mathbf{F}^{\text{FSLP}}(\bU_{j-1/2,-}^{n,HO},\bU_{j-1/2,+}^{n,HO})
	)
	=
	\Delta t
	\sourceterm_j(\bU_{j+1/2,-}^{n,HO},\bU_{j+1/2,+}^{n,HO},\bU_{j-1/2,-}^{n,HO},\bU_{j-1/2,+}^{n,HO}).
	\label{eq:O2}
	\end{equation}
	The gravity source term can also be computed with the same formula as in the first-order method by replacing cell-averaged values with the high-precision face-centered values:
	\begin{equation*}
		\sourceterm_j(\bU_{j+1/2,-}^{n,HO},\bU_{j+1/2,+}^{n,HO},\bU_{j-1/2,-}^{n,HO},\bU_{j-1/2,+}^{n,HO})=\begin{pmatrix}
		0\\ \{\rho \dx\phi\}_j^{n,H0} \\ \{\rho u \dx\phi\}_j^{n,HO}
		\end{pmatrix}
	\end{equation*}
	\begin{equation}
	\left\{\begin{array}{l}
		\{\rho \dx\phi\}_j^{n,H0}= \frac{(\rho \dx\phi)^{H0}_{j+1/2}+(\rho  \dx\phi)^{H0}_{j-1/2}}{2}\\
			\{\rho u \dx\phi\}_j^{n,H0}= \frac{u^{*,H0}_{j+1/2}(\rho  \dx\phi)^{H0}_{j+1/2}+u^{*,H0}_{j-1/2}( \rho\dx\phi)^{H0}_{j-1/2}}{2}\\
		(\rho \dx\phi)_{j+1/2}^{HO} = \frac{\rho_{j+1/2,-}^{n,H0}+\rho_{j+1/2,-}^{n,H0}}{2}(\dx \phi)_{j+1/2}^{HO}
		\end{array}\right.
	\end{equation}
	where $(\dx \phi)_{j+1/2}^{HO}$ is a second-order accurate evaluation of the derivative of the 
	gravitational potential at the interface $x_{j+1/2}$. Note that if the potential is known 
	explicitly, it can be computed exactly at the interface's coordinates \hlb{$(\dx \phi)_{j+1/2}^{HO}=\dx\phi(x_{j+\half})$. In the numerical results presented in section \ref{numerical exp}, we restrict ourselves to a simple linear gravitational potential field $\dx \phi=0, \dy \phi=g$.}
	The extension of the well-balanced property is not straightforward and beyond the scope of this paper.  \hlo{The difficulty lies in predicting the exact amount of diffusion required to be added/removed to precisely cancel out the pressure gradients, as the high-order reconstruction processes are non-linear.} Second-order well-balanced methods can be found in \hlo{\cite{thomann2020all, chalons2022exploring, caballero2023implicit, Luna2020,Grosso2021}}.
	The second-order extension \eqref{eq:O2} of the FSLP scheme is positive for density and internal energy as long as it is ensured that:
	\begin{equation}
	\frac{\Delta t}{\Delta x}  \max _{j \in \mathbb{Z}}\left(2\max \left(1/\rho_j^n, 1/\rho_{j+1}^n\right) a_{j+1 / 2} + \left(u_{j-\frac{1}{2}}^*\right)^{+}-\left(u_{j+\frac{1}{2}}^*\right)^{-} \right)< \frac{1}{2}
		\label{CFL unsplit 2}
	\end{equation}
The stability of the second-order method under the conditions above is a direct consequence of the stability of the first-order method.
For the second-order extension in time, one can use either the
 SSP-RK2 method \cite{spiteri2002new, gottlieb1998total}
  or a classical Hancock update \cite{toro}. 
\hlo{The latter option is tested numerically in section 
  \ref{sec:isentropic_vortex} where we check
   the 2nd order of accuracy of the FSLP-MUSCL-Hancock method on the isentropic vortex test case \cite{shu1998essentially}.}
\subsection{Multidimensional extension}
Before going any further, let us introduce the notations for our 2D space  discretization:
we consider two strictly increasing sequences $(x_{i+1/2})_{i\in\mathbb{Z}}$ and $(y_{j+1/2})_{j\in\mathbb{Z}}$ and divide the
real plane into cells where the ${ij}^\text{th}$ cell is the interval
$\left(x_{i-1 / 2}, x_{i+1 / 2}\right)\times\left(y_{j-1 / 2}, x_{j+1 / 2}\right)$. The space steps of the ${ij}^\text{th}$ cell are
$\Dx = x_{i+1/2} - x_{i-1/2}>0$ and $\Delta y_j = y_{j+1/2} - y_{j-1/2}>0$. We consider a discrete initial data $\bU_{ij}^{0}$ defined by $\bU_{ij}^{0}=\frac{1}{\Dx_i \Delta y_j} \int_{x_{i-1 / 2}}^{x_{i+1 / 2}} \int_{y_{j-1 / 2}}^{y_{j+1 / 2}}\bU^{0}(x,y) \mathrm{d} x\mathrm{d} y$, for $(i,j) \in \mathbb{Z}^2$.
Let us introduce the Euler equations of gas dynamicss in two dimensions of space:
\begin{equation}
		\dt \bU + \dx \bF ( \bU ) + \dy \bG ( \bU )  = \sourceterm (\bU,\phi)
		,\qquad\text{for $(x,y)\in\RealSet^2$, $t>0$,}
		\label{Euler's eqn 2D}
\end{equation}

with
\(\bU=(\rho, \rho u, \rho v, \rho E)^T\),
\(\bF ( \bU ) = (\mom,\ u\mom+ p,u\rho w,\ u\rho E + p u)^T\),
\(\bG ( \bU ) = (\rho v,\ v\mom,v\rho w+p,\ v\rho E + p v)^T\),
and
\(\sourceterm (\bU,\phi ) = -\rho \dx \phi (0,1,0,u)^T -\rho \dy \phi (0,0,1,v)^T\) where $v$ is the velocity in the $y$ direction and $\phi$ is smooth enough so that we can consider that $\partial_x \phi$, $\partial_y \phi$ are also regular and bounded. We take advantage of the rotational invariance of the 2D Euler system and discretize the fluxes direction by direction:

\begin{equation}
\left\{\begin{aligned}
\rho_{i,j}^{n+1} &=\rho_{i,j}^n\hlo{-}\frac{\Delta t}{\Delta x}\left(u_{i+1 / 2,j}^* \rho_{i+1 / 2,j}^{n}-u_{i-1 / 2,j}^* \rho_{i-1 / 2,j}^{n}\right) +\frac{\Delta t}{\Delta y}\left(v_{i,j+1 / 2}^* \rho_{i,j+1 / 2}^{n}-v_{i-1 / 2,j}^* \rho_{i-1 / 2,j}^{n}\right),   \\
(\rho u)_{i,j}^{n+1} &=(\rho u)_{i,j}^n\hlo{-}\frac{\Delta t}{\Delta x}\left(u_{i+1 / 2,j}^*(\rho u)_{i+1 / 2,j}^{n}+\Pi_{i+1 / 2,j}^{\theta^x,*}-u_{i-1 / 2,j}^*(\rho u)_{i-1 / 2,j}^{n}-\Pi_{i-1 / 2,j}^{\theta^x,*}\right)\\
&+\frac{\Delta t}{\Delta y}\left(v_{i,j+1 / 2}^*(\rho u)_{i,j+1 / 2}^{n}-v_{i,j-1 / 2}^*(\rho u)_{i,j-1 / 2}^{n}\right)-\{\rho \dx \phi\}^n_{ij}, \\
(\rho v)_{i,j}^{n+1} &=(\rho v)_{i,j}^n\hlo{-}\frac{\Delta t}{\Delta x}\left(u_{i+1 / 2,j}^*(\rho v)_{i+1 / 2,j}^{n}-u_{i-1 / 2,j}^*(\rho v)_{i-1 / 2,j}^{n}\right)\\
&+\frac{\Delta t}{\Delta y}\left(v_{i,j+1 / 2}^*(\rho v)_{i,j+1 / 2}^{n}+\Pi_{i,j+1 / 2}^{\theta^y,*}-v_{i,j-1 / 2}^*(\rho v)_{i,j-1 / 2}^{n}-\Pi_{i,j-1 / 2}^{\theta^y,*}\right)-\{\rho \dy \phi\}^n_{ij}, \\
(\rho E)_{i,j}^{n+1} &=(\rho E)_{i,j}^n\hlo{-}\frac{\Delta t}{\Delta x}\left(u_{i+1 / 2,j}^*(\rho E)_{i+1 / 2,j}^{n}+\Pi_{i+1 / 2,j}^{\theta^x,*} u_{i+1 / 2,j}^*-u_{i-1 / 2,j}^*(\rho E)_{i-1 / 2,j}^{n}-\Pi_{i-1 / 2,j}^{\theta^x,*} u_{i-1 / 2,j}^*\right) \\
&+\frac{\Delta t}{\Delta y}\left(v_{i,j+1 / 2}^*(\rho E)_{i,j+1 / 2}^{n}+\Pi_{i,j+1 / 2}^{\theta^y,*} v_{i,j+1 / 2}^*-v_{i,j-1 / 2}^*(\rho E)_{i,j-1 / 2}^{n}-\Pi_{i,j-1 / 2}^{\theta^y,*} v_{i,j-1 / 2}^*\right)-\{(\rho u\dx+\rho v\dy) \phi\}^n_{ij}.
\end{aligned}\right.
\label{unsplit method 2D}
\end{equation}
with
\begin{equation}
\left\{\begin{array}{l}
u^*_{i+1/2,j}=\frac{u^n_{i+1,j}+u^n_{i,j}}{2}-\frac{1}{2 a_{i+1/2,j}}\left(p^n_{i+1,j}-p^n_{i,j} + \frac{\rho^n_{i+1,j} + \rho^n_{i,j}}{2} (\phi_{i+1,j} - \phi_{i,j})\right), \\
v^*_{i+1/2,j}=\frac{v^n_{i,j+1}+v^n_{i,j}}{2}-\frac{1}{2 a_{i,j+1/2}}\left(p^n_{i,j+1}-p^n_{i,j} + \frac{\rho^n_{i,j+1} + \rho^n_{i,j}}{2} (\phi_{i,j+1} - \phi_{i,j})\right), \\
\Pi^{*,\theta^x}_{i+1/2,j}=\frac{p^n_{i+1,j}+p^n_{i,j}}{2}-\theta_{i+1/2,j}^x\frac{a_{i+1/2,j}}{2}\left(u^n_{i+1,j}-u^n_{i,j}\right),\\
\Pi^{*,\theta^y}_{i,j+1/2}=\frac{p^n_{i,j+1}+p^n_{i,j}}{2}-\theta_{i,j+1/2}^y\frac{a_{i,j+1/2,}}{2}\left(u^n_{i,j+1}-u^n_{i,j}\right),
\end{array}\right.
\label{u*p*2D}
\end{equation}

as well as the source terms discretization:
\begin{equation}
\left\{\begin{array}{l}
	\{\rho \dx\phi\}_{i,j}^n= \frac{(\rho \dx\phi)_{i+1/2,j}+(\rho \dx\phi)_{i-1/2,j}}{2},\\
	\{\rho \dy\phi\}_{i,j}^n= \frac{(\rho \dy\phi)_{i,j+1/2}+(\rho \dy\phi)_{i,j-1/2}}{2},\\
		\{\rho u \dx\phi\}_{i,j}^n= \frac{u^*_{i+1/2,j}(\rho \dx\phi)_{i+1/2,j}+u^*_{i-1/2,j}(\rho \dx\phi)_{i-1/2,j}}{2},\\
			\{\rho u \dy\phi\}_{i,j}^n= \frac{v^*_{i,j+1/2}(\rho \dy\phi)_{i,j+1/2}+v^*_{i,j-1/2}(\rho \dy\phi)_{i,j-1/2}}{2},\\
	(\rho \dx\phi)_{i+1/2,j} = \frac{\rho^n_{j+1}+\rho^n_{i,j}}{2}\frac{\phi_{i+i,j}-\phi_{i,j}}{\Dx},\\
(\rho \dy\phi)_{i,j+1/2} = \frac{\rho^n_{i,j+1}+\rho^n_{i,j}}{2}\frac{\phi_{i,j+i}-\phi_{i,j}}{\Delta y}.
	\end{array}\right.
\end{equation}

\section{Flux-splitting as a relaxation approximation}
The goal of this section is to highlight the connection between the FSLP flux-splitting approach and a relaxation approximation. In the previous sections, we concluded that the FSLP approach could be expressed as an averaging procedure \eqref{average} where
 \(\bU_j^P\) and \(\bU_j^A\) are defined as approximate solutions of two systems \eqref{pressure system} and \eqref{advection system} that respectively only account for the pressure and the advection effects. 
We propose to translate that three-step process thanks to a relaxation approximation. Suppose that $\alpha\in(0,1)$ is a constant and let $\relaxparam$ be a positive parameter, we consider the system
\DeclareRobustCommand\subscriptrelax{\relaxparam}
\begin{subequations}
	\renewcommand{\theequation}{\arabic{parentequation}\alph{equation}\textsubscript{\(\subscriptrelax\)}}
	\begin{align}
		\dt
		\begin{bmatrix}
			\rho^P
			\\
			\rho^P u^P
			\\
			\rho^P E^P
			\\
			\rho^P\Pi^P
			\\
			\rho^P\Tau^P
			\\
			\phi
		\end{bmatrix}
		&&
		+ 
		\frac{1}{\alpha}
		\dx
		\begin{bmatrix}
			0
			\\
			\Pi^P
			\\
			 \Pi^P u^P
			\\
			 a^2 u^P
			\\
			- u^P
			\\
			0
		\end{bmatrix}
		+
		\frac{1}{\alpha}
		\begin{bmatrix}
			0
			\\
			\rho^P
			\\
			\rho^Pu^P
			\\
			0
			\\
			0
			\\
			0
		\end{bmatrix}
	\dx \phi
		&=
		 {\relaxparam}
		\begin{bmatrix}
			\alpha\rho^P+(1-\alpha)\rho^A - \rho^P
			\\
			\alpha\rho^P u^P+(1-\alpha)\rho^A u^A - \rho^Pu^P
			\\
			\alpha\rho^P E^P+(1-\alpha)\rho^A E^A - \rho^P E^P
			\\
			p^\EOS(1/\rho^P,e^P) - \Pi^P
			\\
			1 - \rho^P\Tau^P
			\\
			0
		\end{bmatrix},
		\label{eq: relaxed system param pressure}
		\\
		\dt
		\begin{bmatrix}
			\rho^A
			\\
			\rho^A u^A
			\\
			\rho^A E^A
			\\
			\rho^A\Pi^A
			\\
			\rho^A\Tau^A
		\end{bmatrix}
		&&
		+
		\qty(\frac{1}{1-\alpha})
		\dx
		\begin{bmatrix}
			 \rho^A u^P
			\\
			 \rho^A u^A u^P
			\\
			 \rho^A E^A u^P
			\\
			 \rho^A\Pi^A u^P
			\\
			 u^P
		\end{bmatrix}
		\phantom{	-
			\begin{bmatrix}
				2\sourceterm(\bU^P,\phi)
				\\
				0
			\end{bmatrix}
		}
		&= {\relaxparam}
		\begin{bmatrix}
			\alpha\rho^P+(1-\alpha)\rho^A - \rho^A
			\\
			\alpha\rho^P u^P+(1-\alpha)\rho^A u^A - \rho^Au^A
			\\
			\alpha\rho^P E^P+(1-\alpha)\rho^A E^A - \rho^A E^A
			\\
			p^\EOS(1/\rho^A,e^A) - \Pi^A
			\\
			1 - \rho^A\Tau^A
		\end{bmatrix}.
		\label{eq: relaxed system param advection}
	\end{align}
	\label{eq: relaxed system param}
\end{subequations}
%
The system~(\ref{eq: relaxed system param}\textsubscript{\(\relaxparam\)}) features
a pair of duplicate conservative variables $(\bU^P,\bU^A)$
and 4 other variables: $\Pi^P$, $\Pi^A$, $\Tau^A$ and $\Tau^P$.
The variables $\Pi^P$ and $\Pi^A$ are surrogate for the thermodynamical pressure, while
$\Tau^A$ and $\Tau^P$ play the role of a pseudo-specific volume.
It is possible to view (\ref{eq: relaxed system param}\textsubscript{\(\relaxparam\)})  as
a Suliciu relaxation approximation with a separation of the acoustic and transport operators. Indeed,  (\ref{eq: relaxed system param}\textsubscript{\(\relaxparam\)}) implies that
\begin{subequations}
	\begin{align}
		\dt
		\begin{bmatrix}
			\alpha\rho^P + (1-\alpha)\rho^A
			\\
			\alpha\rho^P u^P + (1-\alpha)\rho^A u^A
			\\
			\alpha\rho^P E^P+ (1-\alpha)\rho^A E^A
		\end{bmatrix}
		+
		\dx
		\begin{bmatrix}
		\rho^A u^P
		\\
		\rho^A u^A u^P + \Pi^P
		\\
		\rho^A E^A u^P + \Pi^P u^P
	\end{bmatrix}
		&=
		\begin{bmatrix}
			0
			\\
			-\rho^P
			\\
			-\rho^P u^P
		\end{bmatrix}
	\dx \phi
		\label{eq: relaxed rhoU^P + rhoU^A eq}
		\\
		\dt
		\qty[\alpha\rho^P\Pi^P
		\!+\!
		(1-\alpha)\rho^A\Pi^A]
		+
		\dx
		\qty
		(
		\rho^A\Pi^Au^P
		\!+\!
		a^2 u^P
		)
		&=
		\relaxparam
		\qty
		[
		\alpha p^\EOS\qty(\frac{1}{\rho^P},e^P\!) 
		\!+\!
		(1\!-\!\alpha) p^\EOS\qty(\frac{1}{\rho^A},e^A\!) 
		\!-\!
		\alpha \Pi^P \! - \! (1\!-\!\alpha)\Pi^A
		\!
		]
		,
		\label{eq: relaxed rhoPi^P + rhoPi^A eq}
		\\
		\dt
		\qty[ \alpha\rho^P\Tau^P + (1-\alpha)\rho^A\Tau^A]
		&=
		\relaxparam
		\qty
		[
		1-\alpha\rho^P\Tau^P-(1-\alpha)\rho^A\Tau^A
		]
		\label{eq: relaxed rhoTau^P + rhoTau^A eq}
		.
	\end{align}
\end{subequations}%

Taking the limit $\relaxparam\to+\infty$ formally enforces that
$\bU^P=\bU^A=\bU=(\rho,\rho u,\rho E)^T$ and $\Pi^A = \Pi^P = p^\EOS(1/\rho,e)$, so that \eqref{eq: relaxed rhoU^P + rhoU^A eq}
enables to retrieve the Euler system~\eqref{Euler's eqn}.
This suggests that we  can use the relaxation system~(\ref{eq: relaxed system param}\textsubscript{\(\relaxparam\)}) as an approximation of \eqref{Euler's eqn} in the limit $\relaxparam\to+\infty$.
The equation~\eqref{eq: relaxed rhoPi^P + rhoPi^A eq} plays here
a similar role as the surrogate pressure equation in the classic Suliciu approximation~\cite{Suliciu1998,chalons2008,coquel2012}.
The sole purpose of equation~\eqref{eq: relaxed rhoTau^P + rhoTau^A eq}
is to ensure that $\alpha\rho^P\Tau^P+(1-\alpha)\rho^A\Tau^A=1$ in the regime
$\relaxparam\to\infty$.
In our discretization strategy, we classically mimic the $\relaxparam\rightarrow \infty$ regime for $t\in[t^{n},t^{n+1})$,
by enforcing
$(\bU^P,\Pi^P,\Tau^A,\bU^A,\Pi^A,\Tau^A)(t=t^n) = (\bU,p^{\text{EOS}}(1/\rho,e),1/\rho,\bU,p^{\text{EOS}}(1/\rho,e),1/\rho)(t=t^n)$
 and by solving the relaxation off-equilibrium
 system~(\ref{eq: relaxed system param}\textsubscript{\(\relaxparam\!=\!0\)}).
The properties of the off-equilibrium system~(\ref{eq: relaxed system param}\textsubscript{\(\relaxparam\!=\!0\)}) are briefly summarized in the following proposition whose proof is given in \ref{section: eigenstructure analysis}.
\begin{proposition} \leavevmode
The system~(\ref{eq: relaxed system param}\textsubscript{\(\relaxparam\!=\!0\)}) is hyperbolic with a set of characteristic velocities given by:
	$\dfrac{u^P}{1-\alpha}$ (with an algebraic multiplicity 4), $0$ (with an algebraic multiplicity 5) and 
	$\pm \dfrac{a}{\alpha\rho^P}$. 
	Moreover, (\ref{eq: relaxed system param}\textsubscript{\(\relaxparam\!=\!0\)}) only involves linearly degenerate fields.
\end{proposition}
The relaxation formulation~(\ref{eq: relaxed system param}\textsubscript{\(\relaxparam\)}) sheds some more light on the similarities between the flux-splitting we propose here and the acoustic/transport operator  splitting strategy presented in \cite{padioleau2019high}. Indeed, the source term and pressure effects can be treated separately from the advection terms. The difference is that although the operators are separated, they are re-distributed within a larger single system instead of two separate systems.

By discretizing the pressure and advection parts of \eqref{eq: flux-splitting discretization} identically than in section \ref{sect:cvx}, we re-obtain the same update formula \eqref{unsplit method}, which yields the FSLP scheme~\eqref{RS}. Finally, let us mention that it is possible to
build an alternate flux-splitting method for the system~\eqref{Euler's eqn}
by seeking the solution of the Riemann problem for (\ref{eq: relaxed system param}\textsubscript{\(\relaxparam\!=\!0\)}). This option is not studied in the present work.

\section{Numerical experiments}\label{numerical exp}

In this section, we consider that the fluid is a perfect gas with the EOS $p=(\gamma-1)\rho e$ and that the potential $\phi$ takes the form $\phi(x,y) = - gy$ for tests that involve the source term.

We will present numerical experiments with the FSLP method and the HLLC Riemann solver \cite{toro} using first and second-order discretizations.
The second-order accuracy is achieved using a MUSCL-Hancock strategy \cite{toro} for both the HLLC and FSLP solvers. Let us mention that the slope reconstruction is performed on the primitive variables with a minmod slope limiter \cite{Godlewski-1990,LeVeque2002,toro}. For the OSLP method, noting $(a /\rho )_{j+1/2}=\max \left(1/\rho_j^n, 1/\rho_{j+1}^n\right) a_{j+1 / 2}$, the time steps $\Dt$ is computed as follows:

\begin{eqnarray}
\Dt = C^\text{CFL}\ \Dx \ \frac{1}{max _{j \in \mathbb{Z}}\left[\max \left\{ 2\max\left[(a /\rho )_{j-1/2},(a /\rho )_{j-1/2}\right],\left(\left(u_{j-\frac{1}{2}}^*\right)^{+}-\left(u_{j+\frac{1}{2}}^*\right)^{-}\right)\right\} \right] },
\label{dt computation OSLP}
\end{eqnarray}

For the FSLP method, it is computed as follows:

\begin{eqnarray}
\Dt = C^\text{CFL}\ \Dx \ \frac{1}{max _{j \in \mathbb{Z}}\left[2\max\left[(a /\rho )_{j-1/2},(a /\rho )_{j-1/2}\right]+\left(\left(u_{j-\frac{1}{2}}^*\right)^{+}-\left(u_{j+\frac{1}{2}}^*\right)^{-}\right)\right] }
,
\label{dt computation FSLP}
\end{eqnarray}
where the parameter $C^\text{CFL}$ is given by the table \ref{table: CFL} so that the CFL conditions~\eqref{acoustic CFL} and\eqref{transport CFL} for the OSLP method, \eqref{CFL unsplit} for the first-order FSLP method and \eqref{CFL unsplit 2} for the second-order FSLP method are all checked. For the HLLC solver, the standard CFL from \cite{toro1994restoration} is used.

\begin{table}[h]
	\centering
\begin{tabular}{ |c|c|c| }
\hline
Numerical scheme& first-order & second-order  \\
\hline
OSLP & 1.0 & N.A. \\
\hline
FSLP & 1.0 & 1/2 \\
\hline
HLLC & 1.0 & 1/2 \\
\hline
\end{tabular}
\caption{Values for $C^\text{CFL}$ used in the simulations.}
\label{table: CFL}
\end{table}

The parameter \(\theta\) related to the low Mach correction is defined
at each interface $(i+1/2,j)$ and $(i,j+1/2)$ by
\begin{align}
	\theta_{i+1/2,j}^x&=\text{max}\left(\mid u_{i,j}\mid /c_i,\mid u_{i+1,j}\mid/c_{i+1,j}\right),
	&
	\theta_{i,j+1/2}^y&=\text{max}\left(\mid v_{i,j}\mid /c_i,\mid v_{i,j+1}\mid/c_{i,j+1}\right).
	\label{choice theta}
\end{align}
\hlb{Note that our choice for the computation of $\theta$ differs from \cite{chalons_all-regime_2016} that uses the inerface velocity $u^*, v^*$.
Both choices give satisfactory results and are valid estimations of the local Mach number $Ma$.
Depending on the interface values of velocities and pressure, one choice can be more or less diffusive than the other.
However, no significant differences have been observed in our experiments. }

\subsection{Sod shock tube test case}\label{sect:sod}
We consider here the classical Sod shock tube test case \cite{Sod1978,toro}: we set $\gamma = 1.4$ and the initial conditions are:
\begin{equation*}
    (\rho, u, p)(x,t=0) =  \begin{cases}
        (1,\ 0,\ 1) &\text{if} \quad x < 0.5,\\
        (0.125,\ 0,\ 0.1)  &\text{if} \quad  x>  0.5.
        \\
    \end{cases}
\end{equation*}
The goal of this test is to study the ability of our solver to handle different wave types. The initial discontinuity generates three waves: a leftward going rarefaction, a contact discontinuity, and a shock that both travel towards the right of the computational domain.
\begin{figure}
\centering
\begin{tabular}{cc}
		\includegraphics[width=0.49\linewidth]{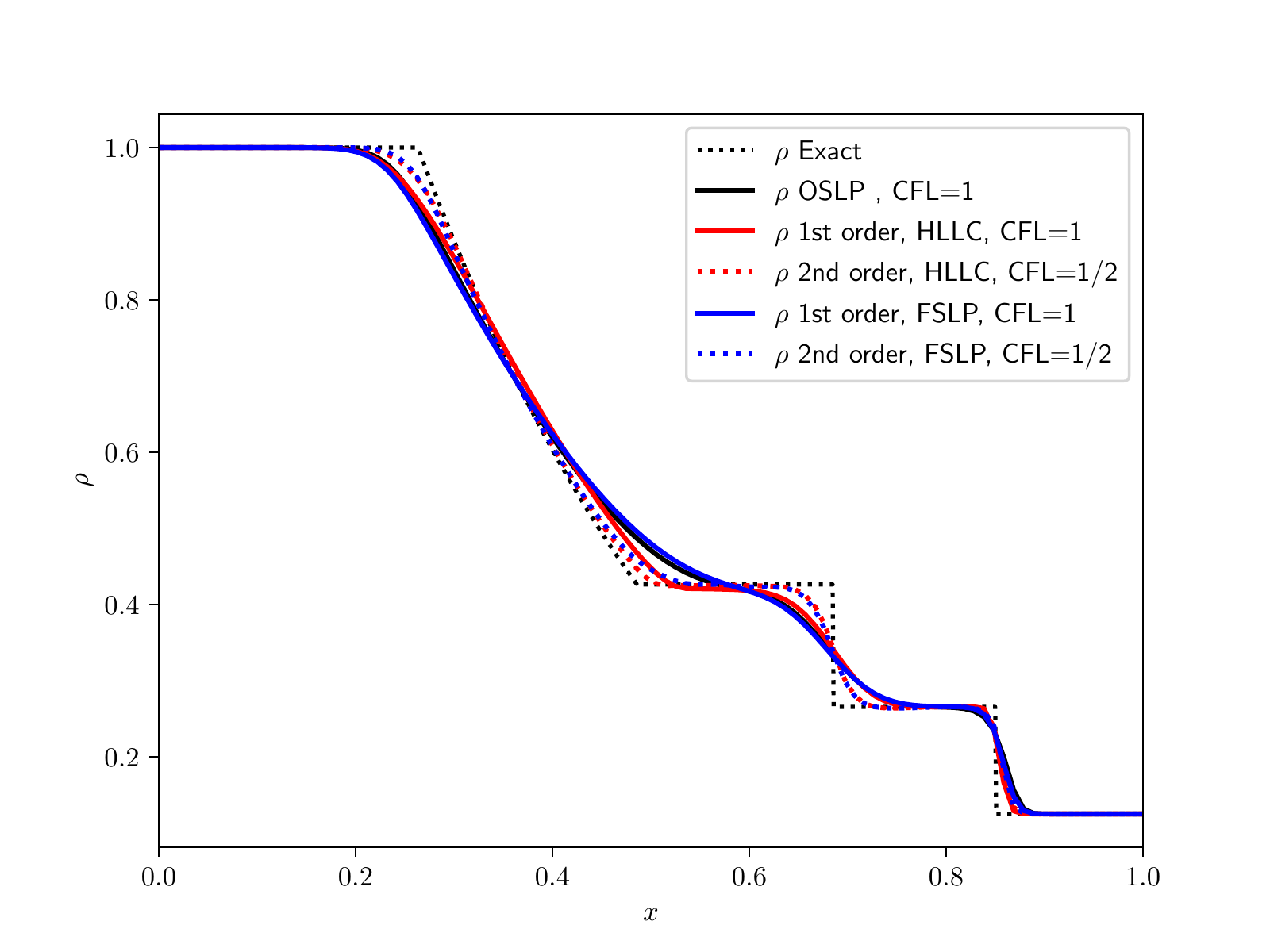}
	&
	\includegraphics[width=0.49\linewidth]{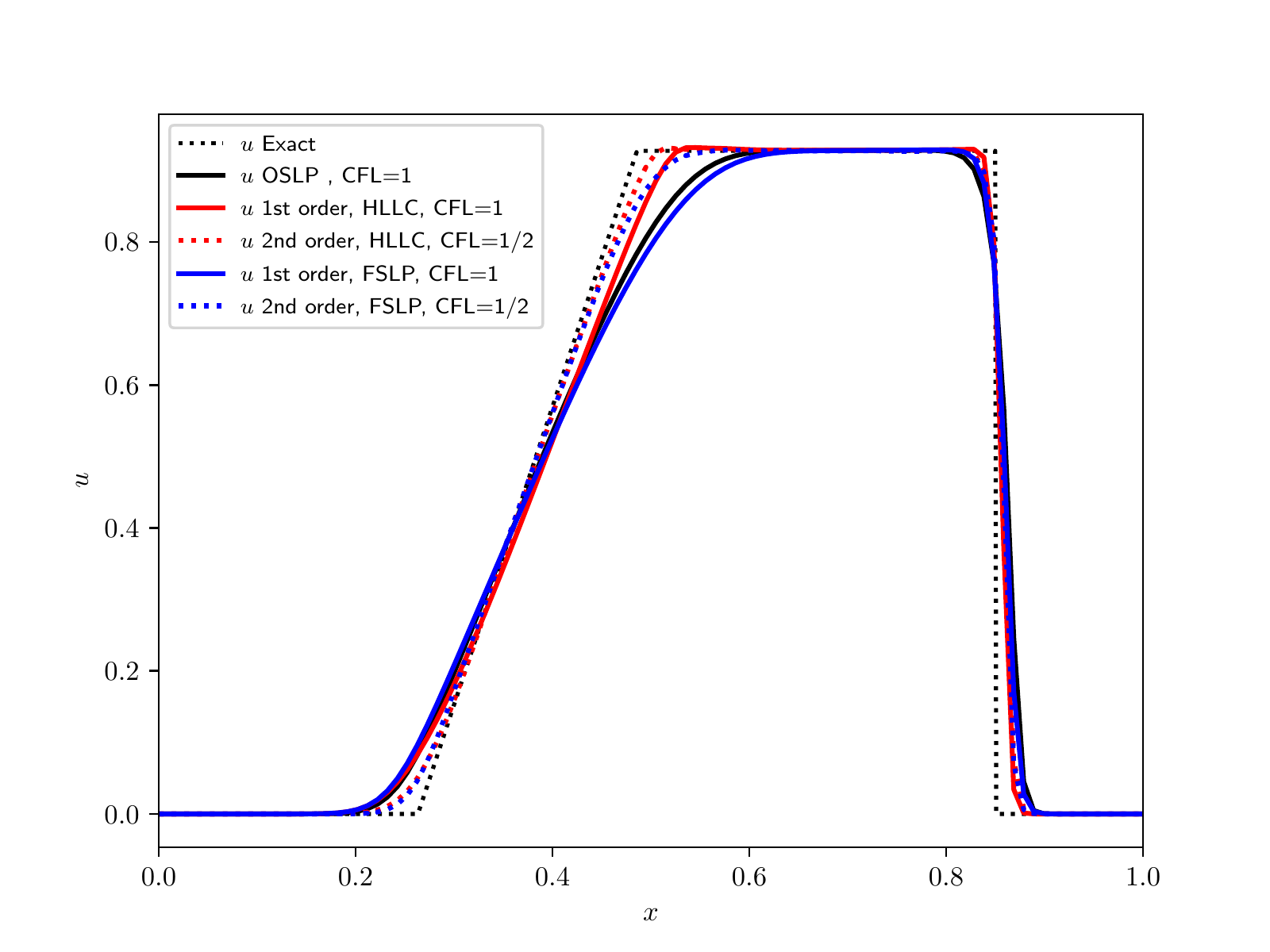}
	\\
	\includegraphics[width=0.49\linewidth]{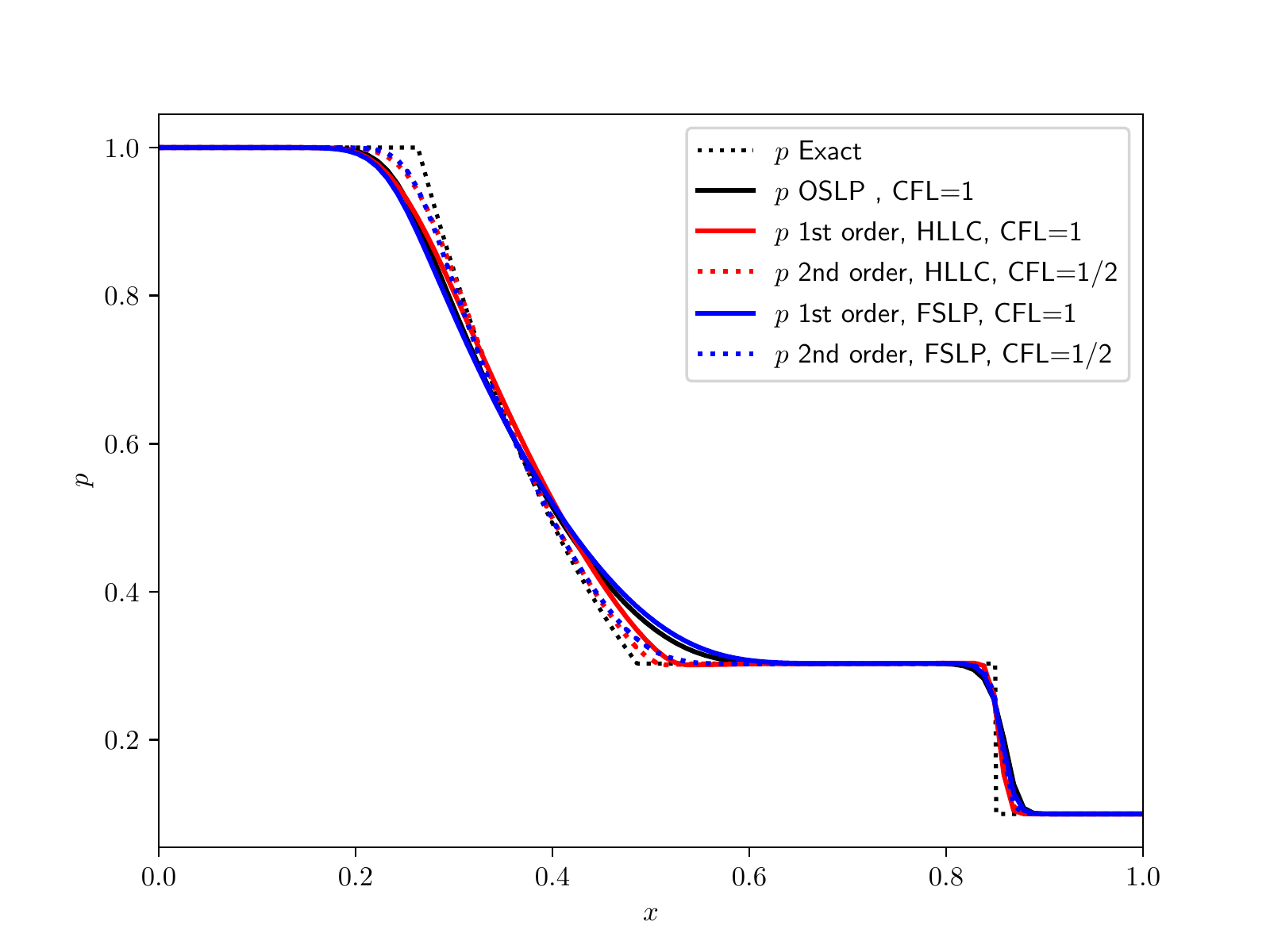}
	&
	\includegraphics[width=0.49\linewidth]{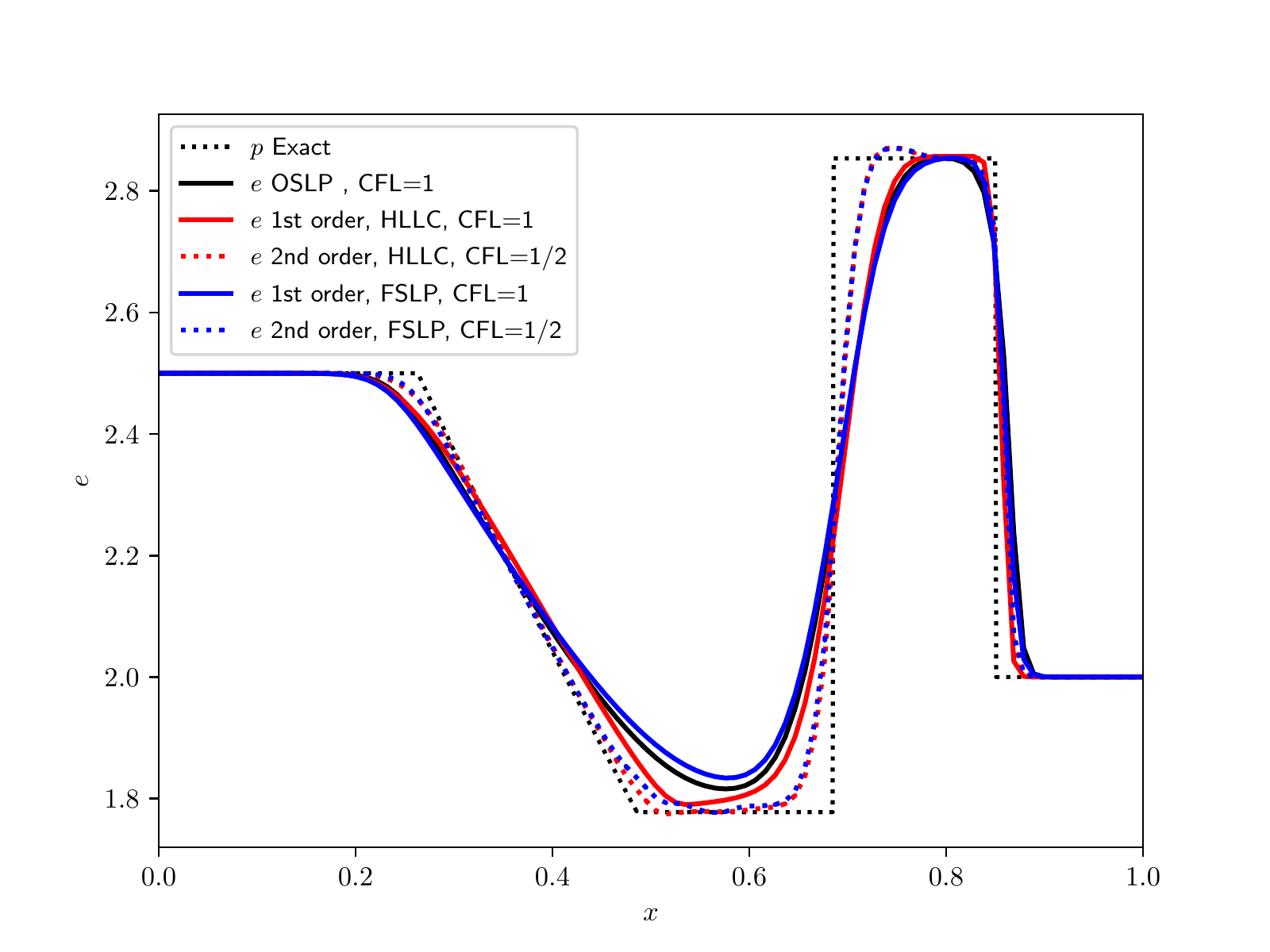}
\end{tabular}
        \caption{
        	Sod shock tube test case. Profile at $t=0.2s$ of the density (top left), velocity (top right), pressure (bottom left), and specific internal energy (bottom right).
        	The results are obtained with the OSLP method, first and second-order    FSLP method, first and second-order HLLC scheme, and the exact solution on a $100$-cell grid.
	}
    \label{sod}
\end{figure}
Figure~\ref{sod} shows the profile obtained at $t=0.2s$ with five different solvers: OSLP, FSLP/HLLC for the first and second-order methods. At first order, the HLLC solver provides the sharpest resolution of the shock and contact discontinuity. The differences between the FSLP and OSLP methods are hardly visible. None of the schemes suffers from spurious oscillations and
both the position and the amplitude of the waves match the exact solution. \hlo{We also note that the OSLP method 
is slightly sharper than the FSLP method on the rarefaction and contact discontinuity. In section \ref*{sec:isentropic_vortex}, we compare 
the accuracy of both method on the isentropic vortex test case.}

\subsection{Two-rarefaction test case}
We now consider the two-rarefaction test proposed by Einfeldt \cite{Einfeldt1991,toro}
for a perfect gas with $\gamma=1.4$. The initial conditions are
\begin{equation*}
	(\rho, u, p)(x,t=0) =  \begin{cases}
		(1,\ -2,\ 0.4), &\text{if $x < 0.5$,}\\
		(1,\ 2,\ 0.4),  &\text{if $x>  0.5$.}\\
	\end{cases}
\end{equation*}
The resulting wave pattern features two rarefaction waves that split from the position $x=0.5$, traveling towards each end of the computational domain.
As a result, a  near vacuum region presenting low densities and pressures  appears in the middle of the domain.

\begin{figure}
	\centering
\begin{tabular}{cc}
	\includegraphics[width=0.49\linewidth]{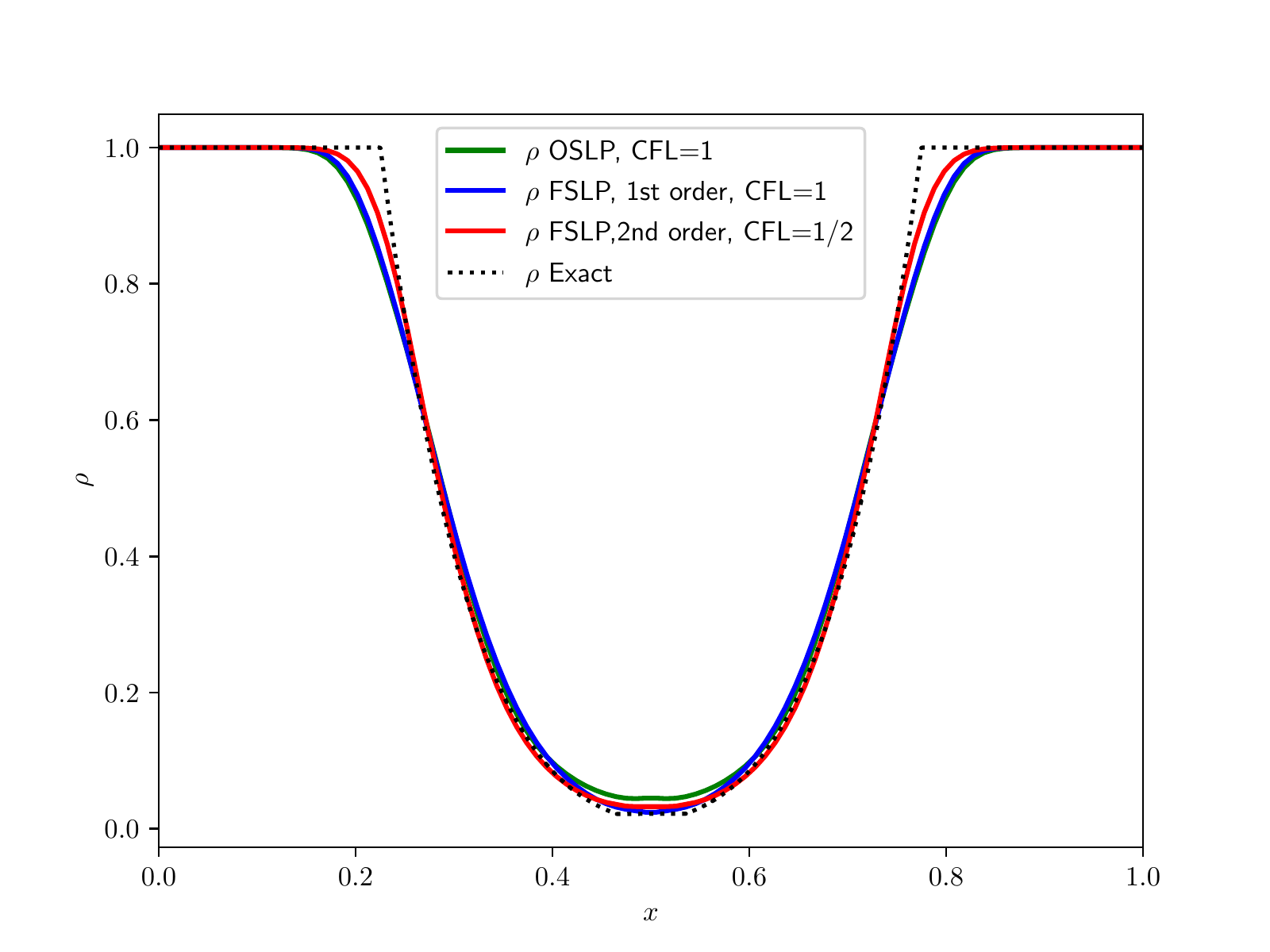}
	&
	\includegraphics[width=0.49\linewidth]{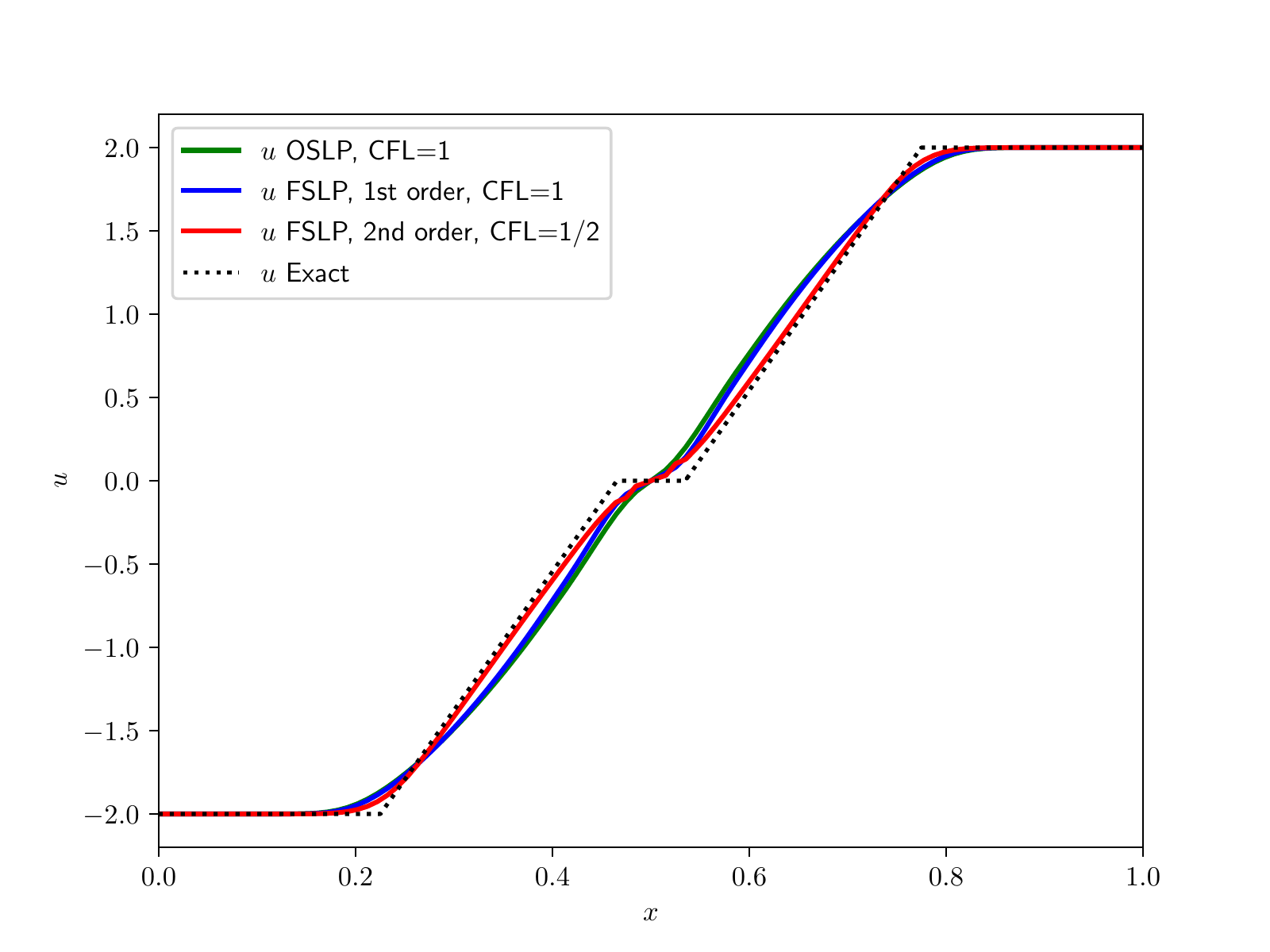}
	\\
	\includegraphics[width=0.49\linewidth]{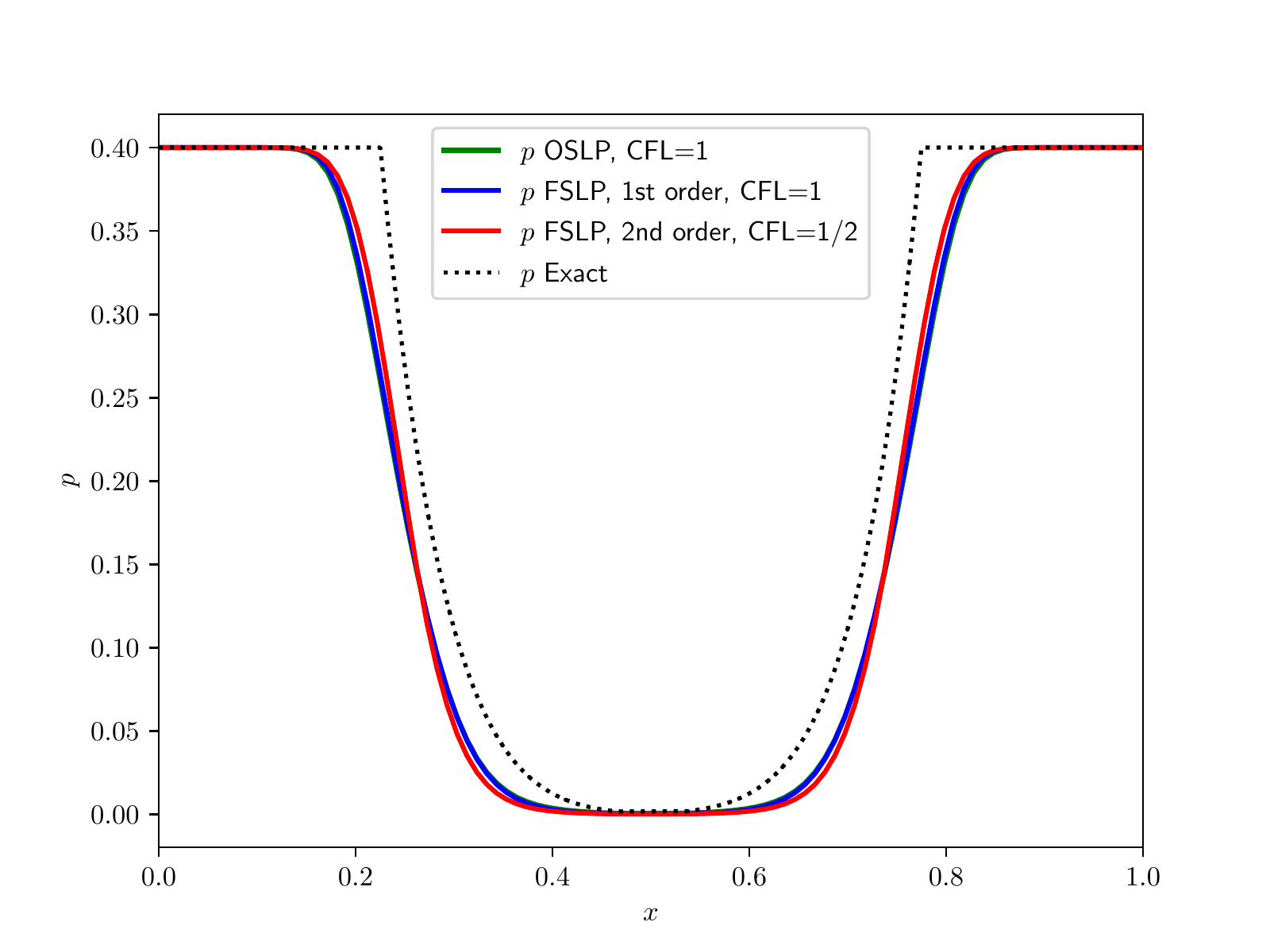}
	&
	\includegraphics[width=0.49\linewidth]{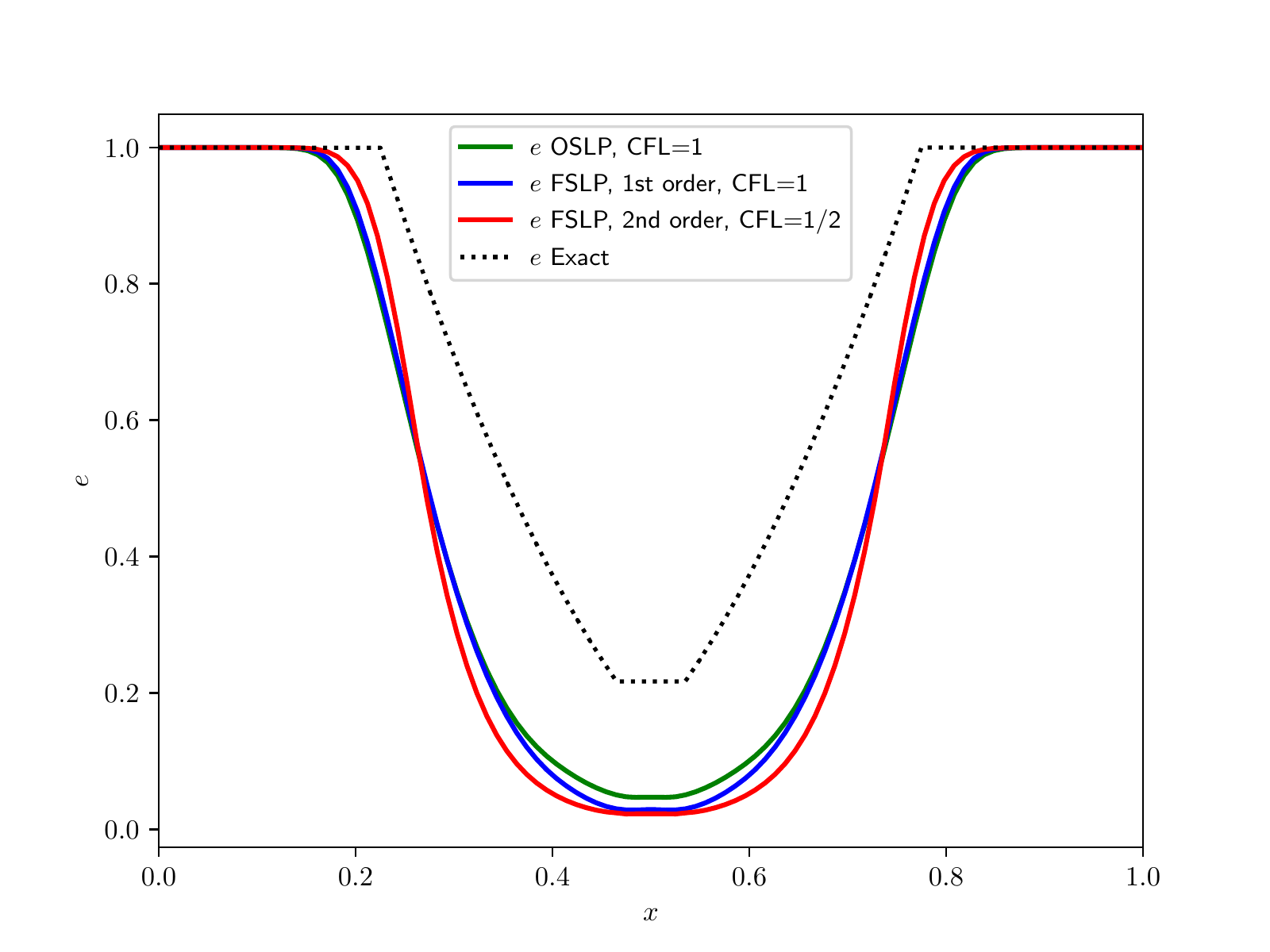}
\end{tabular}
\caption{
  Two-rarefaction test case.
  	Profile at $t=0.1s$ of the density (top left), velocity (top right), pressure (bottom left), and specific internal energy (bottom right).
The results are obtained with the OSLP method, the first and second-order FSLP method, and the exact solution on a $100$-cell grid.
}
 \label{raref}
\end{figure}

Figure~\ref{raref} shows that all methods are robust enough
to preserve positivity for mass, pressure, and energy so that they are able to reach the end of the simulation. Moreover, none of the numerical schemes exhibit entropy-related issues like the apparition of nonphysical shocks within the wave pattern.

\subsection{\hlo{Grid convergence -- The isentropic vortex test}}\label{sec:isentropic_vortex}
\hlo{The accuracy of our FSLP scheme equipped with a MUSCL-Hanckock strategy is considered on a classical 2D test problem
called the nonlinear isentropic vortex advection presented by Shu \cite{shu1998essentially}.
As in \cite{reyes2019variable}, we double the original domain size to avoid self-interactions
of the vortex across the periodic domain.
The test involves a circular region centered at $(x_c, y_c) = (10, 10)$ on a periodic square domain, 
$[0,20]\times [0,20]$, 
where a Gaussian-shaped vortex with a rotating velocity field is initialized.
The problem consists in advecting the vortex along the diagonal direction, therefore any departure from the initial condition (or the exact solution of the problem) will be considered numerical errors of the numerical method under consideration. The initial condition proposed in \cite{shu1998essentially} defines the values of the primitive variables at $t=0$ as follows}
\begin{subequations}
\begin{align}
\rho(x,y) &= \left[ 1 - \left(\gamma -1\right)\frac{\beta^2}{8\gamma\pi^2}e^{1-r^2}\right]^{\frac{1}{\gamma-1}},
\label{eq:vortex_IC_dens}\\
u(x,y) &= 1 - \frac{\beta}{2\pi}e^{\frac{1}{2}\left(1-r^2\right)}(y-y_c),\label{eq:vortex_IC_velx}\\
v(x,y) &= 1 + \frac{\beta}{2\pi}e^{\frac{1}{2}\left(1-r^2\right)}(x-x_c),\label{eq:vortex_IC_vely}\\
p(x,y) &= \rho(x,y)^\gamma,\label{eq:vortex_IC_pres}
\end{align}	
\end{subequations}
\hlo{with $r = r(x,y) = \sqrt{(x-x_c)^2 + (y-y_c)^2}$ and the vortex strength $\beta = 5.$
Due to the velocity field, $(u,v)=(1,1)$, the vortex is translated across the diagonal direction of the computational domain and returns to the initial position at $t = 20s$. The numerical error is then compared at this instant using the initial condition as the value of the exact solution.
We run 6 simulations corresponding to the resolutions 
$[Nx,Ny]=[N,N],\ N\in\{32,64,128,256,512,1024\}$ and display the $L^1$ and 
$L^\infty$ 
errors in figure \ref{fig:convergence}. 
The $L^1$  and $L^{\infty}$ errors are computed for the density as $\Delta x \Delta y\sum_{i,j}\mid\rho^n_{i,j}-\rho_0^{i,j}\mid$ and 
$\max_{i,j}\mid\rho^n_{i,j}-\rho_0^{i,j}\mid$ respectively.
One can see that convergence rate of the numerical method follows a second-order slope, validating our high-order extension.}

\begin{figure}
 \centering
 \includegraphics[width=0.5\linewidth]{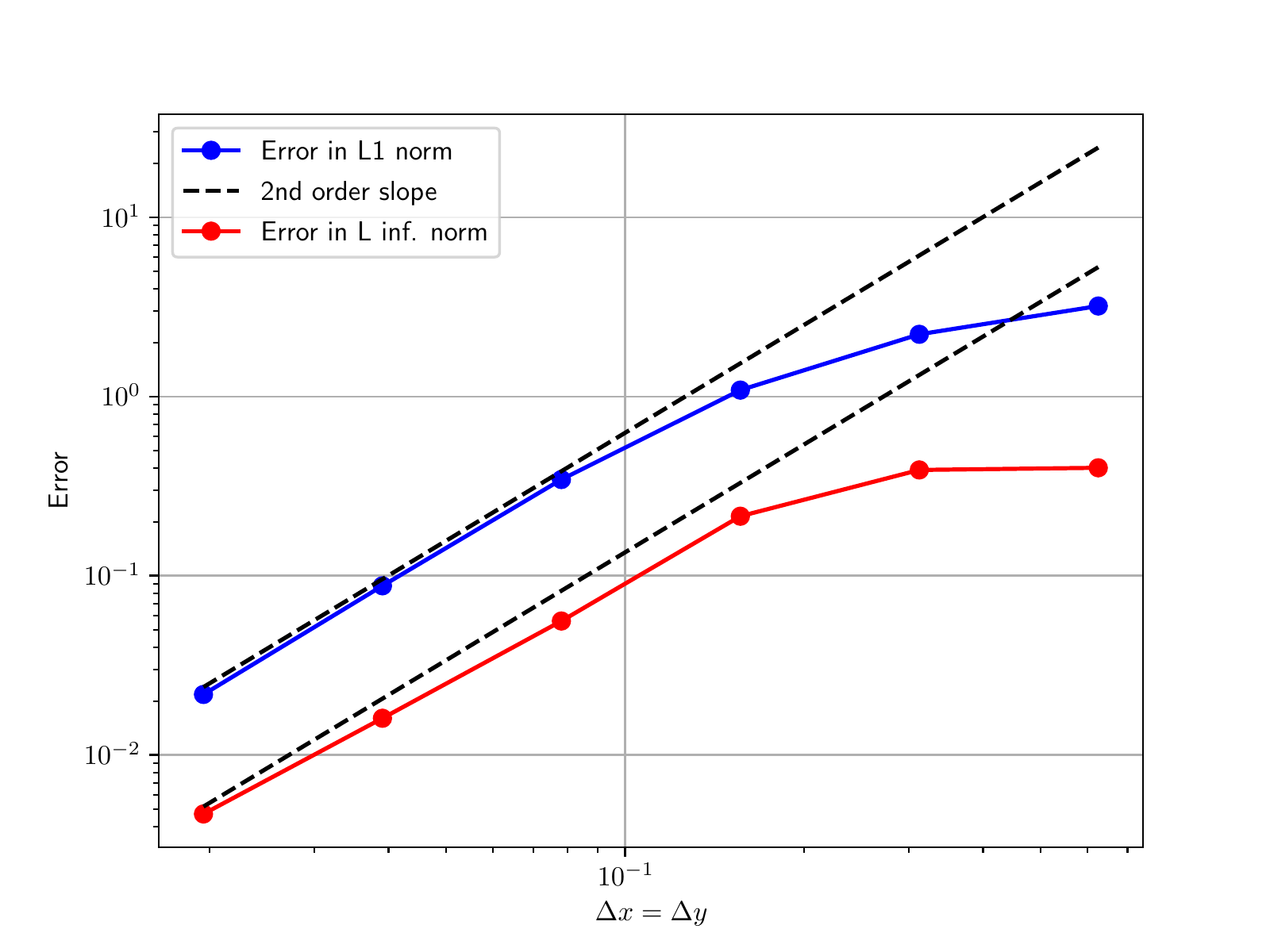}
 \caption{Convergence study of the FSLP method extended to second order via a MUSCL-Hancock strategy}
 \label{fig:convergence}
\end{figure}

\hlo{The isentropic vortex test cases also allows us to compare the accuracy of the OSLP and FSLP methods.
We ran two simulations on $512^2$ grids with both methods (at first order of accuracy). 
The $L^1$ error of the FSLP method is about $10\%$ higher than the OSLP method.
Note that this number may vary for different test cases and resolutions. In section~\ref*{sect:sod}, we also observed that the FSLP method is slightly less accurate on the Sod shock tube test case.}

\subsection{The Gresho Vortex}\label{section: gresho}
The Gresho vortex \cite{Gresho1990} involves a
a stationary vortex that can be parameterized by the maximum value of the Mach number $Ma$ across the computational domain. Therefore this test is very useful for studying the performance of numerical schemes in the low-Mach regime. We consider a perfect gas with $ \gamma = 1.4$. Using polar coordinates $(r, \theta)$, the initial conditions read:
\begin{subequations}
    \begin{align}
        \rho(r,\theta,t=0)&=1,
        \\
        \left(u_{r}, u_{\theta}\right)
        (r,\theta,t=0)
        &=
        \begin{cases}
            (0,5 r) & \text{if $0 \leq r<0.2$,} \\
            (0,2-5 r) & \text{if $0.2 \leq r<0.4$,} \\
            (0,0) & \text{if $0.4 \leq r$,}
        \end{cases}
        \\
        p(r,\theta,t=0)
        &=
        \begin{cases}
            p_{0}+12.5 r^{2}  & \text{if $0 \leq r<0.2$,} \\
            p_{0}+12.5 r^{2}+4-20 r+4 \ln (5 r)& \text{if $0.2 \leq r<0.4$,} \\
            p_{0}-2+4 \ln 2& \text{if $0.4 \leq r$,}
        \end{cases}
        \label{gresho}
    \end{align}
\end{subequations}
where $p_0 = \frac{1}{\gamma Ma^2}$.
For the simulations, we will use three different values for the reference Mach number:
$Ma\in\{10^{-1}, 10^{-3}, 10^{-5}\}$.
We will compare the distributions of the velocity magnitude obtained at $t=10^{-2}s$ with the
initial conditions.
\begin{figure}
    \centering
    \begin{tabular}{ccc}
        \includegraphics[width=0.3\linewidth]{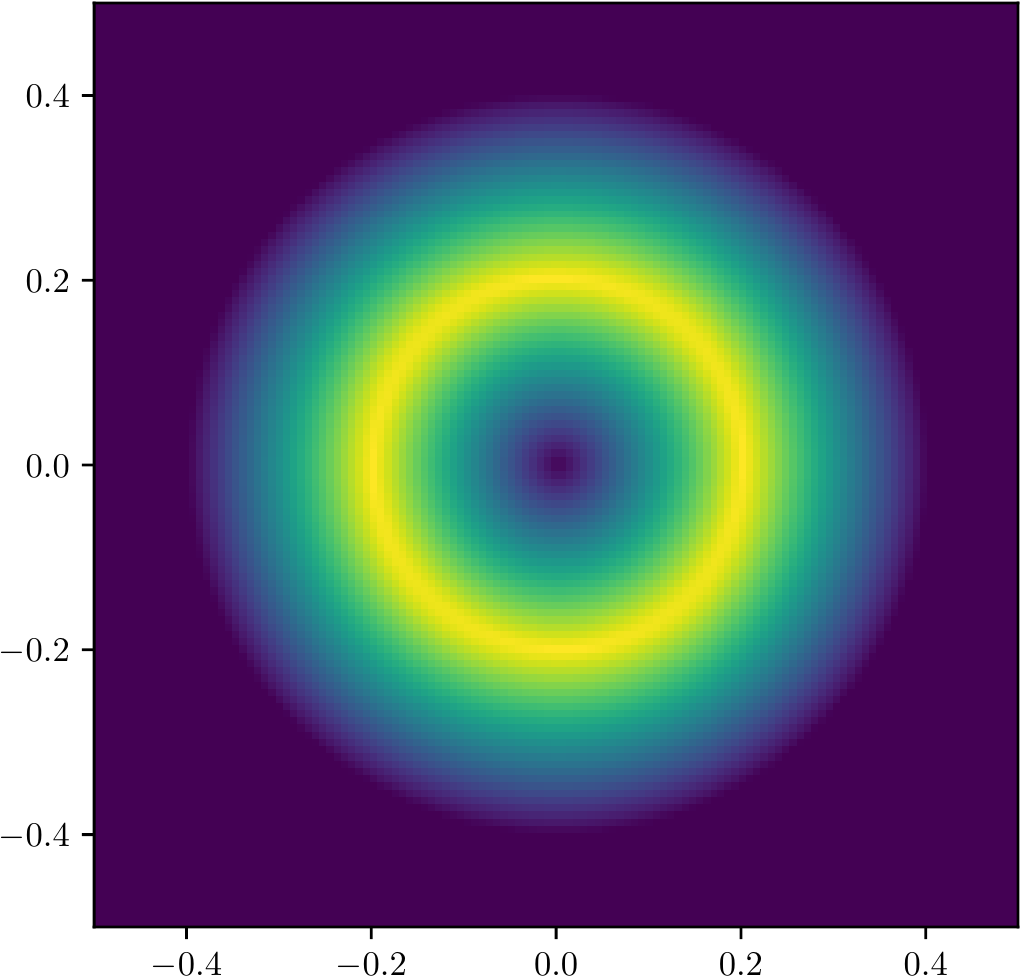}
        &
        \includegraphics[width=0.3\linewidth]{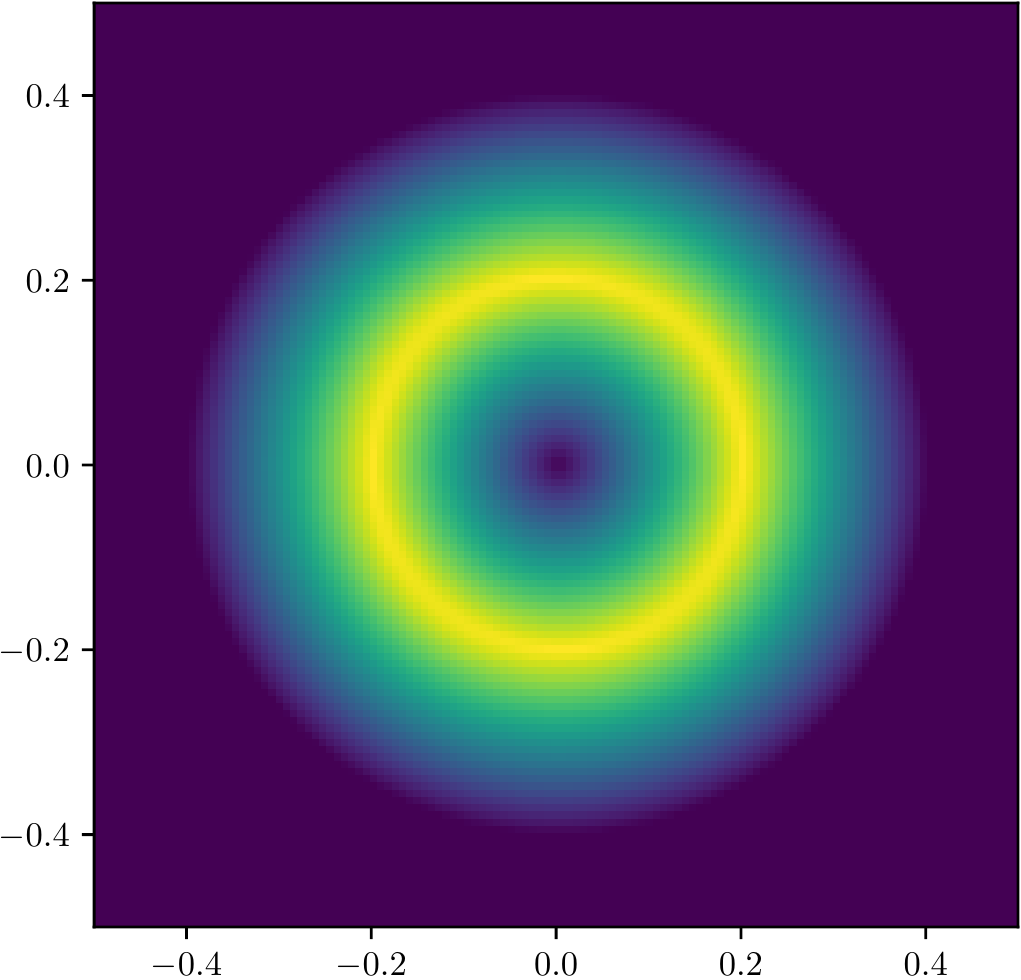}
        &
        \includegraphics[width=0.3\linewidth]{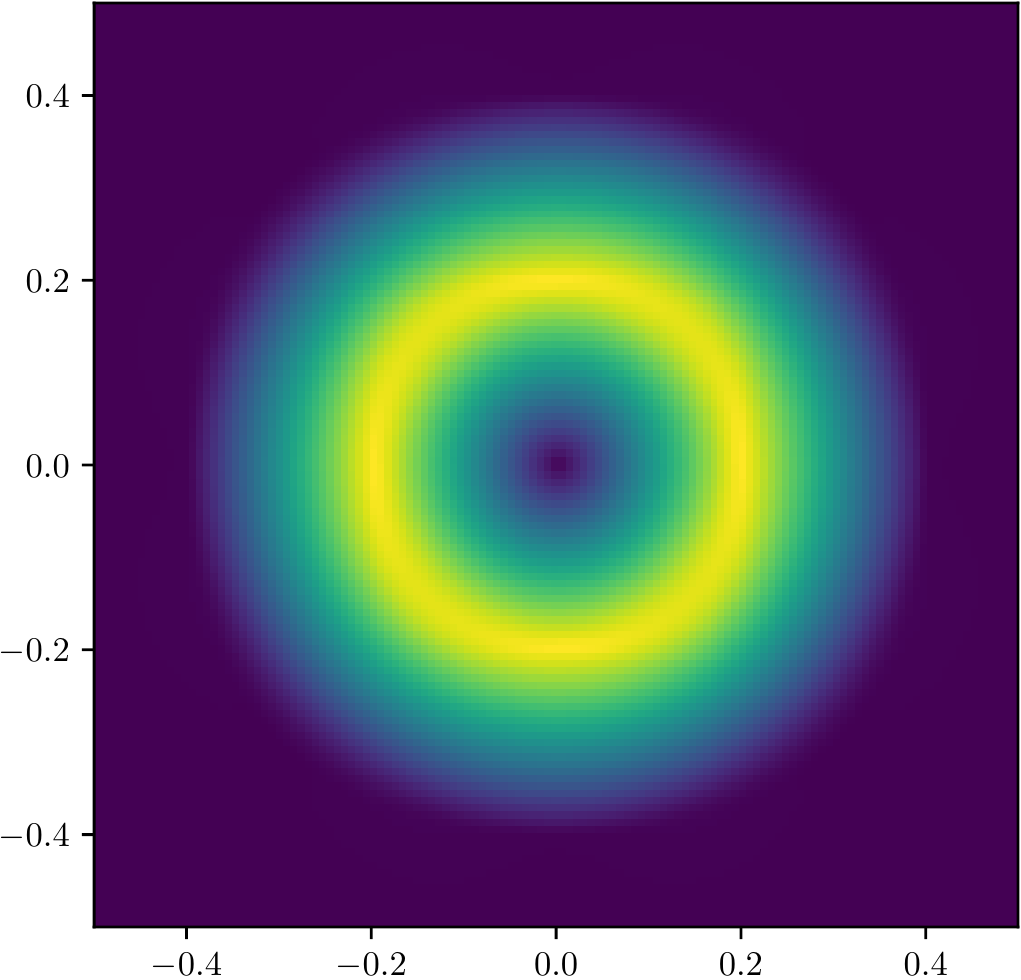}
        \\
        FSLP & OSLP & HLLC
        \\
        \multicolumn{3}{c}{
            \includegraphics[angle=-90,width=0.5\linewidth]{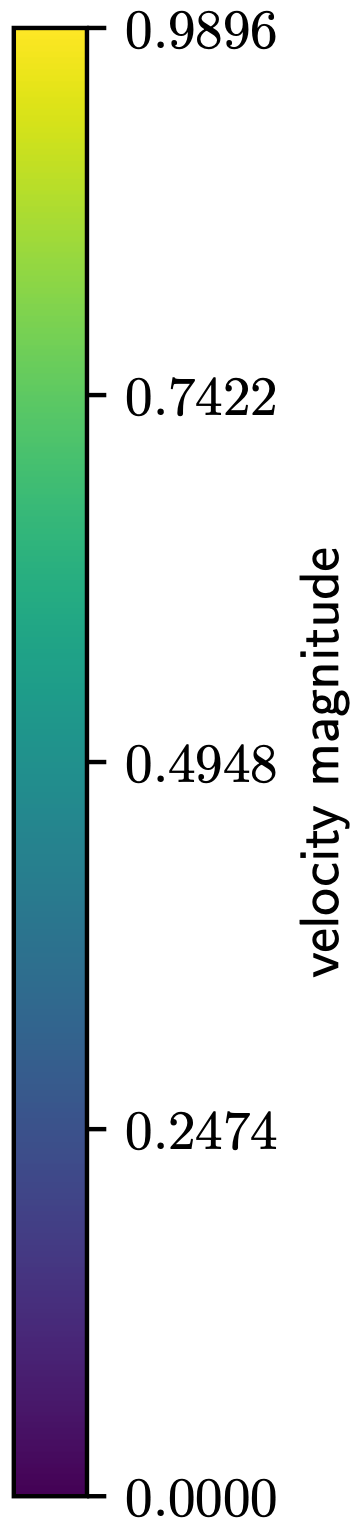}
        }
    \end{tabular}
    \caption{Comparison of the final velocity magnitude map for the Gresho vortex test case with $Ma=10^{-1}$ obtained with the FSLP, OSLP, and HLLC solvers on a $128\times 128$-cell grid at $t=0.1s$.}
    \label{greshofig1}
\end{figure}

\begin{figure}
    \centering
    \begin{tabular}{ccc}
        \includegraphics[width=0.3\linewidth]{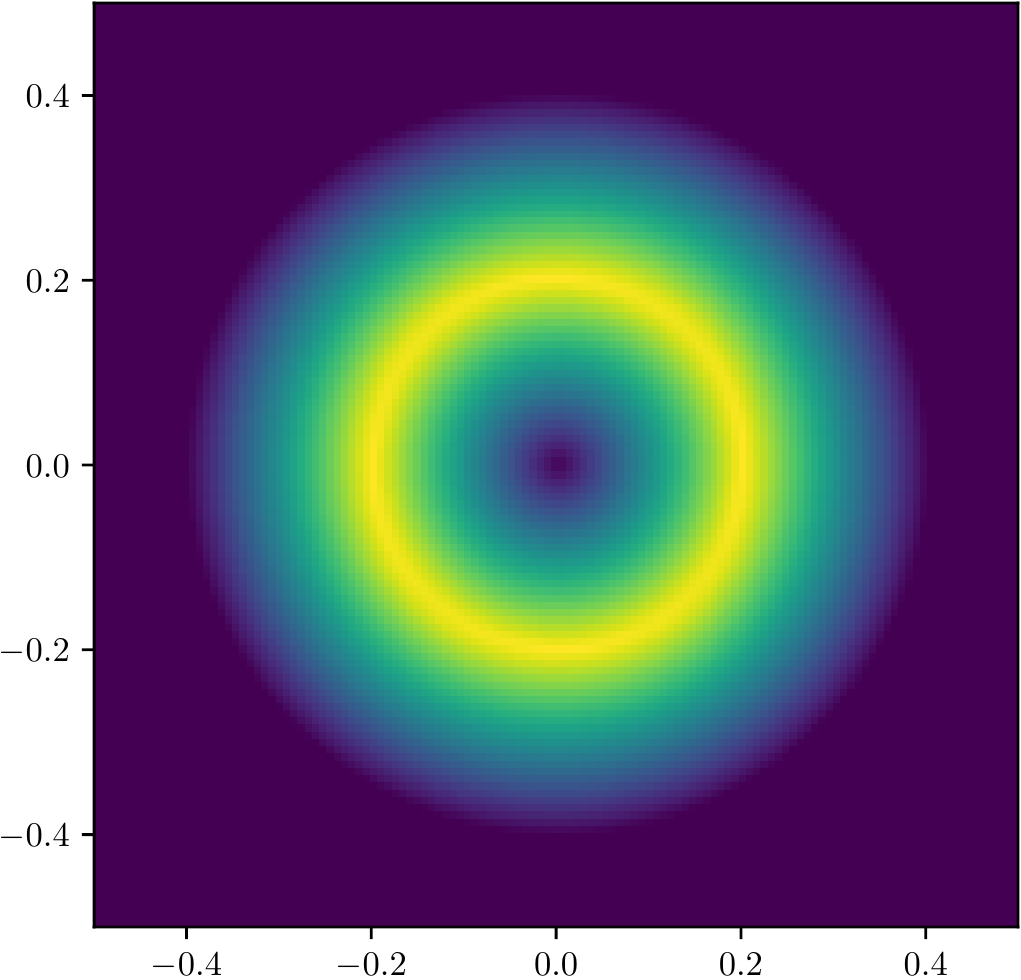}
        &
        \includegraphics[width=0.3\linewidth]{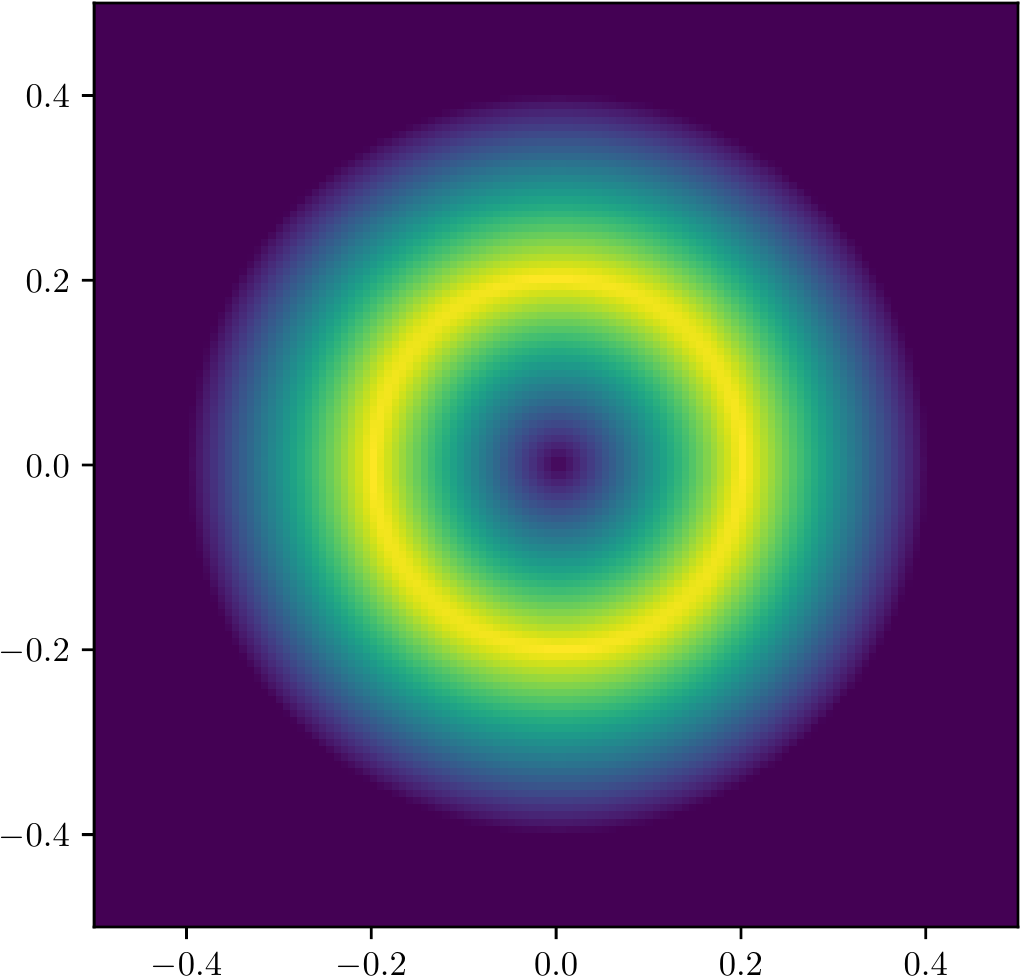}
        &
        \includegraphics[width=0.3\linewidth]{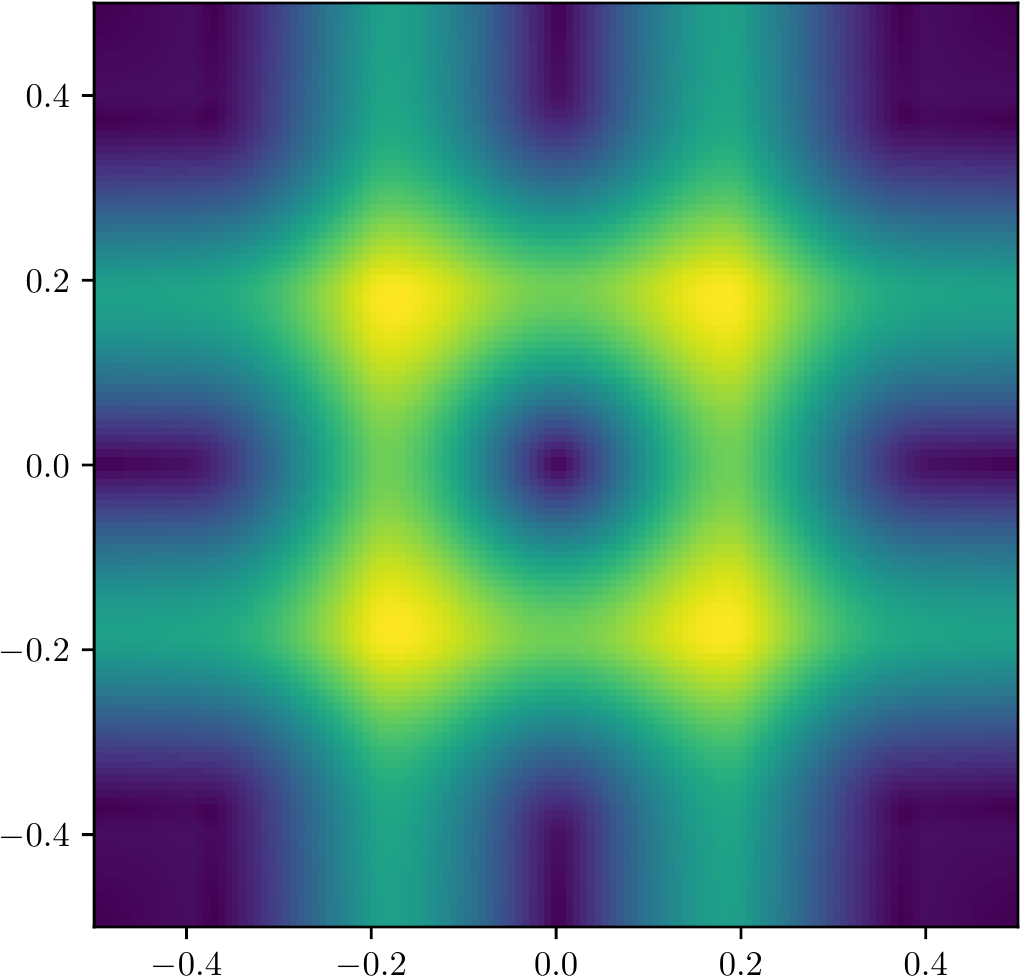}
        \\
        FSLP & OSLP & HLLC
        \\
        \multicolumn{3}{c}{
            \includegraphics[angle=-90,width=0.5\linewidth]{colorbar_gresho.png}
        }
    \end{tabular}
    \caption{Comparison of the final velocity magnitude map for the Gresho vortex test case with $Ma=10^{-3}$ obtained with the FSLP, OSLP, and HLLC solvers on a $128\times 128$-cell grid at $t=0.1s$.}
    \label{greshofig3}
\end{figure}

\begin{figure}
    \centering
    \begin{tabular}{ccc}
        \includegraphics[width=0.3\linewidth]{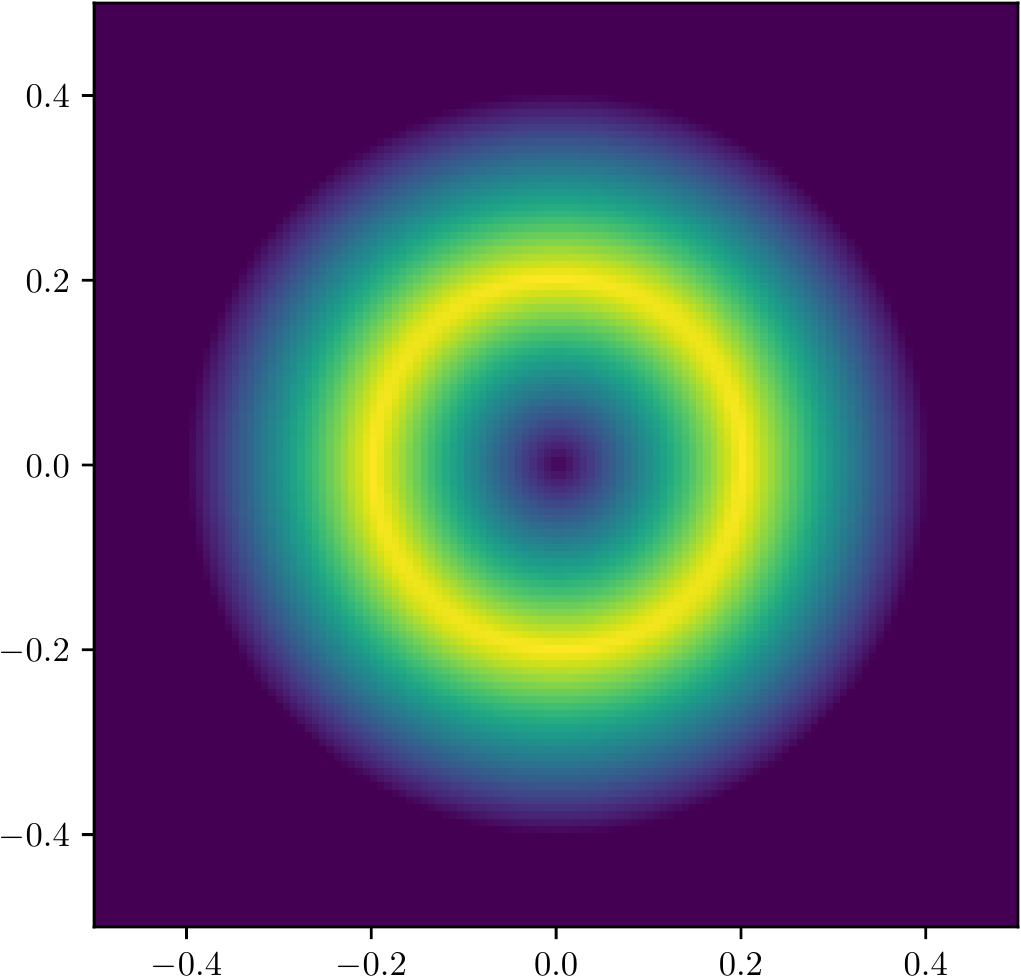}
        &
        \includegraphics[width=0.3\linewidth]{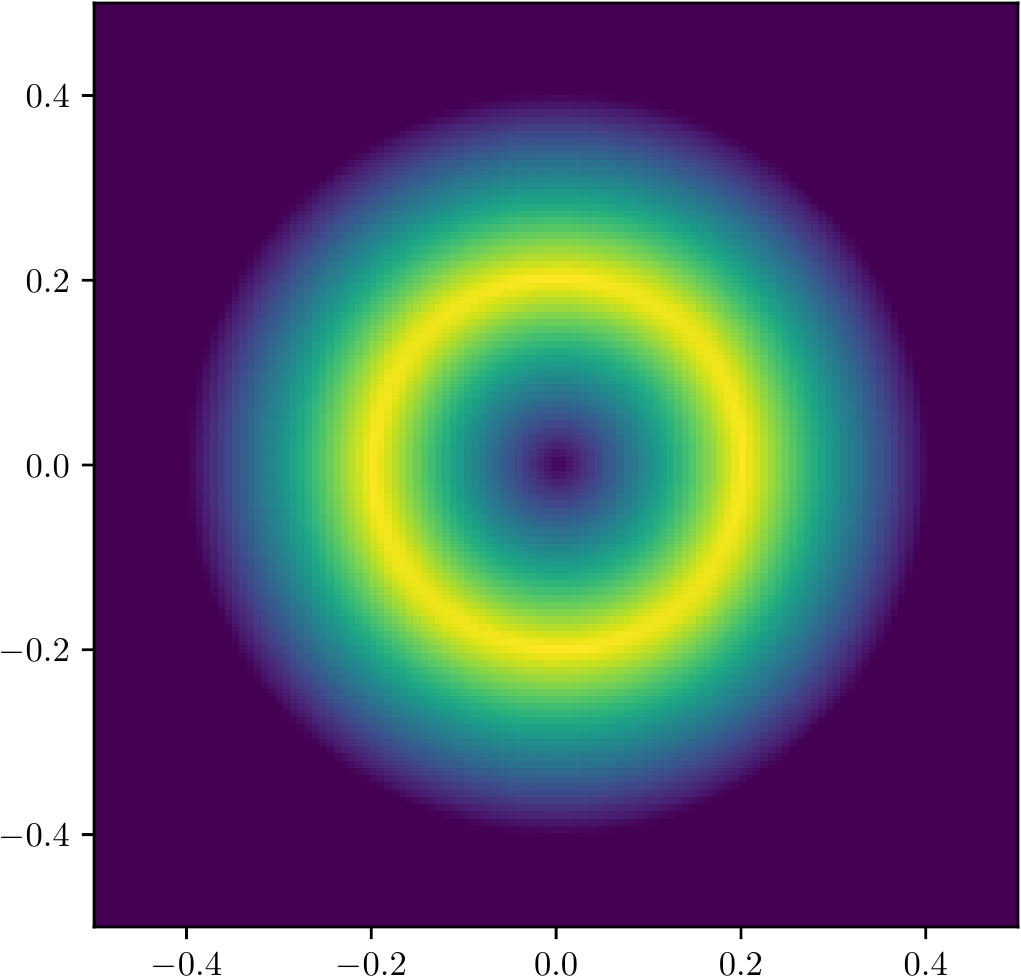}
        &
        \includegraphics[width=0.3\linewidth]{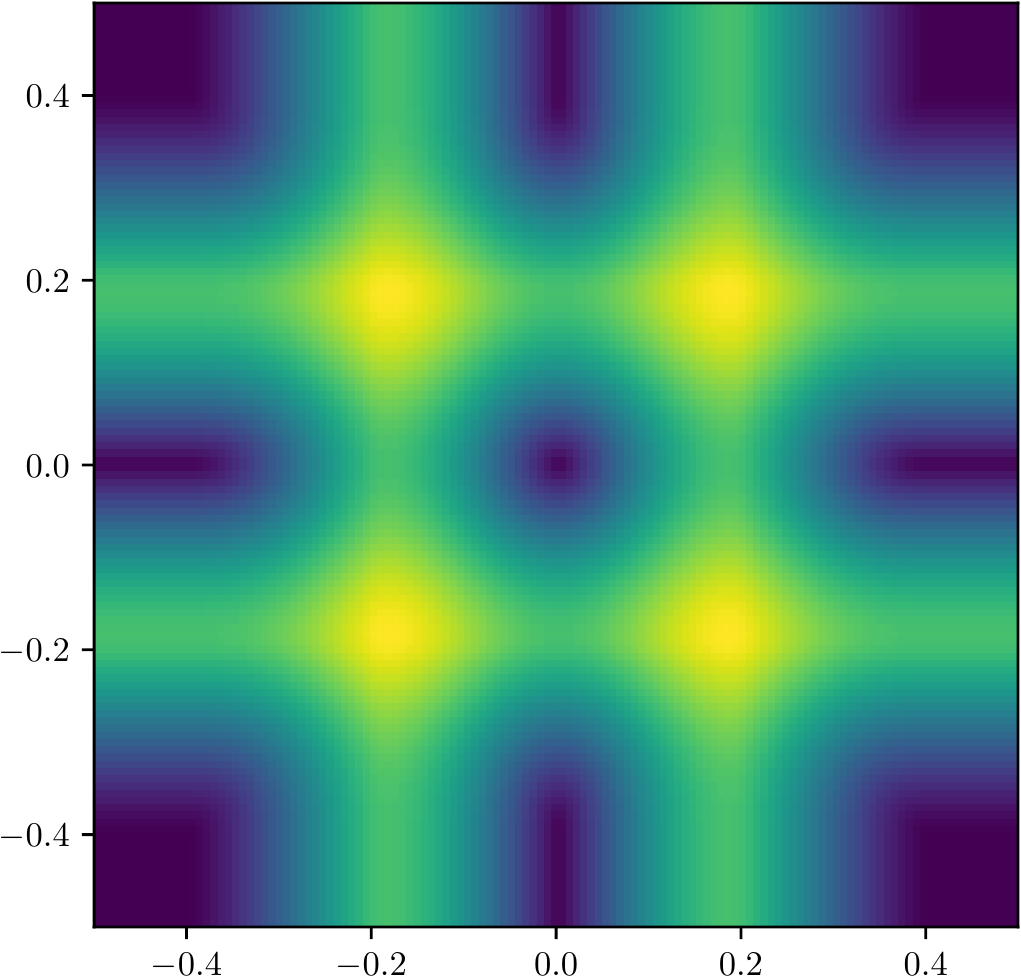}
        \\
        FSLP & OSLP & HLLC
        \\
        \multicolumn{3}{c}{
            \includegraphics[angle=-90,width=0.5\linewidth]{colorbar_gresho.png}
        }
    \end{tabular}
    \caption{Comparison of the final velocity magnitude map for the Gresho vortex test case with $Ma=10^{-5}$ obtained with the FSLP, OSLP, and HLLC solvers on a $128\times 128$-cell grid at $t=0.1s$.}
    \label{greshofig5}
\end{figure}
Figures \ref{greshofig1}, \ref{greshofig3}, \ref{greshofig5} give us the final velocity magnitude map for the Gresho vortex obtained with different solvers and Mach numbers.
For  $Ma = 10^{-1}$, we can see in figure~\ref{greshofig1} that on all three simulations, the initial velocity ring is preserved.
Figure~\ref{greshofig3} displays the results for $Ma = 10^{-3}$: one can see that the FSLP and OSLP methods can both preserve the velocity ring thanks to the low Mach correction while the HLLC methods fail to do so.
The same behavior is observed for $Ma = 10^{-5}$ (see figure~\ref{greshofig5}).
In order to measure the numerical diffusion effect of the solver, we evaluate the ratio $e_\text{kin}/e_\text{kin}^0$ of the kinetic energy obtained at the final instant and the initial instant with
\begin{align}
    e_\text{kin} &= \sum_{j}\frac{1}{2}\rho_j^n \qty( (u_j^n)^2 + (v_j^n)^2 ) \Delta x^2
    ,&
    e^0_\text{kin} &= \sum_{j}\frac{1}{2}\rho_j^0 \qty( (u_j^0)^2 + (v_j^0)^2 )\Delta x^2
    .
\end{align}
The results are displayed in table~\ref{table: gresho kinetic energy}.
\begin{table}[h]
    \caption{Gresho vortex test case: evaluation of the kinetic energy in the computational domain for different values of the Mach number $Ma$.}
    \label{table: gresho kinetic energy}
    \centering
    \begin{tabular}{ |c|c|c|c| }
        \hline
        & $Ma=10^{-1}$ & $Ma=10^{-3}$ & $Ma=10^{-5}$ \\
        \hline
        $e_{kin}/e_{kin}^0$ (at $t=10^{-2}$) --- OSLP scheme & 0.9966 & 0.9966 & 0.9966 \\
        \hline
        $e_{kin}/e_{kin}^0$ (at $t=10^{-2}$) --- FSLP scheme & 0.9966 & 0.9966 & 0.9966 \\
        \hline
        $e_{kin}/e_{kin}^0$ (at $t=10^{-2}$) --- HLLC scheme& 0.9762 & 0.5262 & 0.5167 \\
        \hline
    \end{tabular}
\end{table}
They show that both FSLP and OSLP solvers better preserve the kinetic energy than the HLLC method in the low Mach regime.

\subsection{Two-dimensional Riemann problems}
We now intend to study the ability of the FSLP method to capture
more complex wave patterns in a two-dimensional setting, including shocks and rarefaction waves. To that end, we consider the popular 2D Riemann problem of the literature referred to as Configuration~3 in \cite{Liska2003}. The computational domain is the rectangle $[0,1]\times[0,1]$, with the initial conditions
\begin{equation}
    (\rho, u, v, p)(x,y,t=0) =
    \begin{cases}
        (0.138 ,\ 1.206 ,\ 1.206  ,\ 0.029) &\text{if $x < 0.8$, $y < 0.8$\quad (bottom left)}\\
        (0.5323,\ 0.0   ,\ 1.206  ,\ 0.3) &\text{if $x > 0.8$, $y < 0.8$\quad (bottom right)}\\
        (0.5323,\ 1.206 ,\ 0.0    ,\ 0.3) &\text{if $x < 0.8$, $y > 0.8$\quad (top left)}\\
        (1.5   ,\ 0.0   ,\ 0.0    \ 1.5) &\text{if $x > 0.8$, $y > 0.8$\quad (top right).}
    \end{cases}
\end{equation}
We impose homogeneous Neumann conditions at the boundaries. We compute a reference solution thanks to a second-order HLLC method on a $384\times384$-grid. The waves at play produce a jet that propagates along the diagonal $x=y$ creating an important low Mach region in the center and the top right part of the domain (see figure~\ref{2DRP Mach}).

\begin{figure}
	\centering
		\includegraphics[width=0.49\linewidth]{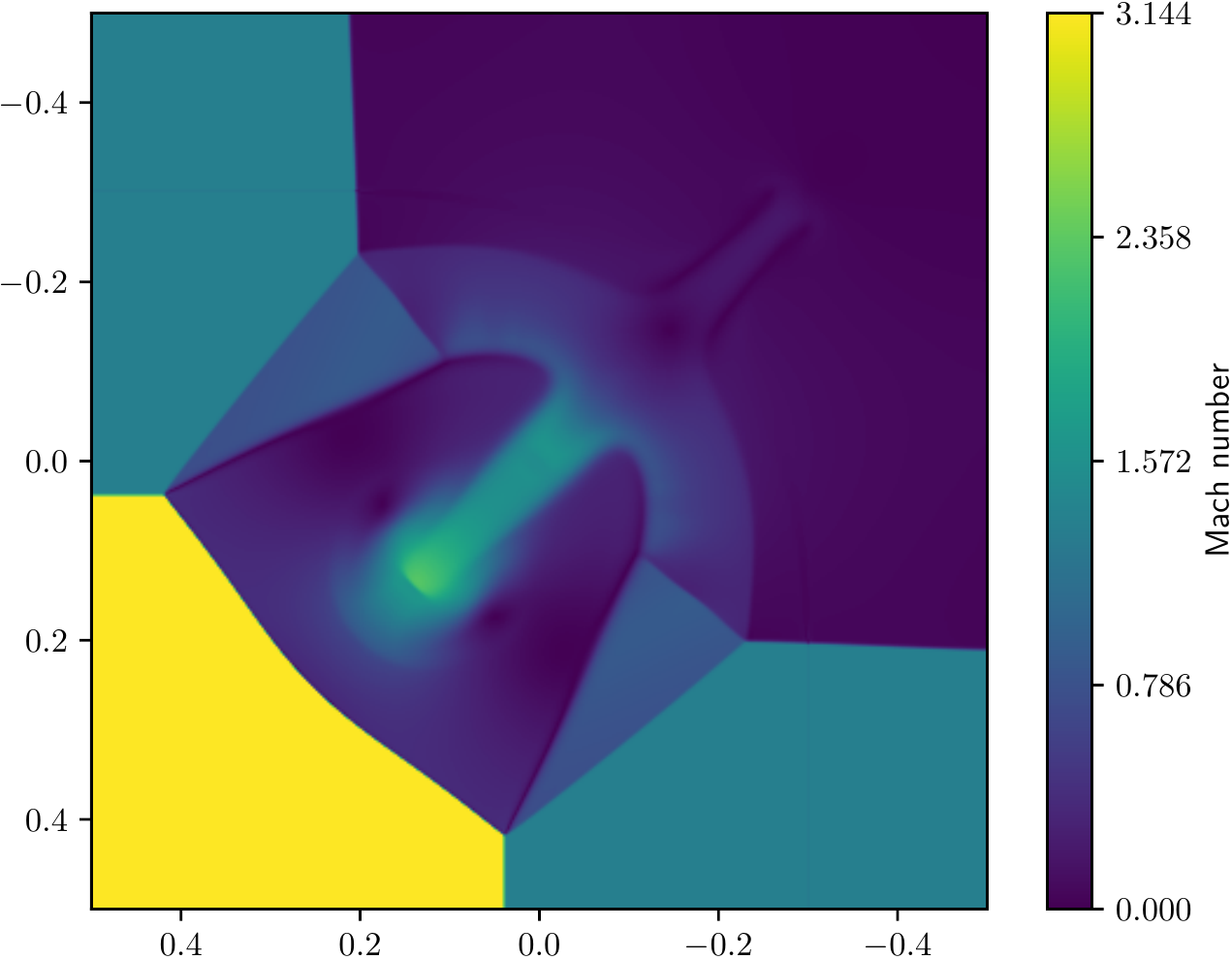}
\caption{2D Riemann problem. Mapping of the Mach number as $t=0.8s$. The reference simulation is obtained with a second-order HLLC method on a $384\times384$-cell mesh.}
s\label{2DRP Mach}
\end{figure}

\begin{figure}
\centering
\begin{tabular}{cc}
\includegraphics[width=0.49\linewidth]{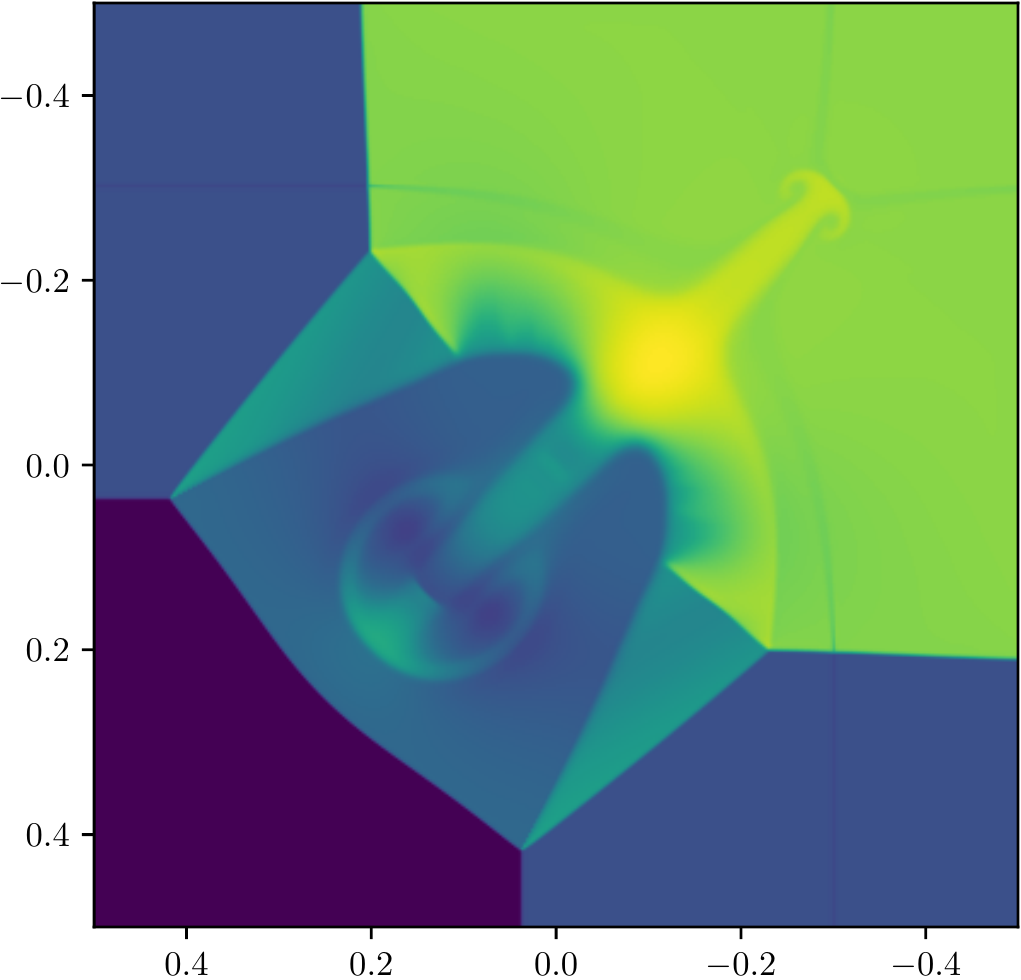}
	&
\includegraphics[width=0.49\linewidth]{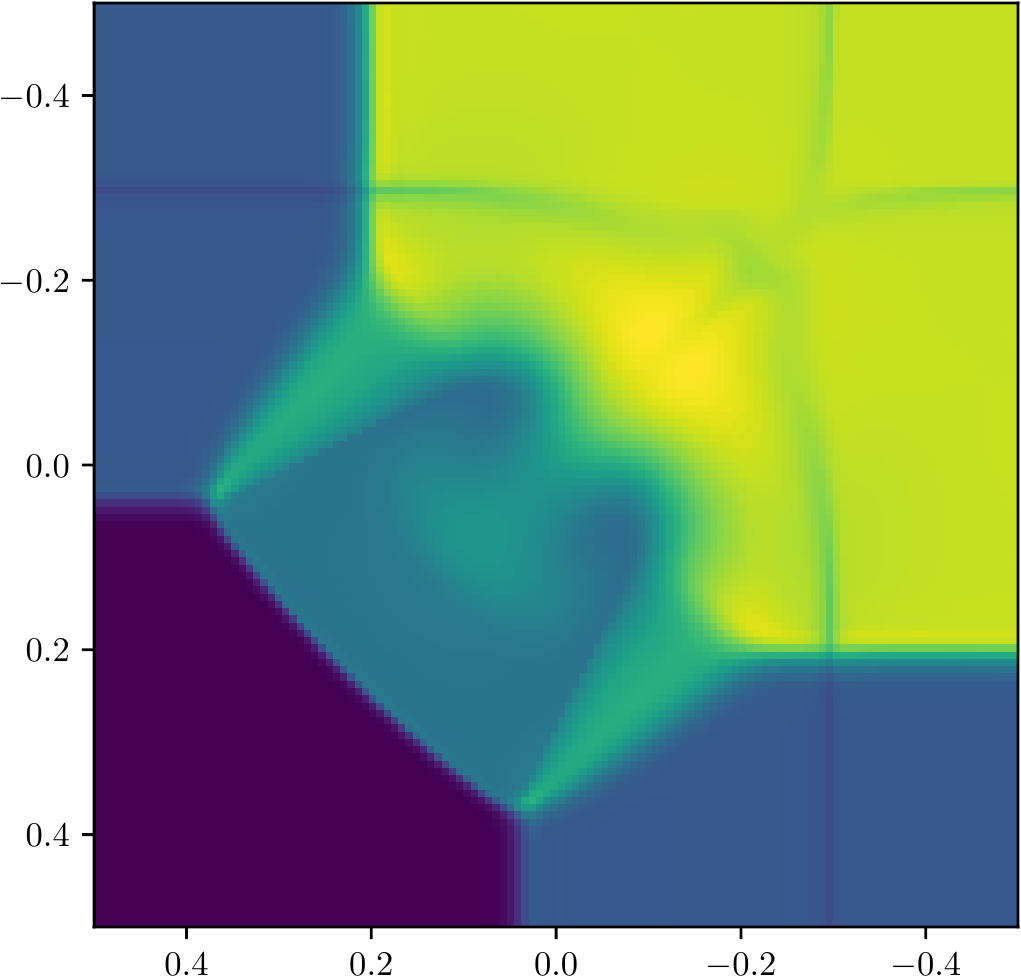}
\\
Reference result
&
FSLP method (order 1)
	\\
\includegraphics[width=0.49\linewidth]{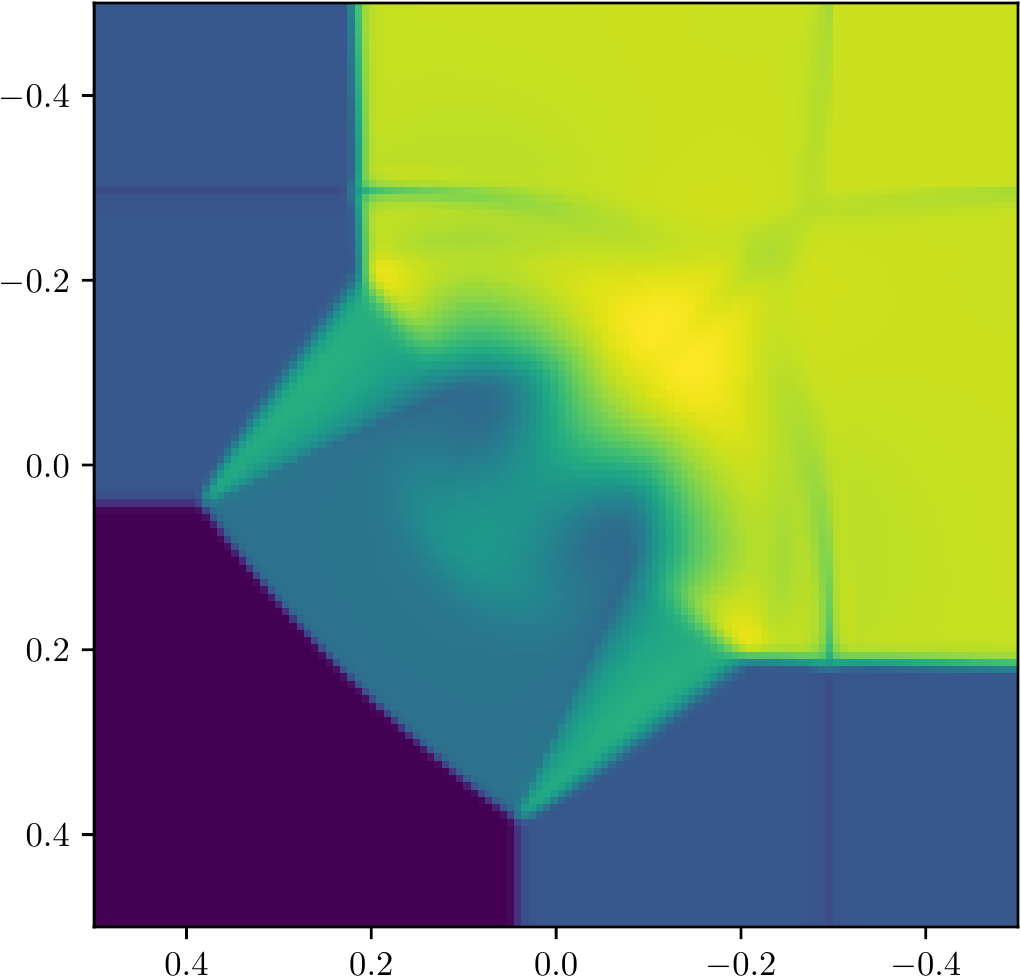}
&
	\includegraphics[width=0.49\linewidth]{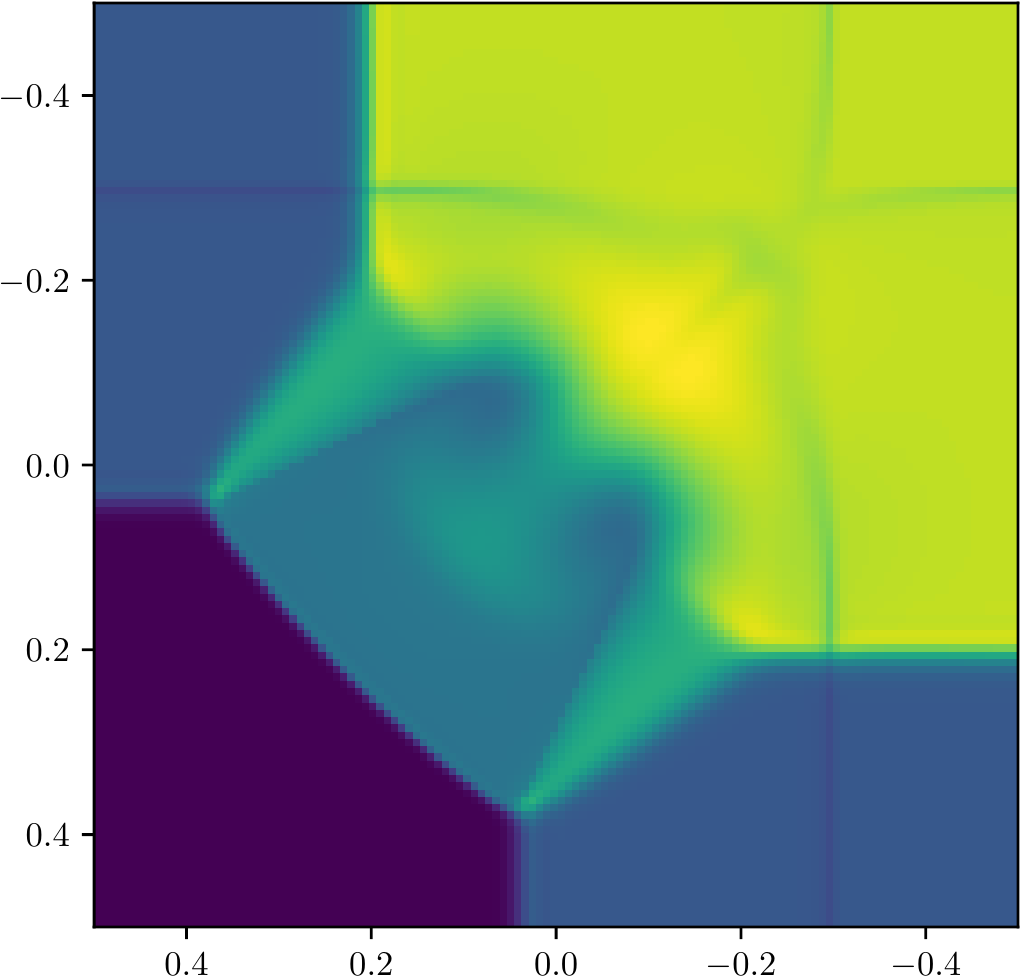}
\\
HLLC method (order 1)
&
OSLP method (order 1)
	\\
\multicolumn{2}{c}{
\includegraphics[angle=-90,width=0.5\linewidth]{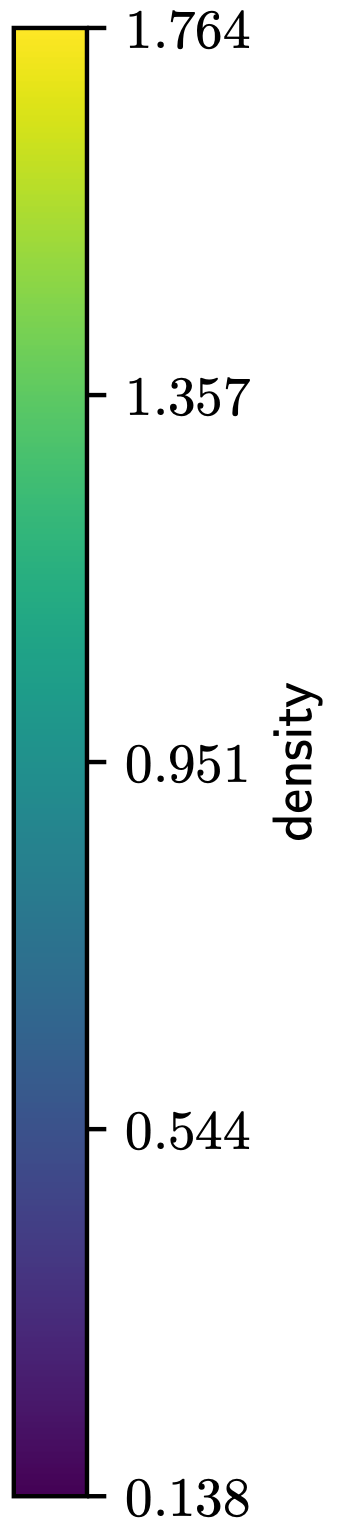}
	}
\end{tabular}
\caption{2D Riemann problem. Mapping of the density number as $t=0.8s$.
The reference simulation is obtained with a second-order HLLC method on a $384\times384$-cell mesh. The other simulations are performed on a $128\times128$-cell grid with the first-order FSLP method (top right), the HLLC first-order method (bottom left), and the first-order OSLP method (bottom right).
}
\label{2DRP density}
\end{figure}
\begin{figure}
	\centering
	\begin{tabular}{cc}
		\includegraphics[width=0.49\linewidth]{2DRP_ref_d.pdf}
		&
		\includegraphics[width=0.49\linewidth]{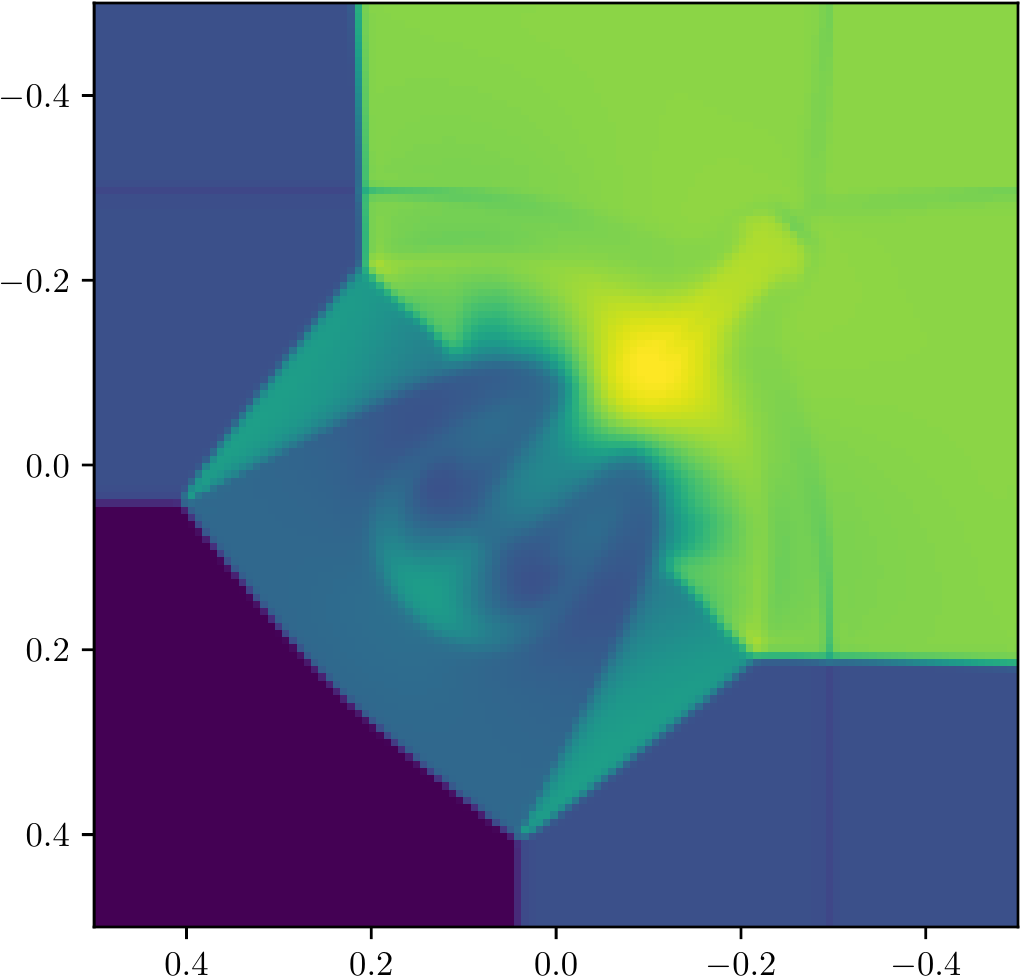}
		\\
		Reference result & HLLC method (second-order)
		\\
		\includegraphics[width=0.49\linewidth]{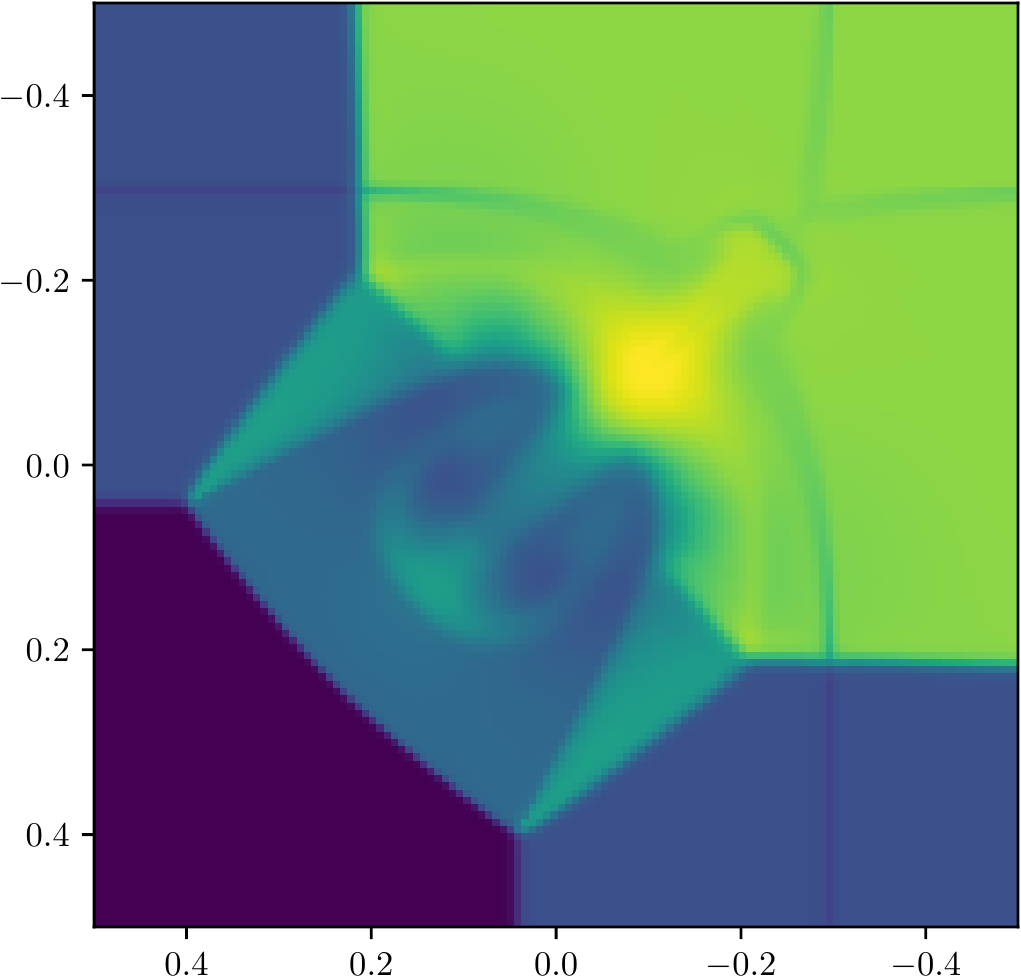}
		&
		\includegraphics[width=0.49\linewidth]{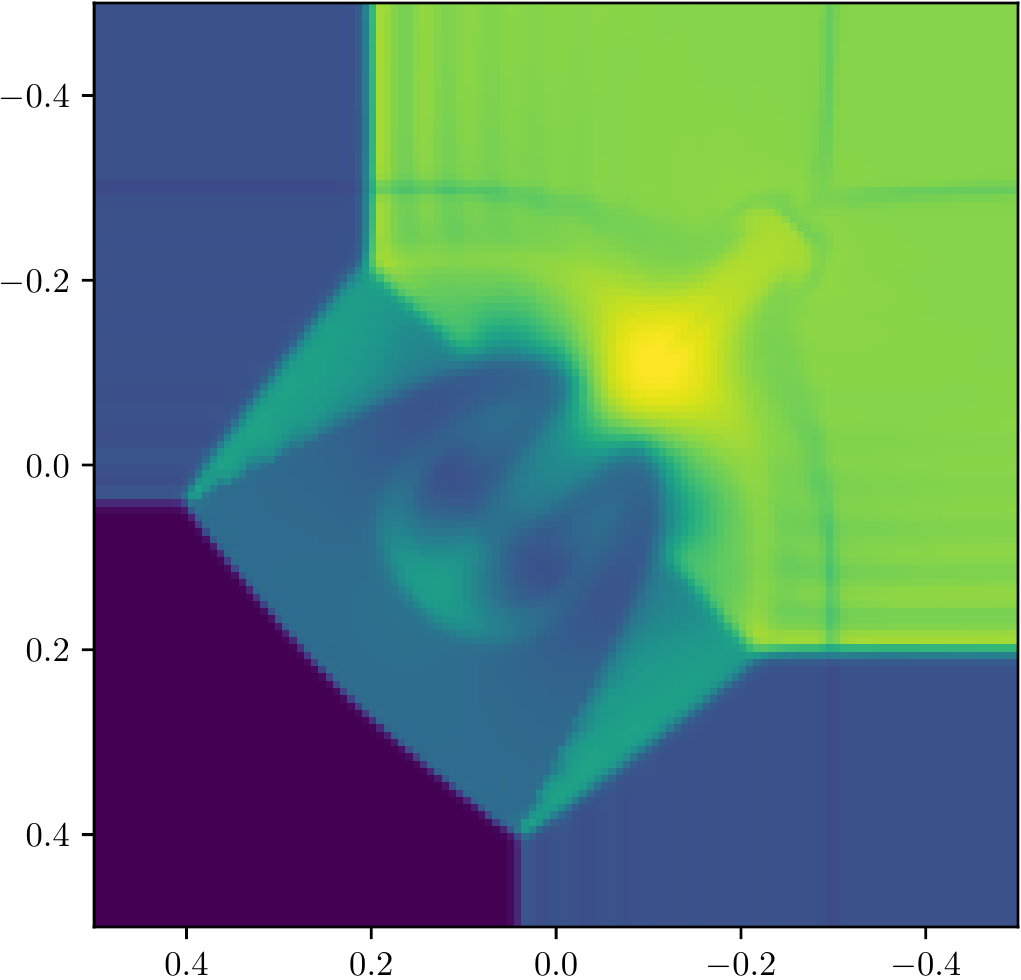}
		\\
		FSLP (second-order without low Mach correction)  &
		FSLP (second-order with low Mach correction)
		\\
		\multicolumn{2}{c}{
		\includegraphics[angle=-90,width=0.5\linewidth]{colorbar_2DRP.png}
		}
	\end{tabular}
\caption{2D Riemann problem. Mapping of the density number as $t=0.8s$.
	The reference simulation is obtained with a second-order HLLC method on a $384\times384$-cell mesh. The other simulations are performed on a $128\times128$-cell grid with the second-order HLLC method (top right), the FSLP second-order method without low-Mach correction (bottom left), the second-order FSLP method with low-Mach correction (bottom right).
}
	\label{2DRP density second-order}
\end{figure}

Figure~\ref{2DRP density} shows a mapping of the density obtained with the OSLP (first-order), the FSLP (first and second-order), and the HLLC (first and second-order) schemes using a $128\times128$-cell mesh. One can see that the overall wave pattern
is rendered successfully by all numerical schemes. The results of the FSLP, OSLP and HLLC schemes for first-order methods are similar. Second-order
methods all better succeed in capturing the shape of the jet as depicted in figure~\ref{2DRP density second-order}.
\hlb{
Although the HLLC scheme poorly performs in the low Mach regime on a coarse grid, this defect vanishes when one refines the grid \cite{Dellacherie2010}.
Therefore, for the present test, we use a simulation performed with the HLLC solver on a 400x400 Cartesian grid as reference solution.
The objective is here to attest that comparable accuracy can be obtained with the FSLP solver on a coarser grid.}
Nevertheless, we can note that spurious oscillations appear in the simulation performed with the second-order FSLP scheme with low-Mach correction. These spurious waves propagate along the $x$ and $y$ axes in the top right part of the domain. We believe they are caused by the lack of numerical dissipation around the low-Mach shocks due to the combination of the low-Mach correction and the second-order reconstruction.
A more careful choice of $\theta$ than \eqref{choice theta} is required to ensure the discrete entropy inequality (see \eqref{Co split}).
 Improving the second-order discretization for the FSLP scheme would, for example,
  require proposing a better choice for $\theta$ but such a task is beyond the scope of the 
  present work. 

\subsection{Hydrostatic equilibrium test}
In order to challenge the well-balanced ability of the FSLP scheme, we consider the atmosphere at rest test (see, for example \cite{padioleau2019high}).
It involves a fluid column of a perfect gas in a
rectangular $[0,2]\times[0,1]$ domain. For this test, the gravity acceleration is set to $g=-1$ so that $\phi(x,y) = -y$. For the EOS of the fluid, we set $\gamma = 5/3$ and $c_v = 1$, where $c_v$ is the heat capacity at constant volume so that the temperature $T$ of the gas  is given by $e = c_v T$.
We consider periodic boundary conditions for the left and right sides of the domain. At the top and bottom of the domain, wall boundaries are imposed for the normal velocity, while the temperature is linearly extrapolated.
The initial condition is built by imposing a linear temperature profile as follows
\begin{subequations}
	\begin{align}
T(x,y=0,t=0) &=  3.78565
,&
\grad T(x,y,t=0) &= (0, -1.2)^T
,\\
\rho(x,y=0,t=0) &= 1
,&
\grad ( c_v (\gamma-1) \rho T)(x,y,t=0) &= (0,\rho g )^T.
\end{align}
\end{subequations}

The computational domain is discretized over a $100\times 50$ on which we let the solver evolve the profile for $t\in[0,100s]$.
Table~\ref{hydrostatic equilibrium velocity measure} displays the value of the $\max_{i,j}|v_{i,j}^n|$ at $t=100 s$ and
shows that both the OSLP and the FSLP first-order methods preserve the velocity magnitude at zero-machine precision.

\begin{table}[h]
	\centering
	\caption{Hydrostatic equilibrium test. Measure of the velocity magnitude at $t=100s$}
	\label{hydrostatic equilibrium velocity measure}
\begin{tabular}{ |c|c|c| }
 \hline
Solver &OSLP & FSLP\\
\hline
Average speed &$1.342\times 10^{-14}$ & $2.056\times 10^{-14}$\\
\hline
\end{tabular}
\end{table}

 It is important to mention that a direct second-order extension of the well-balanced method, as presented in section~\ref{MUSCL} will fail to preserve the hydrostatic equilibrium. This question of designing a well-balanced high-order method has been successfully investigated in the literature \cite{Castro2017, Luna2020,Grosso2021}. Adapting these techniques to the FSLP scheme is beyond the scope of this paper.

\subsection{Rayleigh-Taylor instability}\label{section: RT}
We now consider the Rayleigh-Taylor test performed in \cite{padioleau2019high}: the computational domain is  $[-1/4, 1/4]\times[-3/4, 3/4]$ and the fluid is a perfect gas with $\gamma = 5/3$. At $t=0$
a dense layer of fluid lies on top of a lighter layer so that the configuration is unstable. The gravity acceleration is $g=-0.1$ thus $\phi(x,y) = - 0.1\times y$. The initial conditions are given by

\begin{subequations}
\begin{align}
		\rho(x, y,t=0) &=\begin{cases}
			1 & \text { for  \(y<0\),} \\
			2 & \text { for \(y\geq 0\),}
		\end{cases}
	\\
		p(x, y, t=0) &=  -\rho\phi,\\
		(u,v)(x, y ,t=0) &=
			\left(
			0,\frac{C}{4}\left(1+\cos \left(4 \pi x \right)\right)
			\left(1+\cos \left(3 \pi y\right)\right)
			\right)
			 .
			 \label{RT initial velocity}
	\end{align}
\end{subequations}
The initial velocity~\eqref{RT initial velocity} imposes a single-mode perturbation
of magnitude $C = 0.01$ that will break the hydrostatic equilibrium.
\begin{figure}
	\centering
	\includegraphics[width=.3\linewidth]{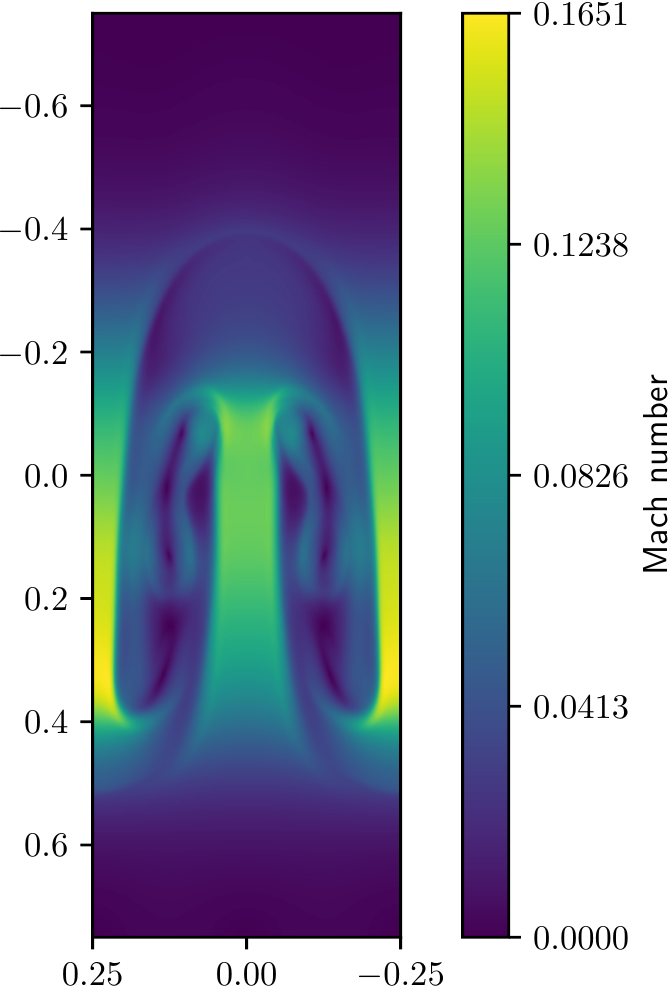}
	\caption{Rayleigh-Taylor instability: mapping of the
		Mach number profile for reference solution obtained with a second-order HLLC method
		on a $200\times \hlb{600}$-cell mesh at $t=12.4s$.}
	\label{RT Mach}
\end{figure}%

\begin{figure}
	\centering
	\begin{tabular}{cccc}
		\includegraphics[width=.225\linewidth]{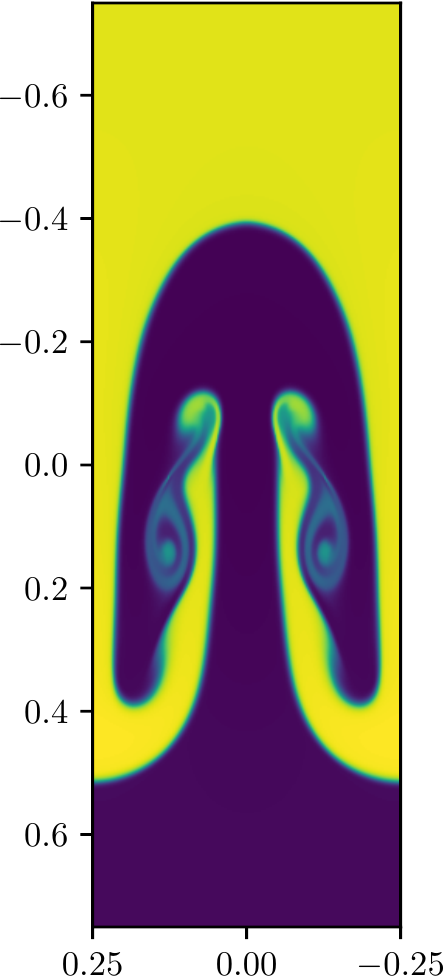}
		&
		\includegraphics[width=.225\linewidth]{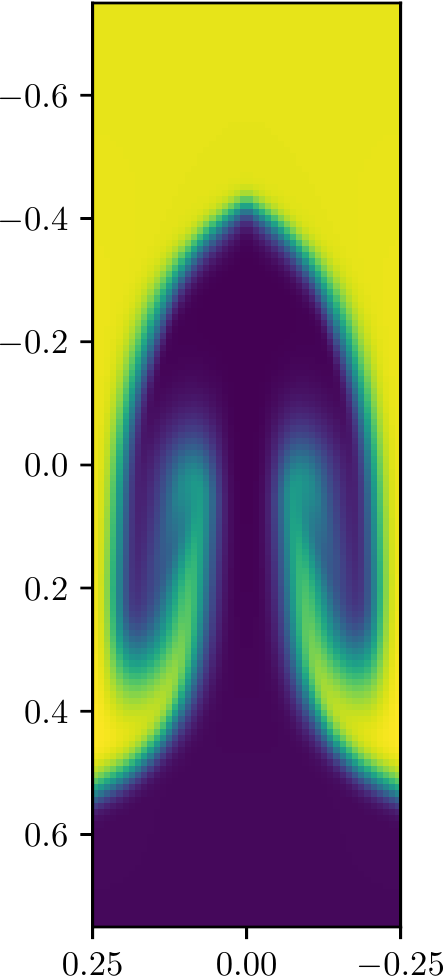}
		&
		\includegraphics[width=.225\linewidth]{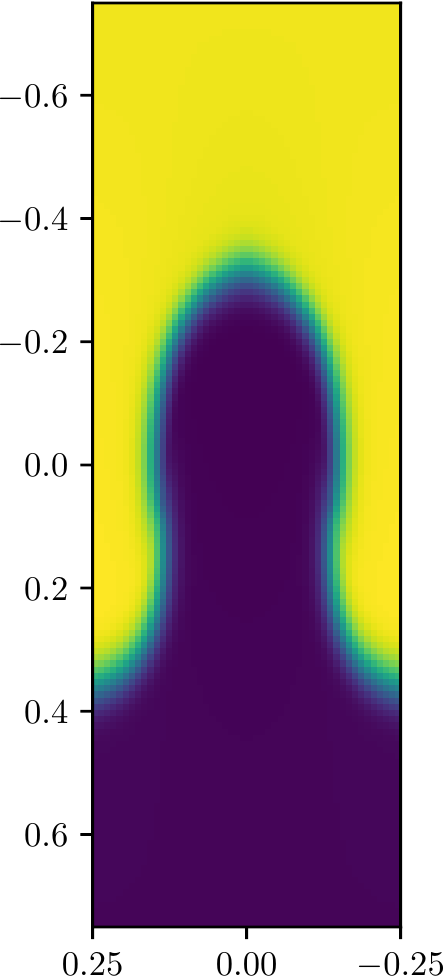}
		&
		\includegraphics[width=.225\linewidth]{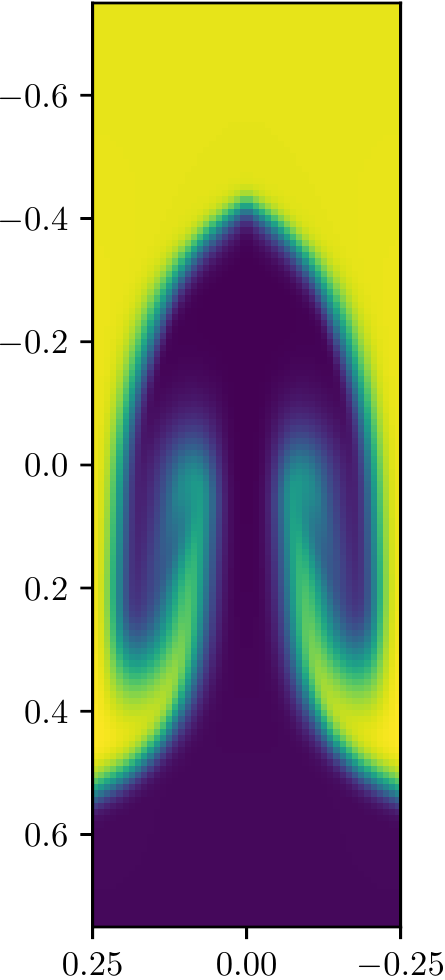}
		\\
		reference result
		&
		FSLP (first-order)
		&
		HLLC (first-order)
		&
		OSLP (first-order)
		\\
		\\
		\includegraphics[width=.225\linewidth]{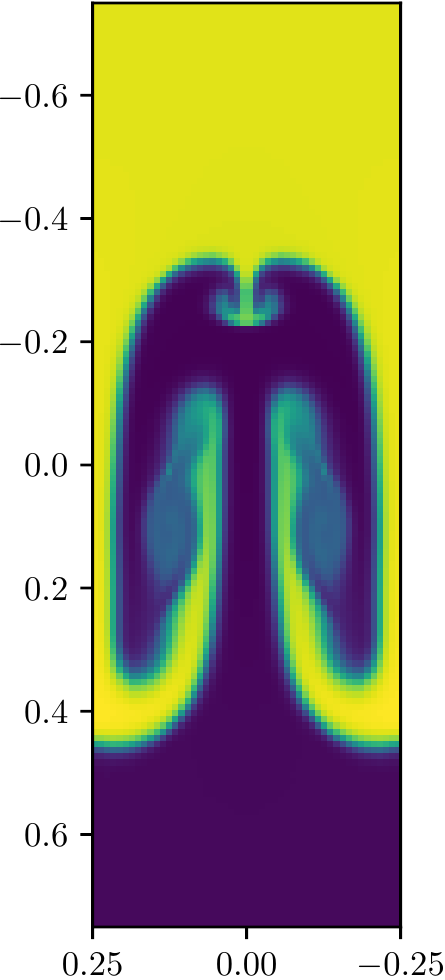}
		&
		\includegraphics[width=.225\linewidth]{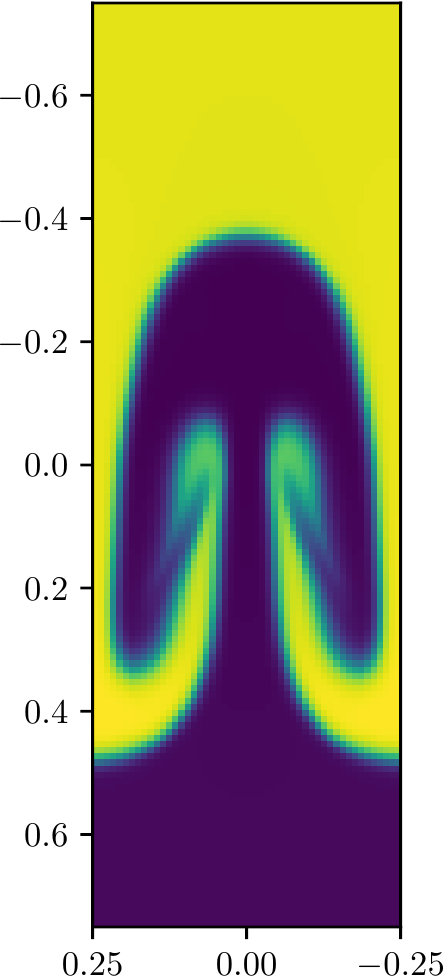}
		&
		\includegraphics[width=0.1\linewidth]{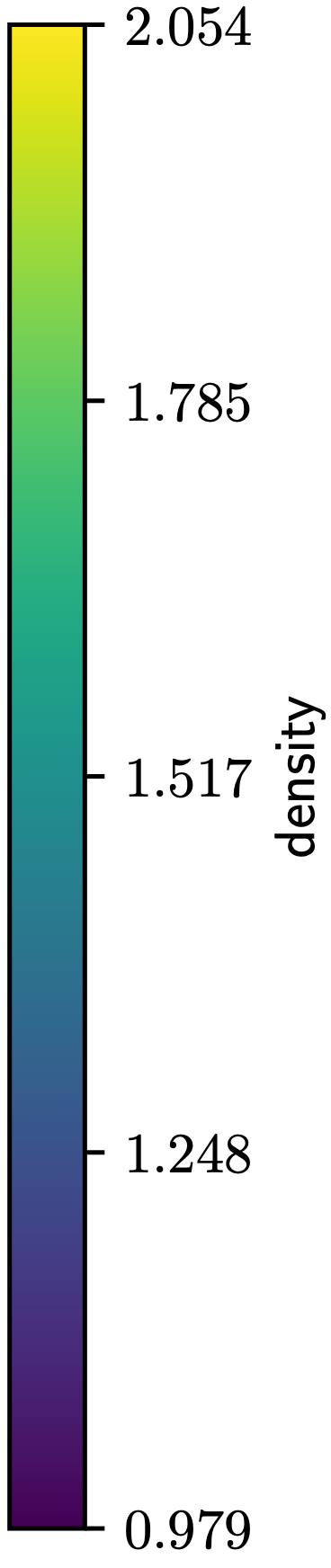}
		&
		\\
		FSLP (second-order)
		&
		HLLC (second-order)
		&&
	\end{tabular}
   \caption{Rayleigh-Taylor instability: mapping of the
	density profile for:
	the reference results obtained with the second-order HLLC method on a
	$200\times 600$-cell grid,
	the FSLP (first and second-order), the HLLC (first and second-order), and the OSLP methods
	on a $50\times 150$-cell mesh at $t=12.4s$.}
\label{RT density}
\end{figure}

This test allows measuring and comparing the effect of the numerical diffusion of each method as it tends to limit the development of high-frequency modes in the instability. Figure~\ref{RT density} and \ref{RT Mach} respectively show the density and Mach number of a reference second-order HLLC simulation obtained with a $200\times600$-cell mesh. We observe a sharp transition between both fluid layers, and the interface presents lateral arms with secondary rolls.

Figure~\ref{RT density} shows simulations ran with both the FSLP solver and the HLLC solver
on a coarse $50\times150$-cell mesh obtained with first and second-order methods.
The HLLC method presents an important amount of numerical diffusion: it only shows a single mode growth, and no lateral arm is created. On the other hand, The FSLP method with low-Mach correction can produce the arms that appear on the reference HLLC simulation. It shows that our new method can better capture high-frequency flow features with much lower resolution than the classic HLLC solver, similar to OSLP.
This is due to the low-Mach nature of this test: as displayed in figure \ref{RT Mach} one can indeed see that
$Ma\in[0,0.165]$. Therefore the low-Mach correction at play in the FSLP solver has an important effect on the result. Note, however how this correction does not fix the important amount of numerical diffusion that appears at the interface between both layers with the FSLP solver. At second-order, the HLLC solution
shown in figure~\ref{RT density} does present lateral arms, similar to the first-order FSLP method. The second-order FSLP method presents many secondary rolls both on the front of the main mode and on the lateral arms.
This agreement with the reference solution displayed in figure~\ref{RT density} shows the higher accuracy of the second-order FSLP method. Finally, let us mention that the results obtained with the OSLP in figure~\ref{RT density} resemble the first-order FSLP simulation of figure~\ref{RT density}.

\subsection{The stationary vortex in a gravitational field}\label{section: vortex with gravity}

The stationary vortex in a gravity field test \cite{thomann2020all} is a modified version of the Gresho vortex \cite{Gresho1990} where a gravitational field and a background hydrostatic equilibrium state are added. It allows testing the low-Mach properties of numerical methods. We consider the sub-case of the setup proposed in \cite{thomann2020all} with $\froude=\mach, RT=1/\mach$, and an adiabatic index $\gamma=5/3$. The potential and the initial conditions are given by:

\begin{equation*}
\Phi(r)=\left\{\begin{array}{lll}
12.5 r^2 & \text { if } & r \leq 0.2 \\
0.5-\ln (0.2)+\ln (r) & \text { if } & 0.2<r \leq 0.4 \\
\ln (2)-0.5 \frac{r_c}{r_c-0.4}+2.5 \frac{r_c}{r_c-0.4} r-1.25 \frac{1}{r_c-0.4} r^2 & \text { if } & 0.4<r \leq r_c \\
\ln (2)-0.5 \frac{r_c}{r_c-0.4}+1.25 \frac{r_c^2}{r_c-0.4} & \text { if } & r>r_c,
\end{array} \right.
\end{equation*}
 with $r_c=0.5$. The density is given by:\begin{equation}
\rho=\exp \left(-Ma^2 \Phi\right),
\end{equation}
The radial velocity is null, and the tangential velocity is given by

\begin{equation}
u_\theta(r)=\frac{1}{u_r}\left\{\begin{array}{lll}
5 r & \text { if } & r \leq 0.2 \\
2-5 r & \text { if } & 0.2<r \leq 0.4 \\
0 & \text { if } & r>0.4
\end{array}\right.
\end{equation}

The pressure is $p=\rho/Ma^2 + p_2$ with:

\begin{equation}
p_2(r)=\frac{1}{u_r^2}\left\{\begin{array}{lll}
p_{21}(r) & \text { if } & r \leq 0.2 \\
p_{21}(0.2)+p_{22}(r) & \text { if } & 0.2<r \leq 0.4 \\
p_{21}(0.2)+p_{22}(0.4) & \text { if } & r>0.4
\end{array}\right.
\end{equation}
where $u_r = 0.4 \pi$ and

\begin{equation}
\begin{aligned}
p_{21}(r)=&\left(1-\exp \left(-12.5 Ma^2 r^2\right)\right) \\
p_{22}(r)=& \frac{1}{\left(1-Ma^2\right)\left(1-0.5 Ma^2\right)} \exp \left((-0.5+\ln (0.2)) Ma^2\right) \\
&\left(r^{-Ma^2}\left(Ma^4(r(10-12.5 r)-2)-4 Ma^4 (\gamma-1)^2+ Ma^2(r(12.5 r-20)+6)\right)\right.\\
&\left.+\exp \left(-\ln (0.2) Ma^2\right)\left(4 -2.5 Ma^2+0.5 Ma^4\right)\right).
\end{aligned}
\end{equation}

We consider the domain $[0,1]^2$ and define the radius from the center $r=(x-0.5)^2+(y-0.5)^2$. The initial Mach number distribution is shown for two configurations corresponding to $\mach=10^{-2}$ and $\mach=10^{-3}$ in figure \ref{greshogravfig3}. We let the vortex evolve until $t=1$ s corresponds to a full revolution and display the final Mach number distribution with different resolutions in figures \ref{greshofig1}, \ref{greshogravfig2}. We also give the final to initial kinetic energy ratio in table \ref{table: greshograv kinetic energy}. It is clear from the figures and the table that the numerical diffusion is indeed roughly independent of the Mach regime.
\begin{figure}
    \centering
\begin{tabular}{cc}
	\includegraphics[width=0.5\linewidth]{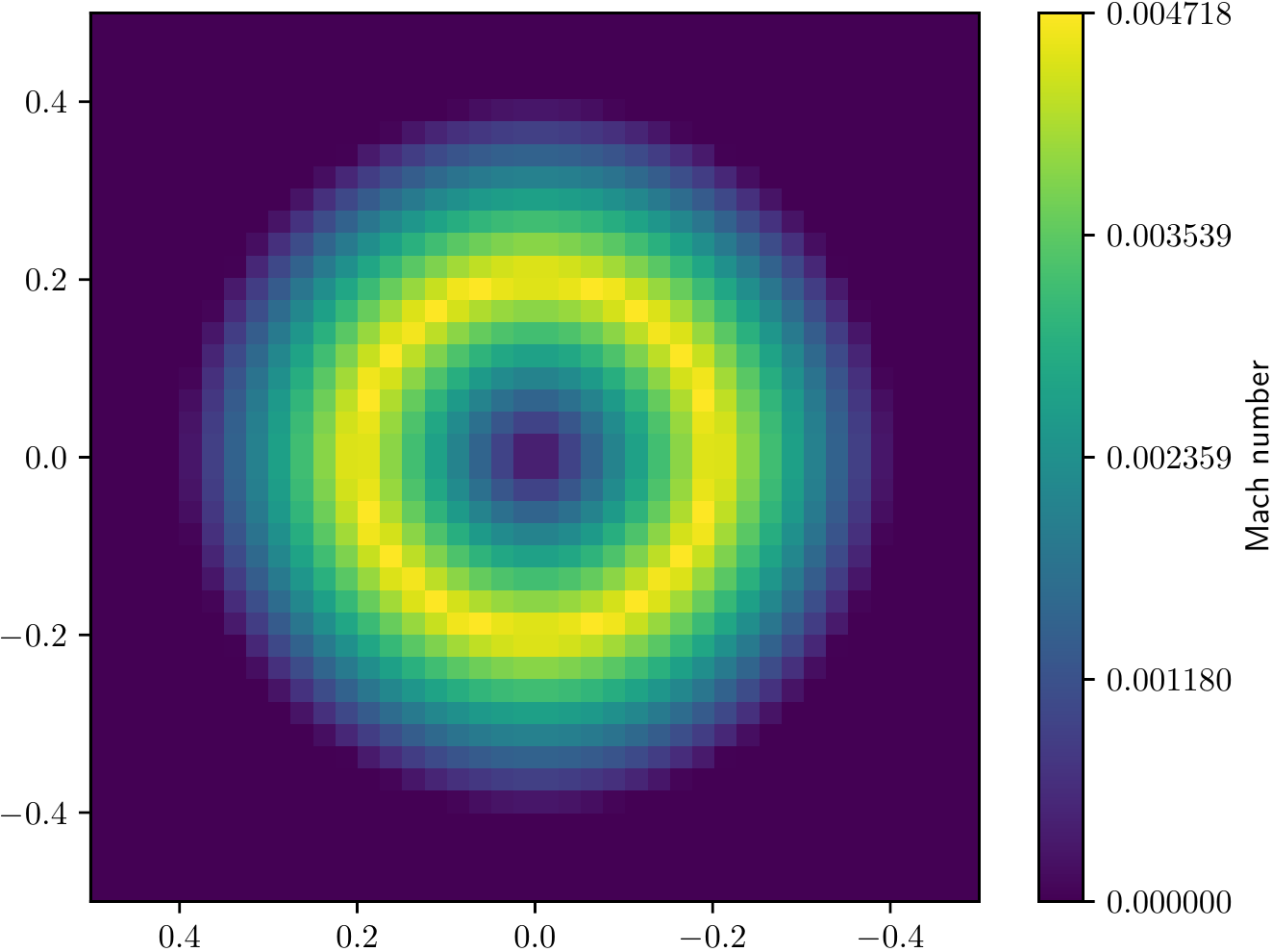}
	&
	\includegraphics[width=0.5\linewidth]{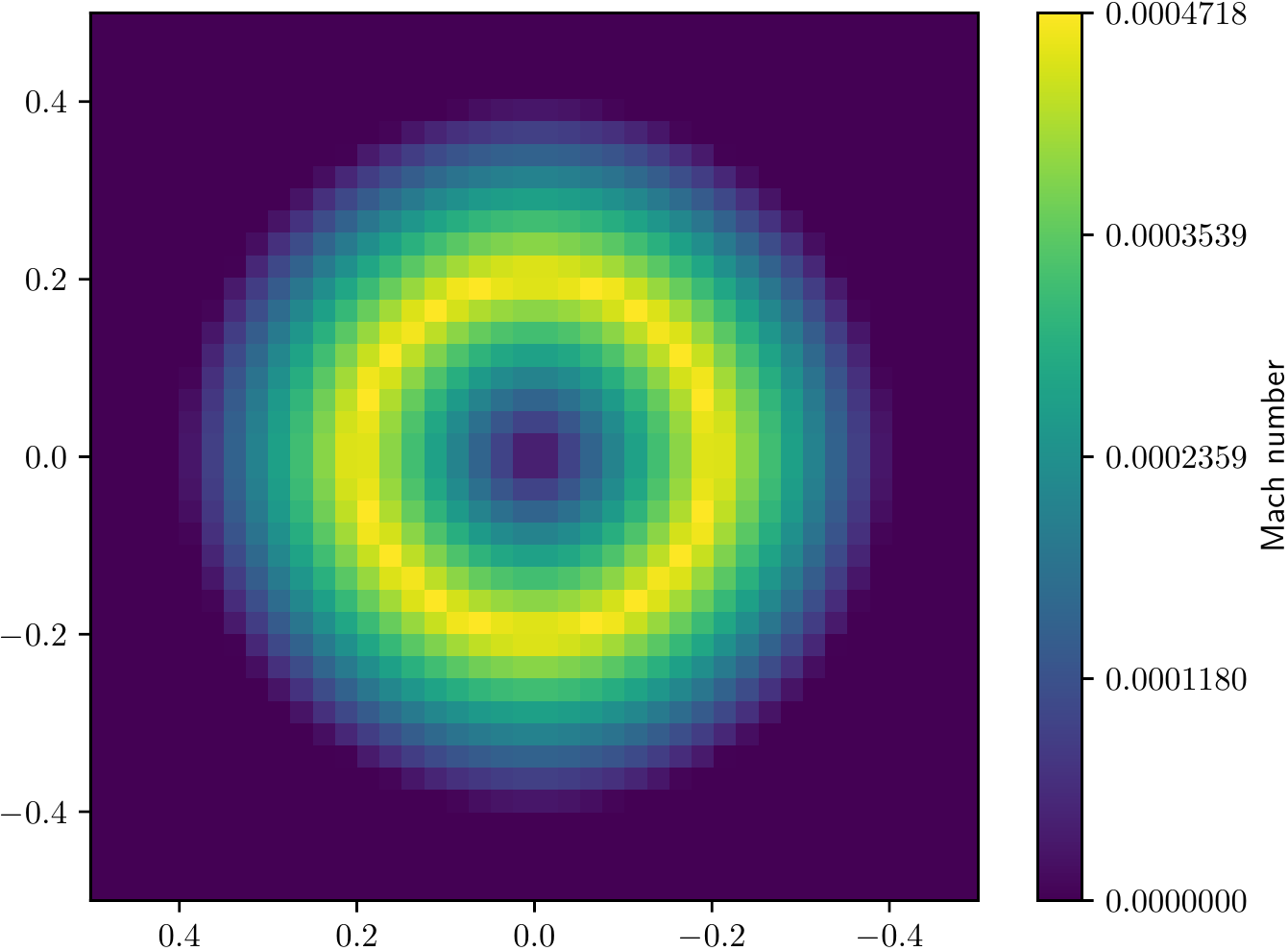}
	\\
	$Ma=10^{-2}$ &$Ma=10^{-3}$
\end{tabular}
    \caption{Comparison of the initial Mach number distribution  for the Gresho vortex test case with $Ma=10^{-2}$ and $Ma=10^{-3}$ obtained with the FSLP method with resolutions $40^2$.}
    \label{greshogravfig3}
\end{figure}

\begin{figure}
    \centering
\begin{tabular}{ccc}
	\includegraphics[width=0.3\linewidth]{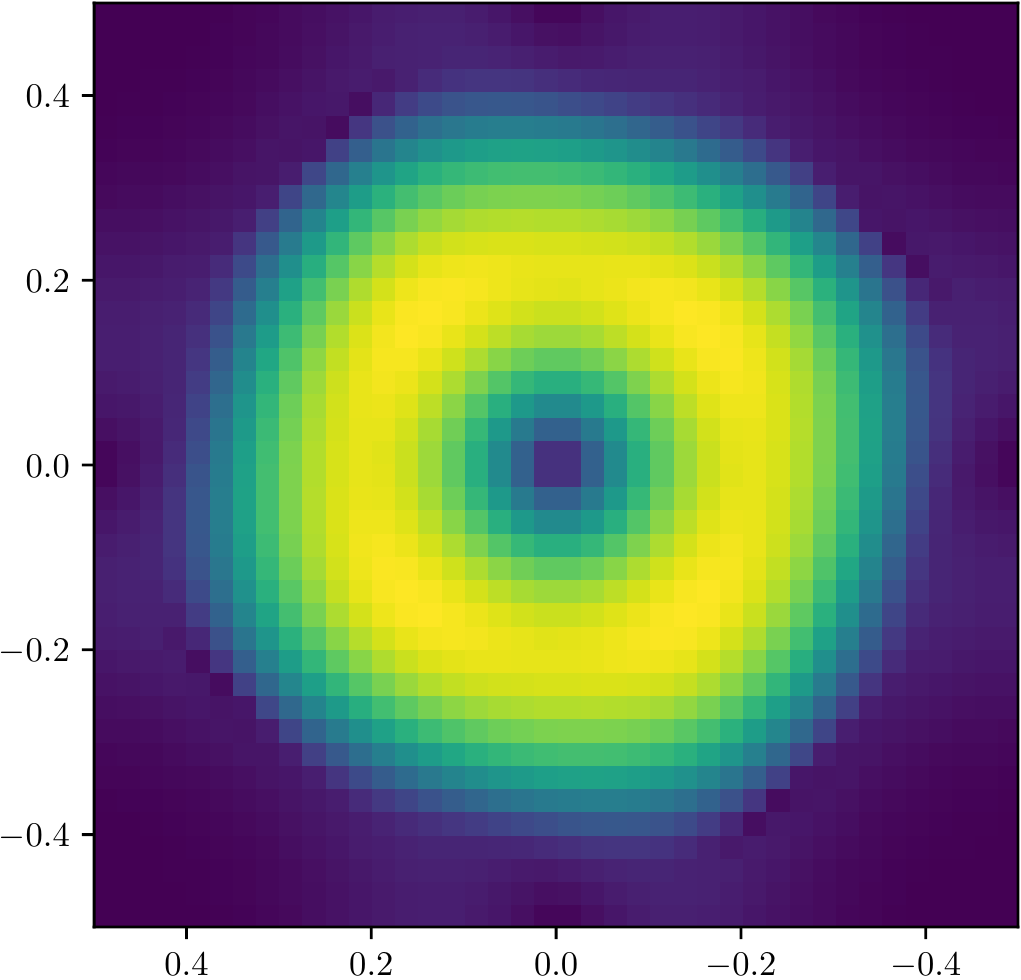}
	&
	\includegraphics[width=0.3\linewidth]{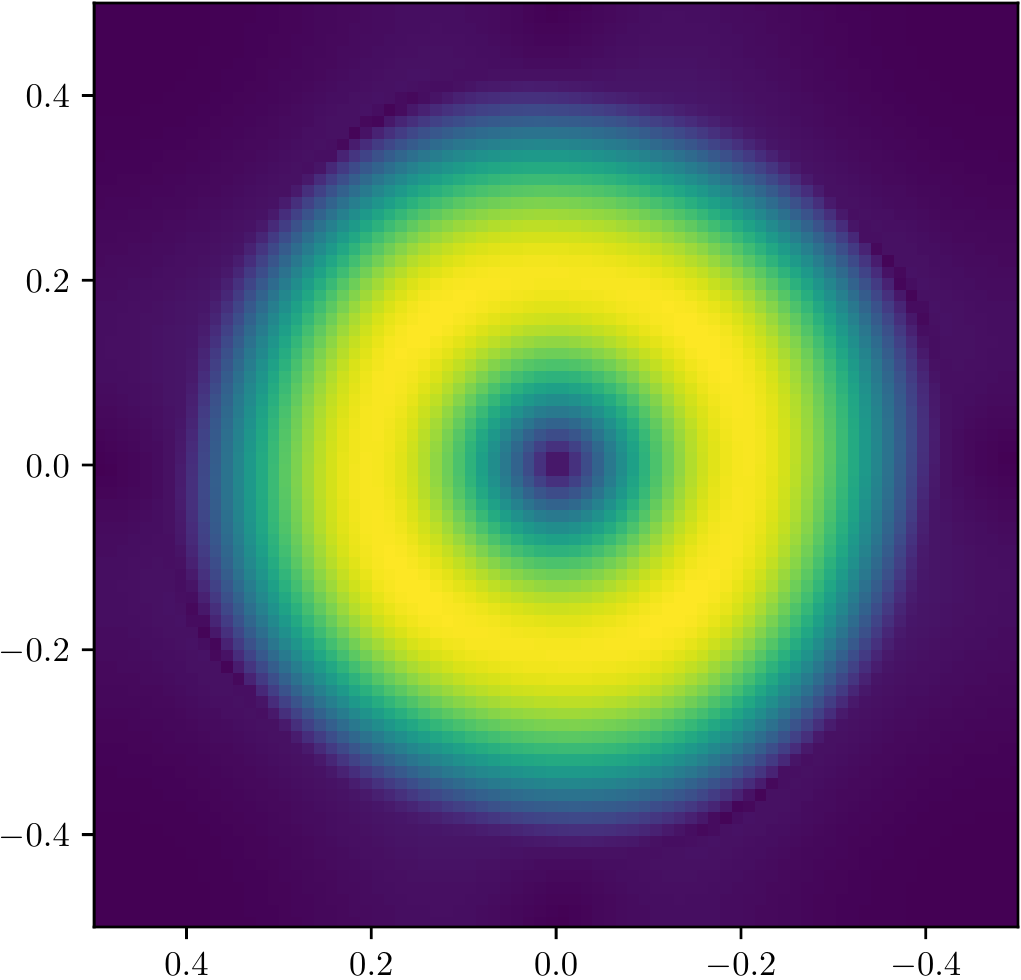}
	&
	\includegraphics[width=0.3\linewidth]{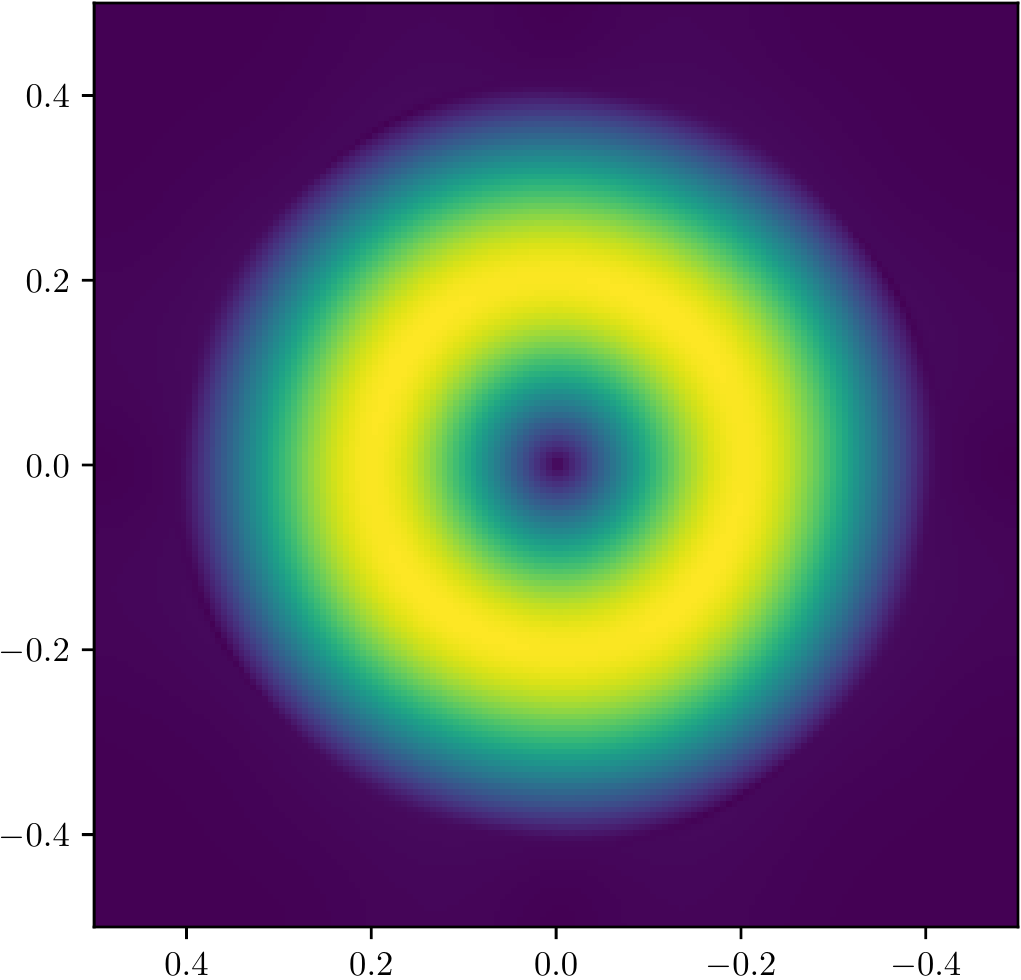}
	\\
	$40^2$ & $80^2$ & $160^2$
	\\
\multicolumn{3}{c}{
	\includegraphics[angle=-90,width=0.5\linewidth]{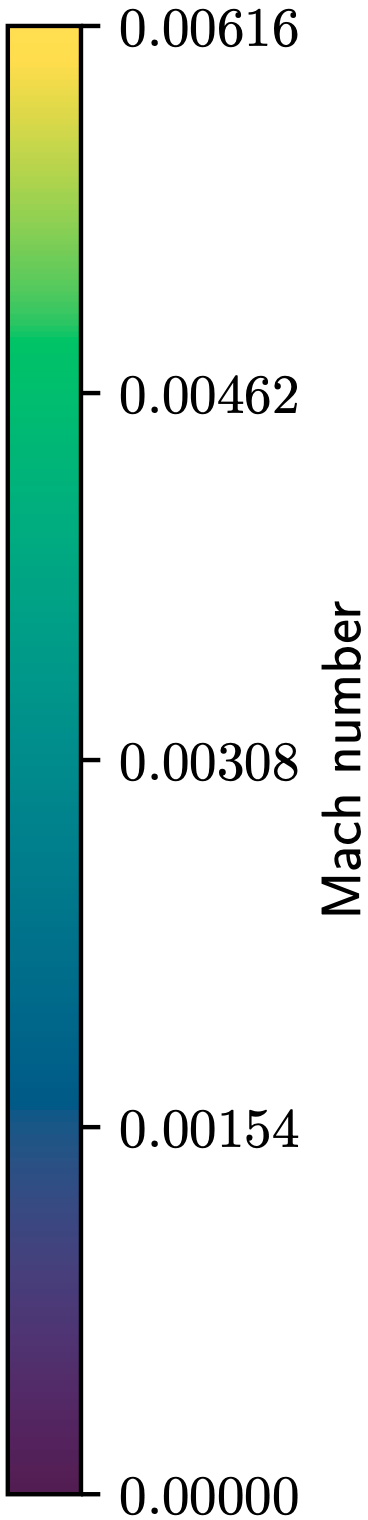}
}
\end{tabular}
    \caption{Comparison of the Mach number distribution for the Gresho vortex test case with $Ma=10^{-2}$ obtained with the FSLP method with resolutions $40^2, 80^2 \ 160^2$ at $t=1s$.}
    \label{greshogravfig1}
\end{figure}

\begin{figure}
    \centering
\begin{tabular}{ccc}
	\includegraphics[width=0.3\linewidth]{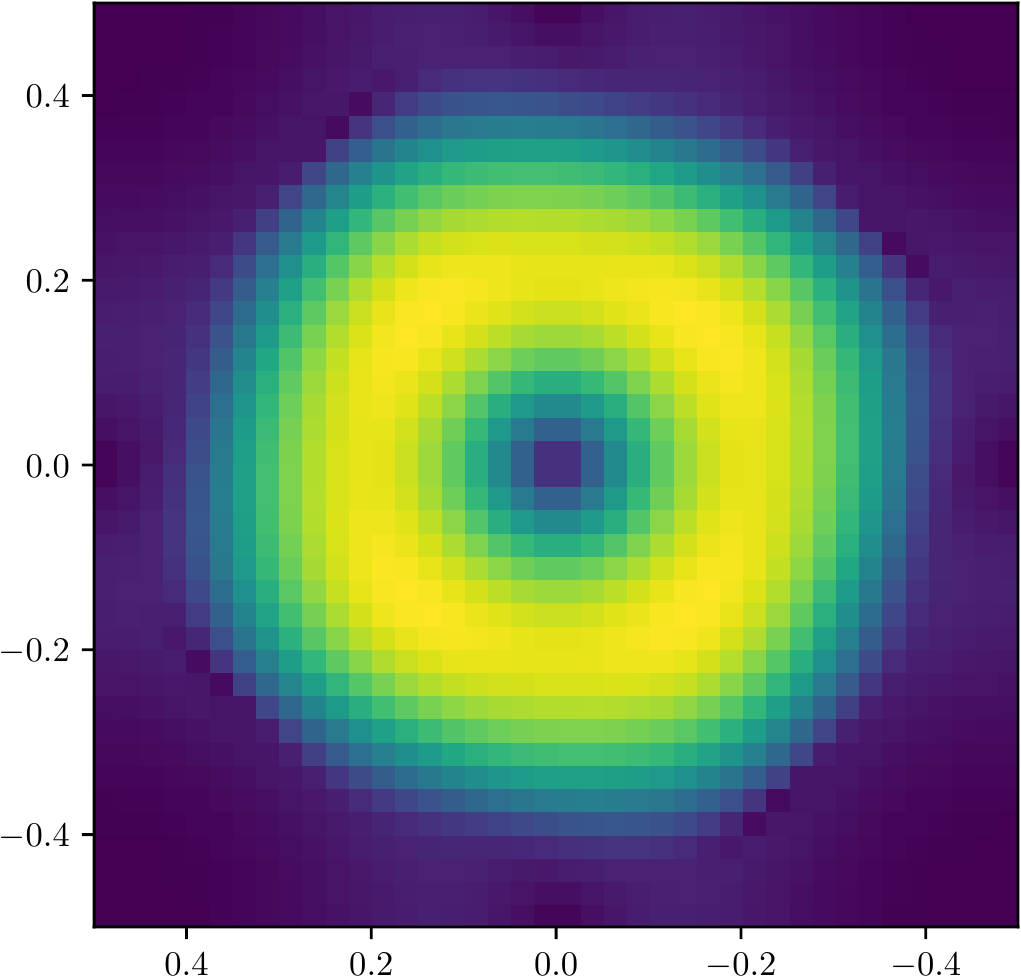}
	&
	\includegraphics[width=0.3\linewidth]{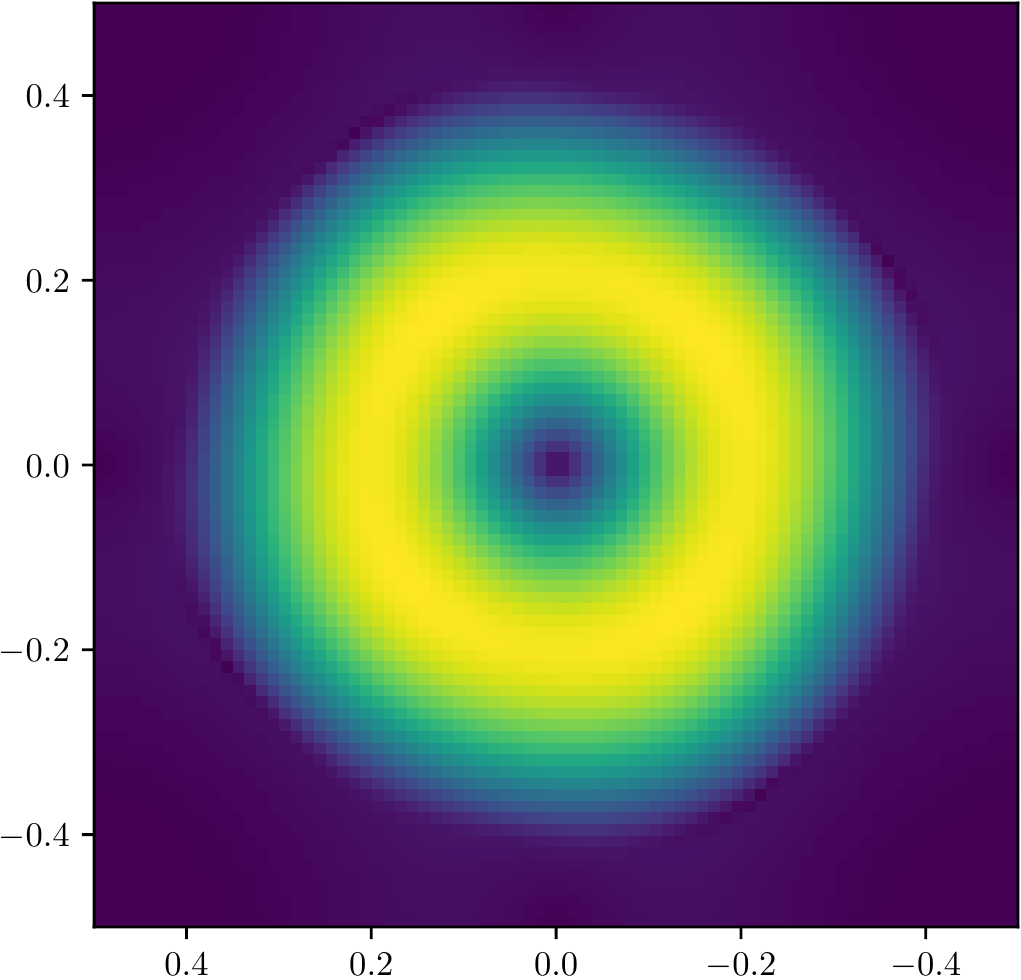}
	&
	\includegraphics[width=0.3\linewidth]{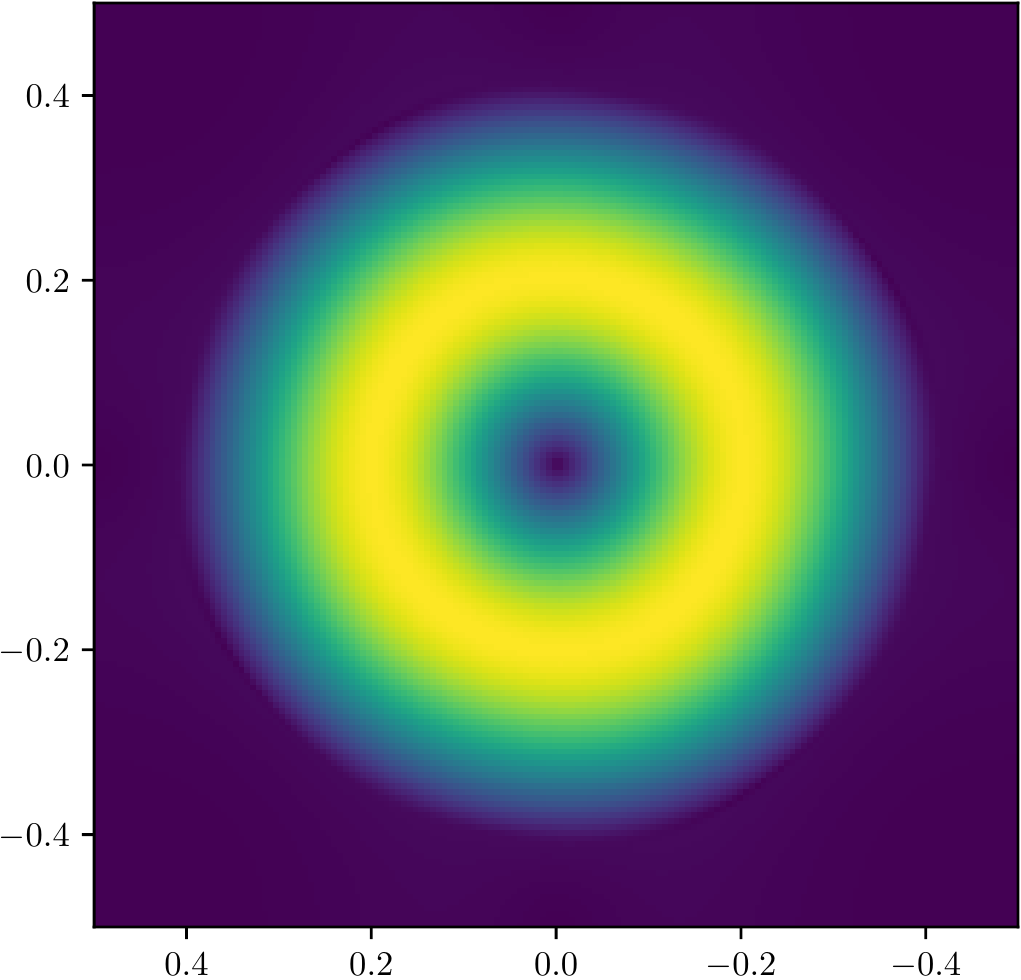}
	\\
	$40^2$ & $80^2$ & $160^2$
	\\
\multicolumn{3}{c}{
	\includegraphics[angle=-90,width=0.5\linewidth]{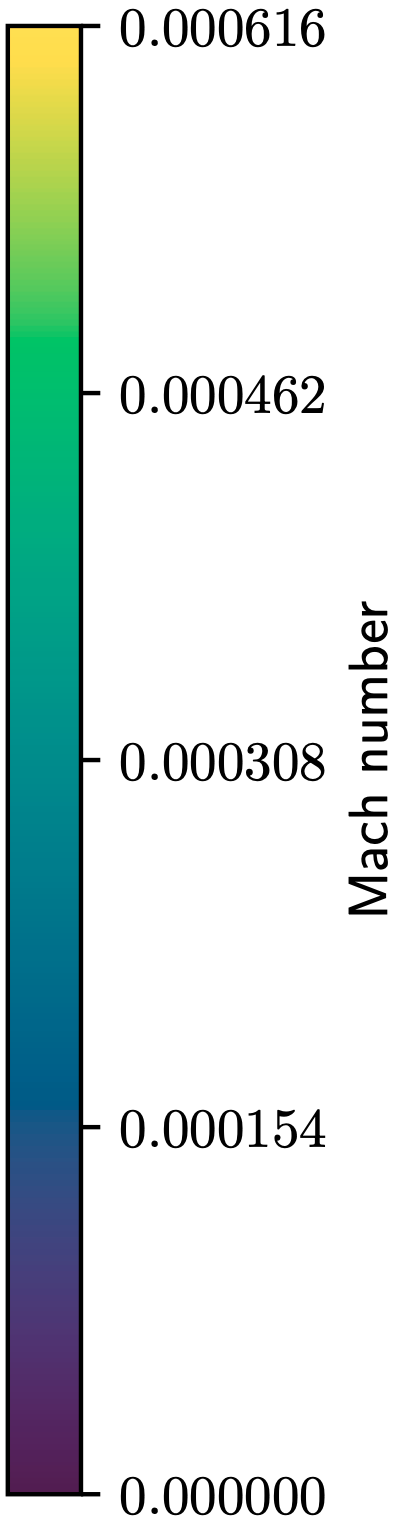}
}
\end{tabular}
    \caption{Comparison of the Mach number distribution for the Gresho vortex test case with $Ma=10^{-3}$ obtained with the FSLP method with resolutions $40^2, 80^2 \ 160^2$ at $t=1s$.}
    \label{greshogravfig2}
\end{figure}

\begin{table}[h]
\caption{Vortex in a gravitational potential test case: evaluation of the ratio kinetic energies $e_{kin}/e_{kin}^0$ at $t=1s$ in the computational domain for different values of the Mach number $Ma$.}
\label{table: greshograv kinetic energy}
\centering
\begin{tabular}{ |c|c|c| }
 \hline
& $Ma=10^{-2}$ & $Ma=10^{-3}$ \\
\hline
$40^2$ & 0.5723 & 0.5727  \\
\hline
$80^2$  & 0.7261 & 0.7258 \\
\hline
$160^2$ & 0.8386 & 0.8388  \\
\hline
\end{tabular}
\end{table}

\subsection{\hlo{Performance comparison: OSLP vs. FSLP}}\label{sect:perf}
\hlo{In this section, we compare the performances of both OSLP and FSLP methods. 
The tests were run on a single Nvidia K80 GPU on a $(512,384,256)$ grid to load the chip's memory fully.
As discussed in section \ref*{sect:cvx}, the relative performances of both methods may vary as 
a function of the Mach number. Indeed, the time step for the FSLP method follows $\Dt\simeq \Dx/(v+c)$ while the OSLP time step 
follows $\Dt=\Dx/max(v,c)$. If $v\simeq c$, the OSLP time step is about two times larger than the FSLP method. However, 
if $c>>v$ (low-Mach regime), both time step coincides. On the other hand, we can expect that a single step of the FSLP method
should be faster than a single step of the OSLP method, as it involves only one kernel instead of two. To illustrate this behavior,
we document two test cases:}
\begin{enumerate}
 \item \hlo{A 3D sod shock tube, to illustrate the $Ma\simeq 1$ behavior,}
 \item \hlo{A 3D gresho vortex, to illustrate the $Ma\simeq 0$ behavior. }
\end{enumerate}
\hlo{Let us investigate the time required for 
both methods to reach a given physical time in both Mach regimes. Table \ref*{table:perf} displays performance results. First,
we note that the OSLP method requires $50\%$ more memory than the FSLP method, as it needs to store the intermediate acoustic states on top 
of the two arrays storing the solution. This allows the FSLP to simulate on a finer grid than the OSLP method given a fixed amount of memory allocted for the computation. 
We also note that one step of the OSLP method requires about $30\%$ more time than the FSLP method, as 
it requires two kernels to be applied successively. Since the Gresho test case is a low-Mach test case, the time step sizes of the OSLP and FSLP methods coincide. As a result,
the FSLP method is $30\%$ faster than the FSLP method. On the other hand, for the Sod problem, the FSLP method requires 117 steps to reach the end time, while the OSLP method only needs 67 steps. As a result, reaching the final time with the FSLP method is $27\%$ longer than with the OSLP method,
despite the FSLP steps being faster. These results give a good idea of the relative efficiency of both methods, but they must be 
mitigated as they are heavily dependent on the implementation and the architecture used. Also, since the FSLP method opens the possiblity for a simple 
2nd order extension, we believe it is still of interest even in the $v\simeq c$ regime. Finally, an implicit-explicit version of 
the FSLP method would likely be competitve with OSLP, as the CFL conditions would coincide. We plan to explore this option in our future work.}

\begin{table}[htbp]
 \centering
 \caption{\hlo{Performance comparison between OSLP and FSLP}} 
 \label{tab:simulation-results}
 \hlo{\begin{tabular}{cccccc}
 \hline
 \textbf{Problem} & \textbf{Method} & \textbf{Steps} & \textbf{Step duration (s)} & \textbf{Total Time (s)} & \textbf{Memory required (MiB)}\\
 \hline
 Gresho & FSLP & 267 & 0.27 (1.0) & 78.2 ( 1.0) & $7.216 \ 10^3$ (1.0) \\
 Gresho & OSLP & 267 & 0.35 (1.31)& 102.2 (1.3) & $1.08 \ 10^4$ (1.5)\\
 \hline
 Sod & FSLP & 117(1.0) & 0.27 (1.0) & 34.9 (1.0) & $7.216 \ 10^3$ (1.0)\\
 Sod & OSLP & 67 (0.57) & 0.35 (1.3) & 25.48 (0.73) & $1.08 \ 10^4$ (1.5) \\
 \hline
 \end{tabular}}
 \label{table:perf}
\end{table}
	
\section*{Reproducing the numerical experiments and figures}
All the simulations shown in this paper were performed with the open source code ARK$^2$-MHD, which  can be found at \url{https://gitlab.erc-atmo.eu/remi.bourgeois/ark-2-mhd/-/tree/test_case_unsplit_paper_%232}. All parameter files and plotting scripts can be found in the folder \url{/test_case_unsplit_paper}.
\section{Conclusion}
We have presented the recasting of an operator splitting Lagrange-Projection solver for gas dynamicss into a corresponding flux-splitting finite-volume method.
This FSLP method is obtained thanks to a simple modification: it only differs from the OSLP method in the states used to compute the transport step.
The method relies on a flux evaluation that separates pressure-related terms from the advection terms in the spirit of \cite{zha1993,deshpande_pvu_1994,Toro2012,borah_novel_2016}.
Two different interpretations of this flux-splitting scheme were proposed to understand better and analyze the resulting method.
First, we showed that the FSLP discretization could be written as a convex combination of two
updated states resulting from approximating two subsystems that respectively account for pressure and advection effects. This approach allowed us to derive the stability properties of the proposed algorithm. Second, we discussed the interpretation of the FSLP method as the result of the discretization of a larger relaxation system that accounts separately for pressure and advection terms within a single step. We showed that the FSLP method is more computationally efficient than the OSLP method \hlo{in the low-Mach regime}.
As a flux-based solver, the resulting FSLP method was straightforwardly extended to multiple dimensions of space and to a high order of accuracy thanks to a standard MUSCL method.

The initial OSLP solver has several interesting numerical advantages:  a well-balanced treatment of the source term and a low Mach fix that provides a uniform truncation error with respect to the Mach number. Both properties were preserved through the recasting process.
The robustness and accuracy of our new flux-splitting method were tested against
a set of benchmark problems, including one and two-dimensional problems, high and low Mach flows
with first and second-order discretizations. The results further confirm the numerical stability of our approach.

In the future, we plan to perform a similar recasting by considering an Implicit-Explicit OSLP solver to prevent the severe CFL limitations imposed by the sound velocity in the low Mach regime. The methods can also be extended to several other flow models like two-phase flow models, magneto-hydrodynamics and the
M1 model for radiative transfer.

\section*{Acknowledgment}
Pascal Tremblin acknowledges support by the European Research Council under Grant Agreement ATMO 757858. Rémi Bourgeois thanks Teddy Pichard for the very fruitful discussions about the convex combination interpretation of the FSLP method.
   	\bibliography{bib}

\appendix



\section{A few classic convexity properties}
We recall hereafter a few classic convexity/concavity properties related to admissible states, entropy, and energy of our flow model that can be found in the literature (see for example \cite{Godlewski2021}). We propose short self-contained proofs of these properties for the sake of completeness.
\begin{lemma}\label{Omega is convex}
	We have the following properties.
\begin{enumerate}[(a)]
	\item
	The function $
	\Lambda:\bU=(\rho,\rho u, \rho E)\in[0,+\infty)\times \mathbb{R}\times [0,+\infty)\mapsto \Lambda(\bU)=(\rho E) - \frac{(\rho u)^2}{2\rho}
	$ is concave.
	\item The set \(\Omega\) defined by \eqref{Def omega} is convex.
	\item The function
	$\specificentropy: (\Tau,u,E)
	\mapsto
	s
	\qty(\Tau,
	E
	-
	 \frac{u^2}{2})
	$ is strictly concave.
	\item The function
	$\entropy:(\rho,\rho\Tau,\rho u,\rho E)
	\mapsto
	\rho s
	\qty(\frac{\rho\Tau}{\rho},
	\frac{(\rho E)}{\rho}
	-
	\frac{(\rho u)^2}{2\rho^2})
	$ is strictly concave.
\end{enumerate}
\end{lemma}
\begin{proof}
Let $\theta_1 = 1-\theta_2 \in [0,1]$ and
\(\bU_k\in[0,+\infty)\times \mathbb{R}\times [0,+\infty)\) for \(k=1,2\), we have
\begin{equation}
	\Lambda\qty(\sum_{k=1,2} \theta_k \bU_k)
	-
	\sum_{k=1,2}
	\theta_k
	\Lambda\qty(\bU_k)
	=
	\sum_{k=1,2}
	\theta_k (\rho E)_k
	+
	\frac{\theta_1\theta_2}{\displaystyle\sum_{k=1,2}\theta_k \rho_k}
	\qty[
	(\rho u)_1\sqrt{\frac{\rho_2}{\rho_1}}
	-
	(\rho u)_2\sqrt{\frac{\rho_1}{\rho_2}}
	]^2
	\geq 0,
\end{equation}
which proves (a). For (b), consider again $\theta_1 = 1-\theta_2 \in [0,1]$
and $\bU_k\in\Omega$, $k=1,2$. If we note $\bU=\sum_{k=1,2}\theta_k\bU_k$, then $\bU\in\Omega$. Indeed, we have that $\sum_{k=1,2}\theta_k\rho_k \geq 0$, and
as $\rho e
=
\Lambda(\bU) \geq \sum_{k=1,2} \theta_k\Lambda(\bU_k)
= \sum_{k=1,2} \theta_k\rho_k  e_k \geq 0$, where $e_k = E_k - (u_k^2)/2 $. This implies that $e\geq 0$.\\

For (c) : The function $K:(u,E)\mapsto E-u^2/2$ is strictly concave and we have
$\specificentropy(\Tau,u,E) = s(\Tau,K(u,E))$. Consider $\lambda\in[0,1]$ and
let us note $\lambda = \lambda_1$ and $\lambda_2=1-\lambda$. We have that
\begin{align}
\specificentropy&\qty(\sum_{k=1,2}\lambda_k\Tau_k,\sum_{k=1,2}\lambda_k u_k,\sum_{k=1,2}\lambda_k E_k)
-
\sum_{k=1,2}\lambda_k
\specificentropy\qty(\Tau_k, u_k, E_k)
\\
&=
s\qty(\sum_{k=1,2}\Tau_k,K\qty(\sum_{k=1,2} \lambda_k u_k,\sum_{k=1,2} \lambda_k E_k))
-
\sum_{k=1,2}\lambda_k
s\qty(\Tau_k, K(u_k, E_k))
\\
&=
s\qty(\sum_{k=1,2}\Tau_k,K\qty(\sum_{k=1,2}\lambda_k u_k,\sum_{k=1,2}\lambda_k E_k))
-
s\qty(\sum_{k=1,2}\Tau_k,\sum_{k=1,2} \lambda_k K\qty(u_k,E_k))
\nonumber
\\
&\qquad\qquad
+
s\qty(\sum_{k=1,2}\Tau_k,\sum_{k=1,2} \lambda_k  K\qty(u_k,E_k))
-
\sum_{k=1,2}\lambda_k
s\qty(\Tau_k, K(u_k, E_k)).
\label{eq: specific entropy ineq 0}
\end{align}
As $K$ is concave, we get that
\begin{equation}
		K\qty(\sum_{k=1,2}\lambda_k u_k,\sum_{k=1,2}\lambda_k E_k) \geq \sum_{k=1,2} \lambda_k  K\qty(u_k,E_k).
\end{equation}
By \eqref{weyl} we know that $e'\mapsto s^\EOS(\overline{\Tau},e')$ is increasing so that
\begin{equation}
		s\qty(\sum_{k=1,2}\Tau_k,K\qty(\sum_{k=1,2} \lambda_k u_k,\sum_{k=1,2} \lambda_k E_k))
		-
		s\qty(\sum_{k=1,2}\Tau_k,\sum_{k=1,2} \lambda_k K\qty(u_k,E_k))
\geq 0.
\label{eq: specific entropy ineq 1}
\end{equation}
We also know that $s$ is concave therefore
\begin{equation}
		s\qty(\sum_{k=1,2}\Tau_k,\sum_{k=1,2} \lambda_k  K\qty(u_k,E_k))
		-
		\sum_{k=1,2}\lambda_k
		s\qty(\Tau_k, K(u_k, E_k))
		\geq 0.
		\label{eq: specific entropy ineq 2}
\end{equation}
By replacing \eqref{eq: specific entropy ineq 1} and \eqref{eq: specific entropy ineq 2} into
\eqref{eq: specific entropy ineq 0} we obtain that
\begin{equation}
\specificentropy\qty(\sum_{k=1,2}\lambda_k\Tau_k,\sum_{k=1,2}\lambda_k u_k,\sum_{k=1,2}\lambda_k E_k)
\geq
\sum_{k=1,2}\lambda_k
\specificentropy\qty(\Tau_k, u_k, E_k) .
\end{equation}

for (d): 	If we note again $\bU = (\rho,\rho u,\rho E)$, by
	\eqref{eq: specific entropy to entropy}, we have
	\begin{equation}
	\entropy(\rho,\rho\Tau,\rho u,\rho E) =
	\rho s
	\qty(\frac{\rho\Tau}{\rho},
	\frac{(\rho E)}{\rho}
	-
	\frac{(\rho u)^2}{2\rho^2})
	=
	S\qty(\rho, \rho\Tau, \rho E - \frac{(\rho u)^2}{2\rho})
	=
	S\qty(\rho, \rho\Tau, \Lambda(\bU))
	.
	\end{equation}
	%
	Now we consider $\bU_k=(\rho_k,\rho_k u_k,\rho_k E_k)\in\Omega$ and
	$\Tau_k \geq 0$, $k=1,2$, we have
	\begin{align}
		&\entropy\qty(
		\sum_{k=1,2}\!\!\!\theta_k \rho_k
		,
		\sum_{k=1,2}\!\!\!\theta_k \rho_k\Tau_k
		,
		\sum_{k=1,2}\!\!\!\theta_k \rho_k u_k
		,
		\sum_{k=1,2}\!\!\!\theta_k \rho_k E_k
		)
		-
		\sum_{k=1,2} \theta_k
		\entropy\qty(\rho_k,\rho_k\Tau_k,\rho_k u_k,\rho_k E_k)
		\nonumber
		\\
		&\qquad
		=
		S\qty(
		\sum_{k=1,2} \theta_k \rho_k
		,
		\sum_{k=1,2}\!\!\!\theta_k \rho_k\Tau_k
		,
		\Lambda\qty(
		\sum_{k=1,2} \theta_k \bU_k
		)
		)
		-
		\sum_{k=1,2} \theta_k
		S\qty(\rho_k,\rho_k\Tau_k,\Lambda(\bU_k))
		\nonumber
		\\
		&\qquad
		=
		S\qty(
		\sum_{k=1,2} \theta_k \rho_k
		,
		\sum_{k=1,2}\!\!\!\theta_k \rho_k\Tau_k
		,
		\Lambda\qty(
		\sum_{k=1,2} \theta_k \bU_k
		)
		)
		-
		S\qty(
		\sum_{k=1,2} \theta_k \rho_k
		,
		\sum_{k=1,2}\!\!\!\theta_k \rho_k\Tau_k
		,
		\sum_{k=1,2} \theta_k
		\Lambda\qty(
		\bU_k
		)
		)
		\nonumber
		\\
		&\qquad\qquad
		+
		S\qty(
		\sum_{k=1,2} \theta_k \rho_k
		,
		\sum_{k=1,2}\!\!\!\theta_k \rho_k\Tau_k
		,
		\sum_{k=1,2} \theta_k
		\Lambda\qty(
		\bU_k
		)
		)
		-
		\sum_{k=1,2} \theta_k
		S\qty(\rho_k,\rho_k\Tau_k,\Lambda(\bU_k))
		\label{eq: intermediate eq}
	\end{align}
	As  $\Lambda$ is concave, we have
		\(
		\Lambda\qty(\sum_{k=1,2} \theta_k \bU_k)
		\geq
		\sum_{k=1,2} \theta_k \Lambda\qty(\bU_k)
		\)
	and as $\mathscr{E}'\mapsto S(\overline{\rho},\overline{\mathscr{V}},\mathscr{E}')$
	is increasing, we have
	\begin{equation}
		S\qty(
		\sum_{k=1,2} \theta_k \rho_k
		,
		\sum_{k=1,2}\!\!\!\theta_k \rho_k\Tau_k
		,
		\Lambda\qty(
		\sum_{k=1,2} \theta_k \bU_k
		)
		)
		-
		S\qty(
		\sum_{k=1,2} \theta_k \rho_k
		,
		\sum_{k=1,2}\!\!\!\theta_k \rho_k\Tau_k
		,
		\sum_{k=1,2} \theta_k
		\Lambda\qty(
		\bU_k
		)
		)
		\geq 0
		.
		\label{eq: intermediate ineq 1}
	\end{equation}
	Using the fact that $S$ is concave, we also get
	\begin{equation}
		S\qty(
		\sum_{k=1,2} \theta_k \rho_k
		,
		\sum_{k=1,2}\!\!\!\theta_k \rho_k\Tau_k
		,
		\sum_{k=1,2} \theta_k
		\Lambda\qty(
		\bU_k
		)
		)
		-
		\sum_{k=1,2} \theta_k
		S\qty(\rho_k,\rho_k\Tau_k,\Lambda(\bU_k))
		\geq 0
		.
		\label{eq: intermediate ineq 2}
	\end{equation}
	Injecting \eqref{eq: intermediate ineq 1} and \eqref{eq: intermediate ineq 2}
	into \eqref{eq: intermediate eq} provides
	\begin{equation}
	\entropy\qty(
	\sum_{k=1,2}\!\!\!\theta_k \rho_k
	,
	\sum_{k=1,2}\!\!\!\theta_k \rho_k\Tau_k
	,
	\sum_{k=1,2}\!\!\!\theta_k \rho_k u_k
	,
	\sum_{k=1,2}\!\!\!\theta_k \rho_k E_k
	)
	\geq
	\sum_{k=1,2} \theta_k
	\entropy\qty(\rho_k,\rho_k\Tau_k,\rho_k u_k,\rho_k E_k)
	.
	\end{equation}

\end{proof}

\section{Approximate Riemann solver for the pressure subsystem}
\label{section: appendix approx riemann solver acoustic}
In this section, we present the derivation of an approximate Riemann solver for the pressure subsystem~\eqref{pressure system}, following the lines of \cite{chalons_all-regime_2016, padioleau2019high}. We express \eqref{pressure system} in the following compact form:
\begin{align}
\dt \bU + 2\dx \acousticop(\bU) &= \sourceterm(\bU)
,&
\dt \Pi + \dx (2a^2 u) &=0
,&
\dt(\rho \Tau)
-2\dx u^P &=0
,&
\dt \phi &=0.
\label{pressure system compact}
\end{align}
where
\(
\acousticop(\bU)^T =
(0,\Pi,\Pi u)
\).
Let $\Delta x_L>0$, $\Delta x_R>0$, we consider $\bar{x}\in\mathbb{R}$ and the following piecewise initial data
\begin{equation}
	\begin{pmatrix}
		\bU,
		\Pi,
		\Tau,
		\phi
	\end{pmatrix}
	(x, t = 0)
	=
	\begin{cases}
		(\bU_L,\Pi_L,\Tau_L,\phi_L)
		& \text{ if \(x \leq \bar{x}\), }
		\\
		(\bU_R,\Pi_R,\Tau_R,\phi_R)
		& \text{ if \(x>\bar{x}\), }
	\end{cases}
\end{equation}
that verifies the equilibrium relations:
\begin{equation}
	(\bU_k,\Pi_k,\Tau_k,\phi_k)
=
\qty[
	(\rho_k,\rho_k u_k, E_k)^T,p^\EOS\qty(\frac{1}{\rho_k},e_k),
\frac{1}{\rho_k},\phi_k
],
\quad
k=L,R
,
\label{eq: equilibrium relations initial conditions RP}
\end{equation}
with
$\phi_L = \frac{1}{\Delta x_L}\int_{-\Delta x_L}^0\phi(\bar{x}+x)\,\dd x$
and  $\phi_R = \frac{1}{\Delta x_R}\int_{0}^{\Delta x_R}\phi(\bar{x}+x)\,\dd x$.
%
%
%
%
We seek a self-similar function $(\appRP,\PiRP,\TauRP,\phiRP)$ composed of four  constant states separated by three discontinuities as follows:
\begin{equation}
	(\appRP,\PiRP,\TauRP,\phiRP)
	\qty(\frac{x-\bar{x}}{t} ; \bU_{L},\Pi_L,\Tau_L,\phi_L, \bU_{R},\Pi_R,\Tau_R,\phi_R)
	=
	\begin{cases}
		(\bU_{L},\Pi_L,\Tau_L,\Phi_L),
		& \text { if  \(\frac{x-\bar{x}}{t} \leq-\frac{2a}{\rho_L}\),}
		\\
		(\bU_{L}^{*},\Pi_L^*,\Tau_L^*,\Phi_L), & \text { if \(-\frac{2a}{\rho_L} <\frac{x-\bar{x}}{t} \leq 0\)},
		\\
		(\bU_{R}^{*},\Pi^*_R,\Tau_R^*,\Phi_R), & \text { if  \(0<\frac{x-\bar{x}}{t} \leq \frac{2a}{\rho_R}\)},
		\\
		(\bU_{R},\Pi_R,\Tau_R,\Phi_R), & \text { if  \(\frac{2a}{\rho_R} <\frac{x-\bar{x}}{t}\)},
	\end{cases}
\label{eq: approx riemann solver def}
\end{equation}
where the intermediate states $\bU^*_k$, $\Pi_k^*$ and $\Tau_k^*$  are required to satisfy the four following properties.
\begin{enumerate}
	\item The approximate Riemann solver should be consistent in the integral sense with the pressure subsystem~\eqref{pressure system compact}: for $\Dt$ such that
	$\frac{2a}{\min(\rho_L,\rho_R)}\Dt < \frac{1}{2}\min(\Dx_{L},\Dx_{R})$, we have
	\begin{equation}
		\begin{bmatrix}
			2\acousticop\left(\bU_{R}\right)
		-
		2\acousticop\left(\bU_{L}\right)
		\\
		2 a^2(u^*_R - u^*_L)
		\\
		-(2 u^*_R - 2 u^*_L)
		\end{bmatrix}
		=
		-\frac{2a}{\rho_L}
		\begin{bmatrix}
		\bU_{L}^{*}-\bU_{L}
		\\
		(\rho\Pi)_{L}^{*}-(\rho\Pi)_{L}
		\\
		(\rho\Tau)_{L}^{*}-(\rho\Tau)_{L}
		\end{bmatrix}
		+
		\frac{2a}{\rho_R}
		\begin{bmatrix}
		\bU_{R}-\bU_{R}^{*}
		\\
		(\rho\Pi)_{R}-(\rho\Pi)_{R}^{*}
		\\
		(\rho\Tau)_{R}-(\rho\Tau)_{R}^{*}
		\end{bmatrix}
		+
		(\Dx_{L}+\Dx_{R})
		\left\{\sourceterm\right\},
		\label{intg form}
	\end{equation}
	with $\left\{ \sourceterm\right\}$ a function that is a consistent approximation of $\sourceterm$, that is to say:
	\begin{equation}
		\lim _{
			\substack{
				\Phi_{L},\Phi_{R}\to\phi(\bar{x})
				\\
				\Delta x_{L}, \Delta x_{R}
				\to 0
				\\
				(\bU_R,\Pi_R), (\bU_L,\Pi_L)
				\to (\bar{\bU},p^\EOS(\bar{\rho},\bar{e}))
				}
		}
		\left\{ \sourceterm \right\}=\sourceterm(\bar{\bU},\phi)(x=\bar{x}).
		\label{derr}
	\end{equation}

	\item In the case $\phi_L = \phi_R$, it should be degenerate to an approximate Riemann for the homogeneous problem obtained with \eqref{pressure system compact} when $\sourceterm=\vb{0}$.

	\item If $\bU_L$ and $\bU_R$ satisfy the following discrete version of the hydrostatic condition~\eqref{hydrostatic continu}:
	\begin{align}
		\Pi_R-\Pi_L
		&=
		 -\frac{\rho_L+\rho_R}{2}(\phi_R - \phi_L)
		 ,&
		 u_L &= u_R =0,
		\label{eqstat}
	\end{align}
	then $(\bU^*_L,\Pi^*_L)=(\bU_L,\Pi_L)$
	and
	$(\bU^*_R,\Pi^*_R)=(\bU_R,\Pi_R)$.
\end{enumerate}
Let us build the states $(\bU^*_R,\Pi^*_R)$ and
$(\bU^*_L,\Pi^*_L)$ so that they verify the above properties. We note
\begin{equation}
	 \Pi^*_R -\Pi^*_L + \Pijump = 0
	 .
	 \label{eq: pi jump rel}
\end{equation}
First, we impose that $\rho_L^*$ and $\rho_R^*$ are consistent with the exact solution of \eqref{pressure system compact} by setting $\rho_L^*=\rho_L$ and $\rho_R^*=\rho_R$. Then we also require that the Rankine-Hugoniot jump conditions obtained in the case $\sourceterm=\vb{0}$
are valid across the waves of velocity
$-2a/\rho_L$ and $+2a/\rho_R$
\begin{align}
	\frac{2a}{\rho_L}
\begin{bmatrix}
		\bU_{L}^{*}-\bU_{L}
		\\
		(\rho\Pi)_{L}^{*}-(\rho\Pi)_{L}
		\\
		(\rho\Tau)_{L}^{*}-(\rho\Tau)_{L}
\end{bmatrix}
	+
\begin{bmatrix}
		2\acousticop\left(\bU_{L}^{*}\right)-2\acousticop\left(\bU_{L}\right)
		\\
		2a^2 u_{L}^{*}-2a^2 u_{L}
		\\
		 -2u_{L}^{*}+ 2u_{L}
\end{bmatrix}
	&=0
	,& -\frac{2a}{\rho_R}
	\begin{bmatrix}
		\bU_{R}-\bU_{R}^{*}
		\\
		(\rho\Pi)_{R}-(\rho\Pi)_{R}^{*}
		\\
		(\rho\Tau)_{R}-(\rho\Tau)_{R}^{*}
	\end{bmatrix}
	+
\begin{bmatrix}
	2\acousticop\left(\bU_{R}\right)-2\acousticop\left(\bU_{R}^{*}\right)
	\\
	2a^2 u_{R}-2a^2 u_{R}^{*}
	\\
	-2u_{R}+2u_{R}^{*}
\end{bmatrix}	&=0.
	\label{eq: RH for +/-a}
\end{align}
Finally, we postulate that the velocity is continuous across the stationary wave by setting
\begin{equation}
	u^*_L = u^*_R = u^*
	,
	\label{eq: u^* continuity}
\end{equation}
and we also impose that $(\Pi u)^*_k = \Pi_k^* u_k^* = \Pi_k^* u^*$, $k=L,R$. Then,
relations~\eqref{intg form}, \eqref{eq: RH for +/-a}, \eqref{eq: pi jump rel} yield
\begin{subequations}
	\begin{align}
		\rho_L^* &= \rho_L
		,&
		\rho_R^* &= \rho_R
		,\\
		E^*_L
		&=
		E_L -\frac{1}{a}\left((\Pi^*+ \frac{\Pijump}{2})u^* - \Pi_Lu_L\right)
		,&
		E^*_R
		&=
		E_R +\frac{1}{a}\left((\Pi^*- \frac{\Pijump}{2})u^* - \Pi_Ru_R\right) ,\\
		u^{*} &=u_{R}^{*}=u_{L}^{*}=\frac{u_{R}+u_{L}}{2}-\frac{1}{2 a}\left(\Pi_{R}-\Pi_{L}\right)-\frac{\Pijump}{2 a}
		,&
		\Pi^{*} &=\frac{\Pi_{R}+\Pi_{L}}{2}-\frac{a}{2}\left(u_{R}-u_{L}\right), \\
		\Pi_{L}^{*} &=\Pi^{*}+\frac{\Pijump}{2}
		,&
		\Pi_{R}^{*} &=\Pi^{*}-\frac{\Pijump}{2},
		\\
		\Tau_L^*&=\frac{1}{\rho_L} + \frac{1}{a}(u^*-u_L)
		,&
		\Tau_R^*&=\frac{1}{\rho_R} - \frac{1}{a}(u^*-u_R),
		\label{eq: Tau_k^* def}
\end{align}
	\label{Solution riemann gravité}
\end{subequations}
where the jump $\Pijump$ can be identifed as
\begin{align}
	\Pijump  &= \frac{\Dx_{L}+\Dx_{R}}{2}\left\{\rho \dx \phi\right\}
	,&
	\Pijump u^* &= \frac{\Dx_{L}+\Dx_{R}}{2}\left\{\mom \dx \phi\right\}.
\end{align}
At this point, the functions
$\left\{\rho \dx \phi\right\}$ and \(\left\{\mom \dx \phi\right\}\)
are still yet to be specified. Let us consider the constraint~3: if it is satisfied then for a state that verifies \eqref{eqstat} the
jumps $\Pijump$ and $\Pijump u^*$ necessarily take the value
\(\Pijump = - (\Pi_R - \Pi_L)\) and $\Pijump u^*=0$ . A simple choice that fulfills this requirement is
\begin{subequations}
	\begin{align}
		\{
		\rho \dx\phi
		\}
		&=
		(\rho_L +\rho_R)
		\frac{\phi_R-\phi_L}{\Delta x_L +\Delta x_R}
		,&
		\{
		\rho \dx\phi
		\}
		=
		(\rho_L +\rho_R)
		u^*
		\frac{\phi_R-\phi_L}{\Delta x_L +\Delta x_R}
		.
		\label{eq: discrete source term definition}
	\end{align}%
\end{subequations}
Relations~\eqref{Solution riemann gravité} and \eqref{eq: discrete source term definition} give a complete definition of the approximate Riemann solver~\eqref{eq: approx riemann solver def}.
This solver yields a definition for the conservative numerical flux
\(\acousticop_\Delta(
	\bU_{L},\Pi_{L},\phi_{L}
	,\bU_{R},\Pi_{R},\phi_{R}
	)
\)
and a source term discretization (located at the interface)
\(\sourceterm_\Delta(
	\bU_{L},\Pi_{L},\phi_{L}
	,\bU_{R},\Pi_{R},\phi_{R}
	)
\)
thanks to the consistency in the integral sense. We get
\begin{subequations}
	\begin{align}
		\acousticop_\Delta(
			\bU_{L},\Pi_{L},\phi_{L}
			,\bU_{R},\Pi_{R},\phi_{R}
			)
	&=
	\frac{\acousticop(\bU_R,\Pi_R)+\acousticop(\bU_L,\Pi_L)}{2}
	-\frac{a}{2\rho_L}(\bU_L^* - \bU_L)
	-\frac{a}{2\rho_R}(\bU_R - \bU_R^*)
	,\\
	\sourceterm_\Delta(
		\bU_{L},\Pi_{L},\phi_{L}
		,\bU_{R},\Pi_{R},\phi_{R}
		)
	&=
	[0,-\{\rho\dx\phi\},-\{\rho u\dx\phi\}]^T
	,
	\end{align}%
\end{subequations}
so that for two neighbouring states $(\bU_{j}^n,\Pi_{j}^n,\phi_{j})$
and $(\bU_{j+1}^n,\Pi_{j+1}^n,\phi_{j+1})$ across the cell interface
$j+1/2$ that separates the cell $j$ and the cell $j+1$, the numerical
conservative flux $(0,\Pi^*_{j+1/2},\Pi^*_{j+1/2} u^*_{j+1/2})$ is defined by
\begin{equation}
	(0,\Pi^*_{j+1/2},\Pi^*_{j+1/2} u^*_{j+1/2}) =
	\acousticop_\Delta(
		\bU^n_{j},\Pi^n_{j},\phi_{j}
		,\bU^n_{j+1},\Pi^n_{j+1},\phi_{j+1}
		),
		\label{pressure system flux definition}
\end{equation}
and the discrete souce term $\sourceterm_{j}$ within the cell $j$ is
given by
\begin{align}
	\sourceterm_{j}
&=
\frac{\Delta x_{j+1/2}}{2\Delta x_j} \sourceterm_{j+1/2}
+
\frac{\Delta x_{j-1/2}}{2\Delta x_j} \sourceterm_{j-1/2}
	,&
	\sourceterm_{j+1/2}
	&=
	\sourceterm_\Delta(
		\bU^n_{j},\Pi^n_{j},\phi_{j}
		,\bU^n_{j+1},\Pi^n_{j+1},\phi_{j+1}
		)
	.
\end{align}


Let us now give some properties of the approximate Riemann solver.
Let us note
$
e_k^* = E_k^* - (u_k^*)^2/2
$
,
the following lemma is a direct consequence of \eqref{eq: RH for +/-a} that exhibits
 a reminiscent property associated with the Riemann invariants associated of
 the system~\eqref{pressure system compact} when $\sourceterm=\vb{0}$.
\begin{lemma}
\begin{align}
	e_k^* - \frac{(\Pi_k^*)^2}{2a^2}
	&=
	e_k - \frac{(\Pi_k)^2}{2a^2}
	,&
	\Tau_k^* + \frac{\Pi_k^*}{a}
	&=
	\Tau_k + \frac{\Pi_k}{a}
	,&
	k&=L,R
	.
\label{eq: riemann invariants across acoustic waves}
\end{align}
\end{lemma}

The following positivity result is a direct consequence of \eqref{eq: Tau_k^* def}.
\begin{proposition}\leavevmode
\begin{enumerate}
	\item If $a$ is chosen large enough then $\Tau_L^*>0$ and $\Tau_R^*>0$.
	\item $\Tau_L^*>0$ and $\Tau_R^*>0$ is equivalent to
	$u_L -a\Tau_L = u_L -a/\rho_L < u^* < u_R +a\Tau_R = u_R +a/\rho_R$.
\end{enumerate}
\end{proposition}
%

Following the lines of \cite{Chalons2016}, we first prove two preliminary stability-related results. The differences from Lemma 1 of \cite{Chalons2016} is that the Riemann states we are dealing with here depend on the $\Pijump$ terms and that the specific volume we use is $\Tau$ instead of ${1}/{\rho}$ (that are different in the sub-system framework).  However, the proof turns out to be almost identical.

\begin{proposition}
	Consider the intermediate states defined by \eqref{Solution riemann gravité}.

	 and noting $s_k=s^{\EOS}(\Tau_k,s_k)$, we have

	\begin{equation}
		e_{k}^{*}-e^{\EOS}\left(\Tau_{k}^{*}, s_{k}\right)-\frac{\left(p^{\EOS}\left(\Tau_{k}^{*}, s_{k}\right)-\Pi^{*}_k\right)^{2}}{2 a^{2}} \geq 0,
		\label{lemme 1}
	\end{equation}
	with $e^*_k = E^*_k - \frac{{u^*_k}^2}{2}$.
\end{proposition}

\begin{proof}
	We only describe the case $k=R$. Consider the function:
		\begin{multline}
			\chi(\Tau)=e^{\EOS}\left(\Tau, s_{R}\right)-\frac{p^{\EOS}\left(\Tau, s_{R}\right)^{2}}{2 a^{2}}-e^{\EOS}\left(\Tau_{R}^{*}, s_{R}\right) +\frac{p^{\EOS}\left(\Tau_{R}^{*}, s_{R}\right)^{2}}{2 a^{2}} \\
			+p^{\EOS}\left(\Tau_{R}^{*}, s_{R}\right)\left(\Tau+\frac{p^{\EOS}\left(\Tau, s_{R}\right)}{a^{2}}-\Tau_{R}^{*}-\frac{p^{\EOS}\left(\Tau_{R}^{*}, s_{R}\right)}{a^{2}}\right) .
		\end{multline}

	One can check that $\chi'(\Tau)=\left(p^{\EOS}\left(\Tau_{R}^{*}, s_{R}\right)-p^{\EOS}\left(\Tau, s_{R}\right)\right)\left(1-\rho^{2} c^{2}\left(\Tau, s_{R}\right) / a^{2}\right)$. We have $\partial_\Tau p< 0$ from \ref{weyl}, we also assume that $a$ is large enough. We have two different cases:
	\begin{equation}
		\begin{aligned}
			\Tau^*_R < \Tau < \Tau_R & \implies  \chi'(\Tau) > 0 & \implies  \chi(\Tau^*_R)<\chi(\Tau)<\chi(\Tau_R).\\
			\Tau^*_R > \Tau > \Tau_R& \implies   \chi'(\Tau) < 0&\implies \chi(\Tau^*_R)<\chi(\Tau)<\chi(\Tau_R).
		\end{aligned}
	\end{equation}

	As $\chi(\Tau_R^*) =0$, we have $\chi(\Tau_R)>0$ , in both cases.
	Accounting for \eqref{eq: riemann invariants across acoustic waves}, we get
	\begin{align}
	0<
	\chi(\Tau_R) &=  e^{\EOS}\left(\Tau_R, s_{R}\right)-\frac{p^{\EOS}\left(\Tau_R, s_{R}\right)^{2}}{2 a^{2}}-e^{\EOS}\left(\Tau_{R}^{*}, s_{R}\right) +\frac{p^{\EOS}\left(\Tau_{R}^{*}, s_{R}\right)^{2}}{2 a^{2}}
	\nonumber
	\\
	&\qquad\qquad
	+p^{\EOS}\left(\Tau_{R}^{*}, s_{R}\right)\left(\Tau_R+\frac{p^{\EOS}\left(\Tau_R, s_{R}\right)}{a^{2}}-\Tau_{R}^{*}-\frac{p^{\EOS}\left(\Tau_{R}^{*}, s_{R}\right)}{a^{2}}\right)
	\nonumber
	\\
	&=
	e_R^*-\frac{\qty(\Pi_R^*)^{2}}{2 a^{2}}-e^{\EOS}\left(\Tau_{R}^{*}, s_{R}\right) +\frac{p^{\EOS}\left(\Tau_{R}^{*}, s_{R}\right)^{2}}{2 a^{2}}
	+p^{\EOS}\left(\Tau_{R}^{*}, s_{R}\right)
	\left(
		\frac{\Pi_R^*}{a^{2}}-\frac{p^{\EOS}\left(\Tau_{R}^{*}, s_{R}\right)}{a^{2}}
		\right)
	\nonumber
	\\
	&=   e_{R}^{*}-e^{\EOS}\left(\Tau_{R}^{*}, s_{R}\right)-\frac{\left(p^{\EOS}\left(\Tau_{R}^{*}, s_{R}\right)-\Pi^{*}_R\right)^{2}}{2 a^{2}}
	.
	\end{align}
	Similar lines can be used for $k=L$.
\end{proof}

We present a result concerning the behavior of the numerical scheme in the low Mach regime defined in
section~\ref{section: low mach}: we consider a one-dimensional smooth solution of
the pressure subsystem~\eqref{pressure system compact}
such that \(\partial_{\tilde{x}}\tilde{p} + \tilde{\rho}\widetilde{(\dx\phi)} = O(\mach^2)\).
Then, we proceed as in \cite{chalons_all-regime_2016} by evaluating the truncation
error (in the sense of the Finite Difference) obtained by substituting these low Mach flow parameters into
the finite volume update formula derived from the fluxes \eqref{pressure system flux definition}.
We obtain the following results.
 \begin{proposition} In the low Mach regime, the rescaled discretization of the pressure system is consistent with
\begin{align}
	\partial_{\tilde{t}} \tilde{\rho}&=0, & \partial_{\tilde{t}} (\tilde{\rho}\tilde{u})+\frac{1}{\mach^{2}} (\partial_{\tilde{x}}\tilde{p} + \tilde{\rho}\widetilde{(\dx\phi)} )
		&=
		O(\Delta \tilde{t})+O\left(\frac{\Delta \tilde{x}}{\mach}\right)
	,&
	\partial_{\tilde{t}} (\tilde{\rho}\tilde{E})+\partial_{\tilde{x}}(\tilde{p} \tilde{u}) &=O(\Delta \tilde{t})+O(\mach\Delta \tilde{x}).
	\label{pressure subsystem low mach equivalent eq}
\end{align}
 \end{proposition}
If one performs a similar evaluation for the full FSLP scheme, one can see that the truncation error term $O\left(\frac{\Delta \tilde{x}}{\mach}\right)$ that appears in the momentum equation of
\eqref{pressure subsystem low mach equivalent eq} will be the only error term whose magnitude is not uniform with respect to $\mach$. Similarly as in \cite{Dellacherie2010, chalons_all-regime_2016, padioleau2019high, Dellacherie2016}, this
truncation error term can be traced back to the non-centered part of $\Pi_{j+1/2}^*$. To tackle this issue, we adopt the modification used in \cite{chalons_all-regime_2016,padioleau2019high} by replacing
$\Pi_{j+1/2}^*$ with
\begin{equation}
	\Pi_{j+1 / 2}^{*, \theta}=\frac{1}{2}\left(\Pi_{j}^{n}+\Pi_{j+1}^{n}\right)-\theta_{j+1 / 2} \frac{a_{j+1 / 2}}{2}\left(u_{j+1}^{n}-u_{j}^{n}\right),
	\label{p^* corrigé}
\end{equation}
where $\theta_{j+1/2}\in[0,1]$. This results in the update relation~\eqref{acoustic step} that is a finite approximation of \eqref{pressure system compact} with the flux definition~\eqref{u*p*}. We will see in
 \ref{section: appendix modified approx riemann solver acoustic} how this resulting modified flux can still be associated with an Approximate Riemann solver.



\section{All-regime approximate Riemann solver for the pressure subsystem}
\label{section: appendix modified approx riemann solver acoustic}
Following similar lines as in \cite{chalons_all-regime_2016}: although the modified pressure scheme~\eqref{u*p*} is defined as a flux scheme, it is possible  to find
an approximate Riemann solver \((\appRP^\theta,\PiRP^\theta,\TauRP^\theta)\)
that enables to retrieve the numerical flux
\({\acousticop}_{j+1/2}^\theta = (0,\Pi^{*,\theta}_{j+1/2},\Pi^{*,\theta}_{j+1/2} u^*_{j+1/2})\). We suppose that  \((\appRP^\theta,\PiRP^\theta,\TauRP^\theta)\)
has the same structure as  \((\appRP,\PiRP,\TauRP)\), we consider
\begin{equation}
	(\appRP^\theta,\PiRP^\theta,\TauRP^\theta,\phiRP)
	\qty(\frac{x-\bar{x}}{t} ; \bU_{L},\Pi_L,\Tau_L,\phi_L, \bU_{R},\Pi_R,\Tau_R,\phi_R)
	=
	\begin{cases}
		(\bU_{L},\Pi_L,\Tau_L,\Phi_L),
		& \text { if  \(\frac{x-\bar{x}}{t} \leq-\frac{a}{\rho_L}\),}
		\\
		(\bU_{L}^{*,\theta},\Pi_L^{*,\theta},\Tau_L^{*,\theta},\Phi_L), & \text { if \(-\frac{a}{\rho_L} <\frac{x-\bar{x}}{t} \leq 0\)},
		\\
		(\bU_{R}^{*,\theta},\Pi^{*,\theta}_R,\Tau_R^{*,\theta},\Phi_R), & \text { if  \(0<\frac{x-\bar{x}}{t} \leq \frac{a}{\rho_R}\)},
		\\
		(\bU_{R},\Pi_R,\Tau_R,\Phi_R), & \text { if  \(\frac{a}{\rho_R} <\frac{x-\bar{x}}{t}\)},
	\end{cases}
\label{eq: modified approx riemann solver def}
\end{equation}
where $\Pi_k$, $\Tau_k$ and $\Phi_k$
verify \eqref{eq: equilibrium relations initial conditions RP}, $k=L,R$.
The states $(\bU_{k}^{*,\theta}, \Pi^{*,\theta}_k, \Tau_k^{*,\theta})$, $k=L,R$ are yet to be defined.
First, we impose that \((\appRP,\PiRP,\TauRP)\) verifies the consistency in the integral sense
\begin{equation}
	\begin{bmatrix}
	\frac{1}{\alpha}\acousticop\left(\bU_{R}\right)
	-
	\frac{1}{\alpha}\acousticop\left(\bU_{L}\right)
	\\
	\frac{1}{\alpha} a^2(u_R - u_L)
	\\
	-(\frac{1}{\alpha} u_R - \frac{1}{\alpha} u_L)
	\end{bmatrix}
	=
	-\frac{a}{\alpha\rho_L}
	\begin{bmatrix}
	\bU_{L}^{*,\theta}-\bU_{L}
	\\
	(\rho\Pi)_{L}^{*,\theta}-(\rho\Pi)_{L}
	\\
	(\rho\Tau)_{L}^{*,\theta}-(\rho\Tau)_{L}
	\end{bmatrix}
	+
	\frac{a}{\alpha \rho_R}
	\begin{bmatrix}
	\bU_{R}-\bU_{R}^{*,\theta}
	\\
	(\rho\Pi)_{R}-(\rho\Pi)_{R}^{*,\theta}
	\\
	(\rho\Tau)_{R}-(\rho\Tau)_{R}^{*,\theta}
	\end{bmatrix}
	+
	\frac{\Dx_{L}+\Dx_{R}}{2}
	\begin{bmatrix}
		\frac{1}{\alpha}\left\{\sourceterm\right\}
		\\
		0
		\\
		0
	\end{bmatrix}.
	\label{intg form theta}
\end{equation}
We then enforce that the numerical flux resulting from \eqref{intg form theta}
is $\acousticop_\Delta^{\theta}$, which boils down to require that
	\begin{equation}
	\begin{bmatrix}
			\frac{1}{\alpha}\acousticop_\Delta^\theta
				\\
			\frac{1}{\alpha}a^2 u_\Delta^\theta
			\\
			-\frac{1}{\alpha}u_\Delta^\theta
	\end{bmatrix}	\!\!
			(
				\bU_{L},\Pi_{L},\phi_{L}
				,\bU_{R},\Pi_{R},\phi_{R}
			)
	\! = \!
	\begin{bmatrix}
	\acousticop(\bU_R,\Pi_R)\!+\!\acousticop(\bU_L,\Pi_L)
	\\
	a^2u_R + a^2 u_L
	\\
	-(u_R + u_L)
	\end{bmatrix}
	\!\!
 	-\frac{a}{\rho_L}\!\!
\begin{bmatrix}
		\bU_L^{,\theta} - \bU_L
		\\
		(\rho\Pi)_L^{,\theta} - (\rho\Pi)_L
		\\
		(\rho\Tau)_L^{,\theta} - (\rho\Tau)_L
\end{bmatrix}
	\!\!
	-\frac{a}{\rho_R}\!\!
\begin{bmatrix}
		\bU_R - \bU_R^{,\theta}
		\\
		(\rho\Pi)_R - (\rho\Pi)_R^{,\theta}
		\\
		(\rho\Tau)_R - (\rho\Tau)_R^{,\theta}
\end{bmatrix}
.
\label{eq: flux eq theta}
\end{equation}
Choosing $\rho_k^{*,\theta} = \rho_k$, $k=L,R$, relation~\eqref{intg form theta}
and \eqref{eq: flux eq theta} provide a linear system with respect to
$u_k^{*,\theta}$, $\Pi_k^{*,\theta}$, $\Tau_k^{*,\theta}$ and $E_k^{*,\theta}$, $k=1,2$
whose solution is
\begin{subequations}
\begin{align}
		\rho_L^*&=\rho_L
		,&
		\rho_R^*&=\rho_R
		,\\
		E^{*,\theta}_L & = E_L^* -(1-\theta)\frac{u_R-u_L}{2}u^*
		, &
		E^{*,\theta}_R & = E_R^* +(1-\theta)\frac{u_R-u_L}{2}u^*
		, \\
		u^{*,\theta}_L &= u^* - (1-\theta)\frac{u_R-u_L}{2}
		,&
		u^{*,\theta}_R &= u^* + (1-\theta)\frac{u_R-u_L}{2}
		,\\
		\Pi_{L}^{*,\theta} &=\Pi_{L}^{*}
		, &
		\Pi_{R}^{*,\theta} &=  \Pi_{R}^{*}
		, \\
		\Tau_{L}^{*,\theta} &=\Tau_L^*
		, &
		\Tau_{R}^{*,\theta} &=\Tau_R^*.
\end{align}
\label{Solution riemann gravité all-regime}
\end{subequations}
We now turn to positivity-preserving related properties. Let us note
$e_k^{*,\theta}= E_k^{*,\theta} - {u_k^{*,\theta}}^2/2$, we have the following result.
	\begin{proposition}
		Assuming again that $a$ is large enough, we have
		\begin{equation}
			e_{k}^{*,\theta}-e^{\EOS}\left(\Tau_{k}^{*,\theta}, s_{k}\right)
			-\frac{\left(p^{\EOS}\left(\Tau_{k}^{*,\theta}, s_{k}\right)-\Pi^{*,\theta}_k\right)^{2}}{2 a^{2}}
			 +
			 \frac{(1-\theta)^2(u_R-u_L)^2}{8}\geq 0.
			\label{lemme 2}
		\end{equation}
	\end{proposition}
	\begin{proof}
	Let us consider the case $k=R$, by \eqref{Solution riemann gravité all-regime}
	we get
\begin{align}
		e_R^{*,\theta}- e_R^{*} &=  E_R^{*,\theta}-E_R^{*} -\frac{1}{2}( {u_R^{*,\theta}}^2 - {u_R^{*}}^2 )
		\nonumber
		\\
		&=(1-\theta)\frac{u_R-u_L}{2} u^*
		-
		\frac{1}{2}
		\qty(
			(u^{*})^2 +u^{*}(1-\theta)(u_R-u_L)
			+(1-\theta)^2 \frac{(u_R-u_L)^2}{4}
			 - (u^{*})^2
		)
		\nonumber
		\\
		&=-\frac{1}{8}(1-\theta)^2 (u_R-u_L)^2.
\end{align}
Using \eqref{lemme 1}, we obtain
\begin{multline}
	e_{R}^{*,\theta}-e^{\EOS}\left(\Tau_{R}^{*,\theta}, s_{R}\right)
	=
	e_{R}^{*,\theta}
	-
	e_{R}^{*}
	+
	e_{R}^{*}
	-
	e^{\EOS}\left(\Tau_{R}^{*,\theta}, s_{R}\right)
	=
	-\frac{1}{8}(1-\theta)^2 (u_R-u_L)^2
	+
	e_{R}^{*}
	-
	e^{\EOS}\left(\Tau_{R}^{*,\theta}, s_{R}\right)
	\\
	\geq		\nonumber
	-\frac{1}{8}(1-\theta)^2 (u_R-u_L)^2
	+
	\frac{
		\qty(
			p^{\EOS}\left(\Tau_{k}^{*,\theta}, s_{k}\right)-\Pi^{*,\theta}_k
		)^{2}}{2 a^{2}}.
\end{multline}
Similar lines can be used for the case $k=L$.
\end{proof}
The relation~\eqref{lemme 2} highlights the role of the inequality
\begin{equation}
	\frac{1}{2 a^{2}}\left(p^{\EOS}(\Tau_{k}^{*, \theta}, s_{k})-\Pi^{*}_k\right)^{2}
	-
	\frac{(1-\theta)^{2}\left(u_{R}-u_{L}\right)^{2}}{8} \geq 0, \quad k=L, R
	\label{inegalité sur Co}
\end{equation}
in obtaining stability properties for the modified scheme. We have the following proposition.
\begin{proposition}\label{prop: e*>0 and s*>s modified scheme}
Let us note: $s_k^{*,\theta} = s^\EOS(\Tau_k^{*,\theta},e_k^{*,\theta})$, if \eqref{inegalité sur Co} is satisfied, then
\begin{itemize}
\item the modified approximate Riemann solver~\eqref{eq: modified approx riemann solver def} preserves the positivity of the internal energy, that is to say: $e_k^{*,\theta}>0$, $k=R,L$,
\item the modified approximate Riemann solver~\eqref{eq: modified approx riemann solver def} verifies
$s_k^{*,\theta} \geq s_k $, $k=R,L$,
\item the modified approximate Riemann solver~\eqref{eq: modified approx riemann solver def} is entropy satisfying in the sense that
\begin{equation}
-a(s^{*,\theta}_L -s_L)
+a(s_R - s^{*,\theta}_R )
\geq 0
.
\label{eq: approx riemann solver entropy inequality}
\end{equation}
\end{itemize}
\end{proposition}
\begin{proof}
If \eqref{prop: e*>0 and s*>s modified scheme} is satisfied, then
$e_{k}^{*,\theta} \geq e^{\EOS}\left(\Tau_{k}^{*,\theta}, s_{k}\right)$. By the assumption on the EOS, we have that $e_{k}^{*,\theta}>0$. Now, considering a fixed $\overline{\Tau}>0$, by \eqref{weyl} we know that $e'\mapsto s^\EOS(\overline{\Tau},e')$ is increasing, thus we deduce that
$ s^\EOS(\Tau_k^{*,\theta},e^\EOS(\Tau_k^{*,\theta},s_k^{*,\theta})) = s_k^{*,\theta}
\geq
s^\EOS(\Tau_k^{*,\theta},e^\EOS(\Tau_k^{*,\theta},s_k)) = s_k
$, $k=L,R$. This implies \eqref{eq: approx riemann solver entropy inequality}.
\end{proof}



\section{Eigenstructure of the off-equilibrium~(\ref{eq: relaxed system param}\textsubscript{\(\relaxparam\!=\!0\)})}\label{section: eigenstructure analysis}
We propose in this section to study the eigenstructure of the
relaxation system~(\ref{eq: relaxed system param}\textsubscript{\(\relaxparam\!=\!0\)}).
Let us first express the acoustic part of (\ref{eq: relaxed system param}\textsubscript{\(\relaxparam\!=\!0\)}) using
a change of variables: accounting for
$e^P = E^P - (u^P)^2/2$, the evolution equations for $E^P$, for $\Pi^P$ and $\Tau^P$
in {\renewcommand{\subscriptrelax}{\relaxparam=0}\eqref{eq: relaxed system param pressure}} yield
\begin{align}
	\dt(\rho^P e^P) + 2\Pi^P \dx u^P
	&=0
	,&
	2\dx u^P
	&=
	\dt(\rho^P \Pi^P / a^2)
	\label{eq: acoustic e eq}
	.
\end{align}
We thus obtain the stationary equations
\begin{align}
	\dt\qty
	[e^P - \frac{(\Pi^P)^2}{2a^2}]
	&=0
	,&
	\dt\qty[\Tau^P + \frac{\Pi^P}{a^2}] &= 0
	.
	\label{eq: acoustic stationnary eq}
\end{align}
So now the acoustic subsystem~{\renewcommand{\subscriptrelax}{\relaxparam=0}\eqref{eq: relaxed system param pressure}} takes the simple form
%
%
\begin{subequations}
	\begin{align}
		\dt \phi & \!=\!0
		,&
		\dt \rho^P &\!\!=\!0
		,&
		\dt\qty
		[e^P \!-\! \frac{(\Pi^P)^2}{2a^2}]
		&=0
		,\\
		\dt (\rho^P u^P) + 2\dx\Pi^P +2\rho^P\dx \phi^P&\!\!=\!0
		,&
		\dt (\rho^P\Pi^P) + 2 a^2 \dx u^P &\!\!=\!0
		,&
		\dt\qty[\Tau^P + \frac{\Pi^P}{a^2}] &= 0
		.
		\label{eq: acoustic part v1}
	\end{align}
\end{subequations}
We now turn to the advection part of (\ref{eq: relaxed system param}\textsubscript{\(\relaxparam\!=\!0\)}): the subsystem~{\renewcommand{\subscriptrelax}{\relaxparam=0}\eqref{eq: relaxed system param advection}} takes the simple form
%
\begin{align} 
	\dt \rho^A + \dx (2\rho^A u^P)&=0,
	&
	\dt\qty[
	\rho^A\Tau^A-\frac{\rho^P\Pi^P}{a^2}
	]&=0
	,&
	\dt b^A + 2 u^P\dx b^A&=0
	,\quad
	b^A\in\{u^A, E^A, \Pi^A \}
	.
	\label{eq: advection part v1}
\end{align}
Therefore if we set
	\begin{equation}
		\bW^T = \qty [u^P , \Pi^P , \rho^P, \phi ,
		e^P-\frac{(\Pi^P)^2}{2 a^2},
		\Tau^P+\frac{\Pi^P}{a^2},
		\rho^A\Tau^A-\frac{\rho^P\Pi^P}{a^2},
		u^A , \Pi^A , E^A, \rho^A],
	\end{equation}
we can see that (\ref{eq: relaxed system param}\textsubscript{\(\relaxparam\!=\!0\)}) can be recast
into the following quasilinear system
	\begin{align}
		\dt\bW + \diagomatrix(\bW)\dx\bW
		&=0,&
		\diagomatrix(\bW)
		&=
		\qty[
		\begin{array}{ccccccccccc}
			0                     & \dfrac{2}{\rho^P} & 0 & 2 & 0 & 0 & 0 & 0 & 0 & 0 & 0
			\\
			\dfrac{2 a^2}{\rho^P} &      0            & 0 & 0 & 0 & 0 & 0 & 0 & 0 & 0 & 0
			\\
			0                     &      0            & 0 & 0 & 0 & 0 & 0 & 0 & 0 & 0 & 0
			\\
			0                     &      0            & 0 & 0 & 0 & 0 & 0 & 0 & 0 & 0 & 0
			\\
			0                     &      0            & 0 & 0 & 0 & 0 & 0 & 0 & 0 & 0 & 0
			\\
			0                     &      0            & 0 & 0 & 0 & 0 & 0 & 0 & 0 & 0 & 0
			\\
			0                     &      0            & 0 & 0 & 0 & 0 & 0 & 0 & 0 & 0 & 0
			\\
			0                     &      0            & 0 & 0 & 0 & 0 & 0 & 2 u^P & 0 & 0 & 0
			\\
			0                     &      0            & 0 & 0 & 0 & 0 & 0 & 0 & 2 u^P & 0 & 0
			\\
			0 & 0 & 0 & 0 & 0 & 0 & 0 &  0 & 0 & 2 u^P & 0
			\\
			2\rho^A & 0 & 0 & 0 & 0 &  0 & 0 & 0 & 0 & 0 & 2 u^P
		\end{array}
		]
		.
		\label{eq: overall system non-conservative}
	\end{align}
It is then straightforward to see that the eigenvalues of $\diagomatrix(\bW)$ are
$2 u^P$ (with an algebraic multiplicity 4), $0$ (with an algebraic multiplicity 5) and
$\pm 2 a/\rho^P$.

The eigenvectors
$\qty(\br_0^{(k)})_{k=1,\ldots,3}$,
$\qty(\br_{u^P}^{(k)})_{k=1,\ldots,4}$
and $\br_{\pm}$ that are respectively associated with $0$, $2 u^P$ and $\pm 2 a/\rho^P$ are
\begin{subequations}
	\begin{align}
		\br_0^{(1)}
		&=
		[0,0,1,0,0,0,0,0,0,0,0]^T
		,&
		\br_0^{(2)}
		&=
		\qty[0,-\rho^P,0,1,0,0,0,0,0,0,0]^T
		,\\
		\br_0^{(3)}
		&=
		[0,0,0,0,1,0,0,0,0,0,0]^T
		,&
		\br_0^{(4)}
		&=
		[0,0,0,0,0,1,0,0,0,0,0]^T
		\\
		\br_0^{(5)}
		&=
		[0,0,0,0,0,0,1,0,0,0,0]^T
		,&
		\\
		\br_{u^P}^{(1)}
		&=
		[0,0,0,0,0,0,0,1,0,0,0]^T
		,&
		\br_{u^P}^{(2)}
		&=
		[0,0,0,0,0,0,0,0,1,0,0]^T
		,\\
		\br_{u^P}^{(3)}
		&=
		[0,0,0,0,0,0,0,0,0,1,0]^T
		,&
		\br_{u^P}^{(4)}
		&=
		[0,0,0,0,0,0,0,0,0,0,1]^T
		,\\
		\br_{+}
		&=
		\qty
		[1,a,0,0,0,0,0,0,0,0,-\frac{\rho^A\rho^P}{\rho^P u^P-a}]^T
		,&
		\br_{-}
		&=
		\qty
		[1,-a,0,0,0,0,0,0,0,0,-\frac{\rho^A\rho^P}{\rho^P u^P+a}]^T
		,
	\end{align}
	\label{eq: eigenvectors}
\end{subequations}
so that \eqref{eq: overall system non-conservative} is hyperbolic and only involves linearly degenerate fields.

\end{document}